\documentclass[a4paper, 10pt, twoside, notitlepage]{amsart}

\usepackage{amsmath,amscd}
\usepackage{amssymb}
\usepackage{amsthm}
\usepackage{comment}
\usepackage{graphicx, xcolor}
\usepackage{xcolor}
\usepackage{mathrsfs}
\usepackage[ocgcolorlinks, linkcolor=blue]{hyperref}

\usepackage{bm}
\usepackage{bbm}
\usepackage{url}

\usepackage[utf8]{inputenc}
\usepackage{mathtools,amssymb}
\usepackage{esint}
\usepackage{tikz}
\usepackage{dsfont}
\usepackage{relsize}
\usepackage{url}
\usepackage{xcolor}
\usepackage{graphicx}
\usepackage{mathrsfs}
\usepackage[shortlabels]{enumitem}
\usepackage{lineno}
\usepackage{amsmath}
\usepackage{enumitem}
\usepackage{amsthm} 
\usepackage{verbatim}
\usepackage{dsfont}
\usepackage{tikz}
\numberwithin{equation}{section}

\allowdisplaybreaks

\mathtoolsset{showonlyrefs}

\graphicspath{{images/}}

\newtheorem{theorem}{Theorem}[section]
\newtheorem{lemma}[theorem]{Lemma}
\newtheorem{definition}[theorem]{Definition}
\newtheorem{corollary}[theorem]{Corollary}

\newtheorem{claim}[theorem]{Claim}

\title[Stability for a relativistic wave equation in a waveguide]{
H\"older stability estimates for the determination of time-independent potentials in a relativistic wave equation in an infinite waveguide}

\author[M. Kumar]{Mandeep Kumar}
\address{Department of Mathematics, Indian Institute of Technology Ropar, Rupnagar, Punjab-140001, India}
\email{mandeep.sansanwal@gmail.com}

\author[P. Zimmermann]{Philipp Zimmermann}
\address{Departament de Matem\`atiques i Inform\`atica, Universitat de Barcelona, Barcelona, Spain}
\email{philipp.zimmermann@ub.edu}

\newcommand{\C}{{\mathbb C}}
\newcommand{\R}{{\mathbb R}}
\newcommand{\Z}{{\mathbb Z}}
\newcommand{\N}{{\mathbb N}}

\newcommand{\eps}{\epsilon}

\newcommand{\tempered}{\mathscr{S}^{\prime}}

\newcommand{\distr}{\mathscr{D}^{\prime}}
\newcommand{\norm}[1]{\lVert #1 \rVert}
\newcommand{\abs}[1]{\left\lvert #1 \right\rvert}
\DeclareMathOperator{\Div}{div} 
\DeclareMathOperator{\supp}{supp} 
\DeclareMathOperator{\dist}{dist} 
\newcommand{\im}{\mathsf{i}} 
\newcommand{\re}{\text{ Re}} 

\newcommand{\weak}{\rightharpoonup}
\newcommand{\weakstar}{\overset{\ast}{\rightharpoonup}}

\begin{document}

	\maketitle
	\begin{abstract}
    The main goal of this article is to establish H\"older stability estimates for the Calderón problem related to a relativistic wave equation. The principal novelty of this article is that the partial differential equation (PDE) under consideration depends on three unknown potentials, namely a temporal dissipative potential $A_0$, a spatial vector potential $A$ and an external potential $\Phi$. Moreover, the PDE is posed in an infinite waveguide geometry $\Omega=\omega\times\R$ and not on a bounded domain. For our proof it is essential that the potentials are time-independent as a key tool in this work are pointwise estimates for the Radon transform of the vector potential $\mathcal{A}=(A_0,\im A)$ and external potential $\Phi$. Furthermore, the demonstrated stability estimates hold for a wide range of $H^s$ Sobolev scales and a main contribution is to explicitly determine the dependence of the involved constants and the H\"older exponent on the Sobolev exponents of the potentials $A_0,A$ and $\Phi$.
		
		\medskip
		
		
	\end{abstract}

	\tableofcontents

    \section{Introduction}

    One of the most famous equations in relativistic quantum mechanics is beyond all doubt the \emph{Klein--Gordon equation}
    \begin{equation}
    \label{eq: Klein Gordon eq}
        \left(\partial_t^2 -\Delta+m^2\right)\Psi(x,t)=0\text{ in }\R^3\times \R,
    \end{equation}
    where $\Psi$ is a complex-valued wave function, $m>0$ is the mass of the particle and units are chosen such that $c=\hbar=e =1$. This equation is used to describe the dynamics of relativistic, spin 0 particles, like $\pi$ or $K$ mesons, and can be derived from the \emph{relativistic energy--momentum relation} 
    \begin{equation}
    \label{eq: relativistic energy momentum}
        p_\mu p^{\mu}=E^2-p^2=m^2
    \end{equation}
    together with the \emph{quantization rule}
    \begin{equation}
    \label{eq: quantization rule}
        p^\mu\,\rightarrow\,\im\partial^\mu\,\text{ i.e. }\,\begin{cases}
            E\,\rightarrow \, \im\partial_t, & \\
            p\,\rightarrow\,-\im \nabla.&
        \end{cases}
    \end{equation}
    Here, $E$ denotes the energy, $p$ the usual three momentum, $p^\mu=(E,p)$ the four momentum, $\partial^\mu=(\partial_t,-\nabla)$ and the index $\mu$ runs from 0 to 3. If we consider the dynamics of a charged particle with charge $q=-1$, then we are forced to replace the momentum $p^\mu$ in \eqref{eq: quantization rule} by 
    \begin{equation}
    \label{eq: gauge invariant}
        \pi^\mu\vcentcolon = p^\mu+ A^\mu,
    \end{equation}
    where $A^\mu=(A_0,A)$ is the electromagnetic four potential. Here, $A_0$ is the scalar potential and $A$ is the vector potential, which are related to the electric field $\mathbf{E}$ and magnetic field $\mathbf{B}$ by the formulas
    \begin{equation}
    \label{eq: electric magnetic potential}
        \mathbf{E}=-\nabla A_0-\partial_t A\text{ and }\mathbf{B}=\nabla \wedge A.
    \end{equation}
    Using \eqref{eq: gauge invariant} in the energy momentum identity \eqref{eq: relativistic energy momentum} and applying the quantization rule \eqref{eq: quantization rule}, one obtains the following \emph{modified Klein-Gordon equation}
    \begin{equation}
    \label{eq: modified Klein Gordon}
        \left[(\im\partial_t +A_0)^2-(\im\nabla -A)^2-m^2\right]\Psi(x,t)=0\text{ in }\R^3\times \R,
    \end{equation}
    which generalizes \eqref{eq: Klein Gordon eq} to the case of non-vanishing electromagnetic fields. 
    
    If we restrict the dynamics to an infinite waveguide $\Omega=\omega\times\R$ with $\omega$ being a smoothly bounded domain in $\R^2$ and assume the potential $A_0$ in \eqref{eq: modified Klein Gordon} is purely imaginary ($A_0\to \im A_0$), the third components of $A$ vanishes and replace the mass term $m^2$ by a generic external potential $\Phi$, then we end up with the following \emph{relativistic wave equation}\footnote{Throughout this work the notation $C_T$ for $C\subset\R^n$ and $T>0$ always refers to the space-time cylinder $C_T=C\times (0,T)$.}
    \begin{equation}
    \label{eq: relativistic schroedinger}
        \big[(\partial_{t}+ A_{0})^{2} -  (\nabla_x+ \im A)^{2}  -\partial^{2}_{y} + \Phi\big]u = 0  \text{ in } \Omega_T = \Omega \times (0,T),
    \end{equation}
    where we redefined $A\vcentcolon =(A_1,A_2)$, $T>0$ is a finite time horizon and typical points in $\Omega$ are denoted by $X:=(x,y)$ for $x\in\omega\subset \R^2$ and $y\in\R$. Moreover, we assume in the whole article that all potentials $A_0,A$ and $\Phi$ are time-independent. A motivation for equation \eqref{eq: relativistic schroedinger} can be found in the recent article \cite{reggia2023generalizing} in which the Maxwell equations for complex-valued electric and magnetic fields are investigated and aims to explain why magnetic monopoles could not experimentally detected although many theoretical considerations support their existence. 
    
    The main goal of this article is to show that the vector potential $\mathcal{A}=(A_0,\im A)$ and the potential $\Phi$ can be recovered in a H\"older stable way from the related \emph{Dirichlet to Neumann (DN) map} $\Lambda_{\mathcal{A},\Phi}$. 
    
   \subsection{Mathematical setup and main results}

      Let us next give a more concrete description of the problem considered in this article. Our starting point is the observation that for all potentials $(\mathcal{A}, \Phi)$ satisfying suitable regularity and decay conditions, the initial boundary value problem (IBVP) 
    \begin{equation}\label{maineqn}
    \begin{cases}
     \big[(\partial_{t}+ A_{0})^{2} -  (\nabla_x+ \im A)^{2}  -\partial^{2}_{y} + \Phi\big]u = 0 & \text{in } \Omega_T,\\
        u = f & \text{on } (\partial\Omega)_T,\\
          u(0) = 0,\, \partial_{t}u(0) = 0 & \text{in } \Omega,
    \end{cases}
\end{equation}
    has for all $f\in D^{3/2}((\partial\Omega)_T)$ a unique solution $u\colon \Omega_T\to \R$ satisfying
    \[
        u\in C([0,T];H^{1}(\Omega)) \cap C^{1}([0,T];L^{2}(\Omega))\text{ with }\partial_{\nu}u\in L^{2}((\partial\Omega)_T)
    \]
    and
    \begin{equation}
\label{eq: continuity IBVP}
    \|u\|_{L^{\infty}(0,T;H^{1}(\Omega))} + \|\partial_t u\|_{L^{\infty}(0,T;L^{2}(\Omega))} + \|\partial_{\nu}u\|_{L^{2}((\partial\Omega)_T)} \leq C\|f\|_{D^{3/2}((\partial\Omega)_T)}
\end{equation}
(see Section \ref{sec: Existence}). Here, the space of boundary values is defined by
   \begin{equation}
   \label{eq: data space}
   \begin{split}
       &D^{3/2}((\partial\Omega)_T) \\
       &\quad =\left\{f\in H^{\frac{3}{2},\frac{3}{2}}((\partial\Omega)_T): \left.f\right|_{t=0}= 0,\right.\left. \partial_tf,\partial_\tau f,\partial_y f\in L^2\left((\partial\Omega)_T; t^{-1}dSdt\right)\right\},
    \end{split}
   \end{equation}
   where $\partial_\tau f$ is a tangential derivative of $f$ with respect to $\partial \omega$ and we endow it with the norm
   \[
    \|f\|^2_{D^{3/2}((\partial\Omega)_T)}\vcentcolon =\|f\|^2_{H^{3/2,3/2}((\partial\Omega)_T)}+\sum_{\alpha=t,\tau,y}\|\partial_\alpha f\|^2_{L^2\left((\partial\Omega)_T; t^{-1}dSdt\right)}.
   \]
   For the precise definition of the space-time Sobolev spaces $H^{\frac{3}{2},\frac{3}{2}}((\partial\Omega)_T)$ we refer the reader to Section~\ref{preliminaries}. By the previous unique solvability assertion, we can introduce the DN map by
\begin{equation}
\label{eq: DN map}
\Lambda_{\mathcal{A},{\Phi}} \colon  D^{3/2}((\partial\Omega)_T) \to  L^{2}((\partial\Omega)_T)\text{ with }
\Lambda_{\mathcal{A},{\Phi}}f \vcentcolon =  \partial_{\nu}u|_{(\partial\Omega)_T},
\end{equation}
where $u$ is the unique solution of \eqref{maineqn}, and by \eqref{eq: continuity IBVP} it is a bounded operator. Then one may ask the following natural questions:

\begin{enumerate}[(i)]
		\item\label{Q:Uniq} \textbf{Uniqueness.}  Can one determine the vector potential $\mathcal{A}$ and external potential $\Phi$ in $\Omega$ from the DN map $\Lambda_{\mathcal{A},\Phi}$?
        \item\label{Q: Stab} \textbf{Stability.} Does there exist a modulus of continuity $\lambda\colon [0,\infty]\to [0,\infty]$ such that
        \begin{equation}
        \label{eq: stability est intro}
            \|(\mathcal{A}^1,\Phi_1)-(\mathcal{A}^2,\Phi_2)\|_1\lesssim \lambda(\|\Lambda_{\mathcal{A}^1,\Phi_1}-\Lambda_{\mathcal{A}^2,\Phi_2}\|_2)
        \end{equation}
        for all admissible pairs $(\mathcal{A}^j,\Phi_j)$ and certain norms $\|\cdot\|_k$ for $k=1,2$?
	\end{enumerate}

Before presenting the main result of this article, let us define the admissible class of vector and external potentials considered in this article.
\begin{definition}
\label{def: admissible class}
    For all $R_1,R_2>0$, $s_0\geq 2,s_1\geq 2$ and $s_2\geq 0$, we denote by
    \begin{enumerate}[(i)]
        \item\label{item: base case} $\mathscr{A}(R_1)$ the set of all pairs $(\mathcal{A},\Phi)\in C_0^2(\Omega)\times C(\overline{\Omega})$\footnote{Here, and throughout this article, we follow the convention to not explicitly depict the dependence on the finite dimensional target space in the notation for function spaces. For example, we write $L^{\infty}(\Omega)$ instead of $L^{\infty}(\Omega;\R^2)$.}, $\mathcal{A}=(A_0,\im A)$, satisfying
        \begin{equation}
            \|\mathcal{A}\|_{W^{2,\infty}(\Omega)}+\|\Phi\|_{L^{\infty}(\Omega)}\leq R_1
        \end{equation}
        \item\label{item: higher regular case} and $\mathscr{A}_{s_0,s_1,s_2}(R_1,R_2)$ stands for the set of all pairs $(\mathcal{A},\Phi)\in \mathscr{A}(R_1)$ having the following two properties:
        \begin{enumerate}[(a)]
            \item\label{regularity assump} $(\mathcal{A},\Phi)\in L^{\infty}(\R;H^{s_0}(\R^2)\times H^{s_1}(\R^2))\times L^{\infty}(\R; H^{s_2}(\R^2))$
            \item\label{uniform bound assump} $\|\mathcal{A}\|_{L^{\infty}(\R;H^{s_0}(\R^2)\times H^{s_1}(\R^2))} +\|\Phi\|_{L^{\infty}(\R; H^{s_2}(\R^2))} \leq R_2$.
        \end{enumerate}
    \end{enumerate}
\end{definition}

Now, we state the main result of this article.
\begin{theorem}
\label{main theorem}
    Let $\Omega=\omega\times\R$ be an infinite waveguide with $\omega\subset\R^2$ being a smoothly bounded domain and $T>\text{diam}(\omega)$. Suppose that $\alpha\in(0,\min\{1, (T-\text{diam}(\omega))/3\})$ and let $R_1,R_2>0$, $s_0,s_1\geq 2$ and $s_2\geq 0$.
     For any $\bar{s}_0,\bar{s}_1,\bar{s}_2\in\R$ satisfying
     \begin{equation}
     \label{eq: def coeff stability lemma pot}
         \mathfrak{b}_0\vcentcolon =s_0-\bar{s}_0,\,\mathfrak{b}_1\vcentcolon = s_1-\bar{s}_1,\,\mathfrak{b}_2\vcentcolon = s_2-\bar{s}_2\in (0,1),
     \end{equation}
    there exist $C=C(\omega,T,R_1,R_2) > 0$ and $\eta \in (0,1)$ such that
    \begin{equation}
    \label{eq: final stability estimate}
    \begin{split}
        &\| A^{1}_{0} -A^{2}_{0}\|^2_{L^{\infty}(\R_{y}; {H^{\bar{s}_0}(\R^2)})}+\| \nabla\wedge (A^1-A^2)\|_{L^{\infty}(\R_{y}; {H^{\bar{s}_1-1}(\R^2)})}^2 \\
        &+\norm{\Phi_{1}-\Phi_{2}}_{L^{\infty}(\R_{y}; {H^{\bar{s}_2}(\R^2)})}^2\\
        &\leq C \frac{4^{3\mathfrak{A}}}{(\mathfrak{C}\mathfrak{D})^2}\left(\frac{1+\mathfrak{A}}{\mathfrak{C}}\right)^{\frac{\gamma+\beta}{\gamma}+\frac{2\mathfrak{c}}{1+\mathfrak{a}+\mathfrak{c}}+\frac{\mathfrak{b}_2}{1+\bar{s}_2+\mathfrak{b}_2}}\norm{\Lambda_{\mathcal{A}^1,\Phi_1}-\Lambda_{\mathcal{A}^2,\Phi_2}}_{D^{3/2}((\partial\Omega)_T)\to L^2((\partial\Omega)_T)}^{\eta}
    \end{split}
    \end{equation}
    for all $(\mathcal{A}^{j},\Phi_{j})\in \mathscr{A}_{s_0,s_1,s_2}(R_{1},R_{2})$, $j=1,2$, satisfying the conditions
    \begin{gather}
    \Div_x A^1=\Div_x A^2,\,\|\mathcal{A}^1-\mathcal{A}^2\|_{L^{\infty}(\Omega)}<\frac{2\pi}{T},\\
    \norm{\Lambda_{\mathcal{A}^1,\Phi_1}-\Lambda_{\mathcal{A}^2,\Phi_2}}_{D^{3/2}((\partial\Omega)_T)\to L^2((\partial\Omega)_T)}\leq 1.
    \end{gather}
     Here, the coefficients $\mathfrak{a},\mathfrak{A},\mathfrak{c},\mathfrak{C},\mathfrak{D}$ and the H\"older exponent $\eta\in (0,1)$ are given by 
     \begin{gather}
      \mathfrak{a}\vcentcolon =\max(s_0,s_1),\,\mathfrak{A}\vcentcolon =\max(s_0,s_1,s_2),\\
         \mathfrak{c}\vcentcolon = \min(\mathfrak{b}_0,\mathfrak{b}_1), \, \mathfrak{C}\vcentcolon = \min(\mathfrak{b}_0,\mathfrak{b}_1,\mathfrak{b}_2),\\
          \mathfrak{D}\vcentcolon = \Theta(\mathfrak{b}_0)\Theta(\mathfrak{b}_1)\Theta(\mathfrak{b}_2)
     \end{gather}
     and
    \begin{equation}
        \eta\vcentcolon =\frac{\mathfrak{C}^4\gamma^2}{(1+\mathfrak{A}+\mathfrak{C})^2(\gamma+\beta)^2}\in (0,1),
    \end{equation}
    where $\Theta(\rho)=1-\rho$ and $(\gamma,\beta)$ is any pair of positive numbers satisfying $0<\gamma<1/11$ and $\beta\geq 3+10\gamma$.
\end{theorem}

The proof of this result is presented in Section \ref{subsec: proof main result}, which relies on a separate stability estimate for the vector potential $\mathcal{A}$ and external potential $\Phi$. These estimates can be found below in Lemma \ref{lemma: integral estimate of A} and \ref{lemma: integral estimate of Phi}, which do not require the restriction 
\[
    \norm{\Lambda_{\mathcal{A}^1,\Phi_1}-\Lambda_{\mathcal{A}^2,\Phi_2}}_{D^{3/2}((\partial\Omega)_T)\to L^2((\partial\Omega)_T)}\leq 1
\]
and thus for sharper constants and exponents, we refer the interested reader to these results. These proofs in turn depend on the construction of geometric optics solutions in Section \ref{sec: geometric optics sol} and for these it is essential that the PDE has the form \eqref{eq: relativistic schroedinger}. In fact, a striking property of \eqref{eq: relativistic schroedinger} is that the dispersion relation $\omega(k)$ in the 1-dimensional constant coefficient case without the terms $\partial_y^2$ and $\Phi$ is given by
\[
    \omega(k)=\im A_0\pm |k+A|
\]
and hence, it is complex-valued for $A_0\neq 0$. If $A_0$ were imaginary, then it would lead to a real-valued dispersion relation and thus only has an oscillatory part. In contrast, in our situation it admits solutions of the form
\[
    e^{\im kx\pm\im |k+A|t}e^{-A_0 t},
\]
which have, depending on the sign of $A_0$, an exponentially decaying or increasing contribution but also an oscillatory factor. Indeed, it is a common feature of non time-reversible PDEs, like the heat equation, to have complex-valued dispersion relations and the term $A_0$ exactly breaks down this symmetry. Finally, let us observe that another difference between the relativistic wave equation \eqref{eq: relativistic schroedinger} and the Klein--Gordon equation \eqref{eq: Klein Gordon eq} or its generalization  \eqref{eq: modified Klein Gordon} is that the first one naturally incorporates a real-valued damping term $2A_0\partial_t$ and hence it describes a physical system exhibiting dissipation.

Next, let us make a few comments on the imposed regularity assumptions on the vector potential $\mathcal{A}$. On the one hand the benefit of using $C^2$ regular coefficients is that the well-posedness theory is slightly simpler than in the low regularity setting, but more importantly it allows us later to deduce that the remainder terms in the geometric optics solutions are of class $H^{2,2}(\Omega_T)$ (see Lemma \ref{intid2 2}). On the other hand, the Sobolev regularity for $\mathcal{A}$ and $\Phi$ is solely needed to deduce the stability estimates from the pointwise bounds on the Radon transforms of the potentials (see Lemma \ref{lemma: stability estimate for vector potential} and \ref{lemma: stability estimate for potential}).

Let us close this section by mentioning that on the way in proving our main result, we also show that in fact the $L^{\infty}$ norm of the vector potentials $\mathcal{A}$ can be controlled by the difference of the DN maps in a H\"older stable way (see Corollary \ref{cor: L infinity stability}).

\subsection{Comparison to the existing literature}

Next, let us review some related results in the existing literature. The question \ref{Q:Uniq} has been studied extensively in the literature, and one can outline at least three general methods for establishing this unique determination result. The first approach, stemming from the seminal works \cite{Bel87, BK92}, relies on the so-called boundary control (BC) method together with Tataru's sharp, unique continuation theorem \cite{Tat95}. This method yields recovery of time-independent potentials under very weak assumptions on the transversally anisotropoic manifold $(M, g)$ on which the partial differential equation (PDE) is considered. An alternative approach to derive uniqueness results in the time-independent case started from the seminal work \cite{BK81}, where Carleman estimates were used for the first time in the context of inverse problems. Inspired from the work of Sylvester and Uhlmann \cite{MR873380} approaches based on geometric optics solutions  have also been rather fruitful in deriving uniqueness  results for wave type operators, in particular for the wave equation with time independent potential has been studied by Rakesh and Symes \cite{MR914815} from the measurement of  Neumann to Dirichlet map and it is extended by Isakov \cite{MR1116858} to wave operators with first order perturbation in time derivative and potential for time independent coefficients. In the manifold setup, uniqueness for time-independent coefficients in hyperbolic equations has been studied by Eskin \cite{MR2235639,MR2441006}.

Next, let us discuss some works related to the the stability question \ref{Q: Stab}. The use of geometric optics solutions to derive H\"older stability estimates for wave type operators at least goes back to Sun. By using this technique, Sun has shown in \cite{MR1059582} that time independent coefficients in wave equations posed on a bounded domain can be recoverd in a H\"older stable way from the full Neumann to Dirichlet map. In \cite{MR1158175}, Isakov and Sun extended this stability result to wave operators with linear perturbations in $\partial_t u$ and $u$ (i.e. $qu+Q\partial_t u$) in two and three dimensions by using only local Neumann to Dirichlet data. In the 3 dimensional case they achieved H\"older type stability estimates whereas in 2 dimensions they ended up with a logarithmic modulus of continuity under the assumption $Q=0$. Cipolatti and Lopez considered in \cite{MR2132903} wave operators with first order perturbation in time derivative but without a potential term. By using the full DN map, they derived H\"older and Lipschitz stability estimates depending on the regularity of the coefficient. Furthermore, in the same work they investigated the finite measurement problem. Bellassoued and Benjoud studied in \cite{MR2401822} the wave equation with a magnetic potential using the full DN map and obtained H\"older type stability in the magnetic potential. Ben A\"icha and Bellassoued \cite{MR3706187} extend the previous result to less regular magnetic potential. In partial data setting, Bellassoued et al. \cite{MR2523687} have obtained logarithmic type stability estimates for wave equations with a potential. In manifold setup, the stability of the wave type operator has been studied in several works including \cite{MR1612709,MR2852371,MR4366889,MR3995367}.

Next, let us remark that for the case of time-dependent potentials, which is not considered in this work, there is an obstruction to uniqueness from the DN map in the whole domain. To the best of our knowledge, it seems that in the existing literature  different approaches have been applied to obtain an affirmative answer to the question \ref{Q:Uniq} for time-dependent potentials. All the papers below study wave equations with a time-dependent potentials and therefore let us highlight their main differences. In \cite{RS91, Sal13}, the problem is set on an infinite time horizon, Isakov \cite{Isa91} recovered the potential from an input to output map, the authors of \cite{RR91} recover uniquely the time-dependent coefficient in a suitable subset of the domain from the DN map. We also refer the interested reader to the studies \cite{Kia16a,MR3661867,MR4124641,MR4343270} in which also inverse problems for wave equations with time-dependent potentials are investigated. Stability for the time dependent case in a bounded domain has been studied, for example, in \cite{MR3540317,MR3595191,MR4013301,MR4191617}.


Most of the above articles consider inverse problems in a bounded domain $\Omega$, whereas results for unbounded domains like infinite waveguides are not that extensively studied. A prototypical example of such an inverse problem for an elliptic PDE is the one associated to the \emph{time-independent Schr\"odinger equation}
\begin{equation}
\label{eq: schroed eq}
    (-\Delta+\Phi)u=0\text{ in }\Omega=\omega\times \R.
\end{equation}
In the literature different classes of potentials $\Phi$ have been studied. In the works \cite{CKS17,CKS18} the authors studied the uniqueness and stability question for (real-valued) potentials $\Phi\in L^{\infty}(\Omega)$, which are 1-periodic in the $y$ direction. In \cite{Kia20b}, the case of non-compactly supported potentials but with an $L^1-$integrability condition on the difference $\Phi_1-\Phi_2$ has been investigated and a unique determination result has been established. A stability result related to this model has been considered in \cite{Sou21} under suitable boundedness assumptions and the condition that the potentials a priori coincide on $\partial\Omega$. In \cite{Kia20a} this has been generalized to a simultaneous determination result, where the Laplacian in \eqref{eq: schroed eq} is replaced by the magnetic Laplacian and $\Phi$ is allowed to be complex-valued. This required an $L^2-$integrability condition on the difference of the potentials and an $L^1-$integrability condition on the difference of the magnetic potentials. More precisely, it is shown that in this case one has $dA^1=dA^2$ and $\Phi_1=\Phi_2$ in $\Omega$, where $dA^j$ denotes the exterior derivative of $A^j$. Stability estimates for $dA$ and $\Phi$ has then be obtained in \cite{KS21} under various a priori information.

Some of these results have been generalized to the \emph{time-dependent Schr\"odinger equation}
\begin{equation}
\label{eq: time-dep schroed eq}
    (i\partial_t+\Delta+\Phi)u=0\text{ in }\Omega=\omega\times \R.
\end{equation}
For example, in \cite{CKS15} the authors considered the case of periodic potentials. Another type of inverse problem related to \eqref{eq: time-dep schroed eq} is studied in the articles \cite{KPS14,KPS15}, where the authors assume that $\Phi$ in \eqref{eq: time-dep schroed eq} is known outside a compact set of $\Omega$ and they want to determine $\Phi$ in the remaining part. They also establish a Lipschitz stability estimate. In \cite{BKS18} a magnetic version of \eqref{eq: time-dep schroed eq} is studied on an infinite waveguide in $\R^3$ and they establish unique determination results for $(dA,\Phi)$ modulo the natural gauge as well as a H\"older stability estimate for the vector potential.
For a similar type of inverse problem on an infinite waveguide, we refer the interested reader to \cite{BKS16}, where a time-independent potential is recovered from the initial to Neumann data map instead of the DN map.

Finally, let us make a remark on inverse problems for wave equations on infinite waveguides. In the article \cite{Kia14} the stability of the inverse problem related to \eqref{eq: relativistic schroedinger} with $\mathcal{A}=0$, but otherwise the same setting, has been studied. The author established a local H\"older stability estimate for $\alpha$ H\"older continuous time-independent potentials $\Phi$. More precisely, it is shown that if $\|\Phi_j\|_{C^{0,\alpha}(\overline{\Omega})}\leq R$ for some $R>0$, $0<\alpha<1$ and $T>\text{diam}(\omega)$, then there exists $C=C(R,T,\Omega)>0$ such that
   \begin{equation}
   \label{eq: Holder estimate Kian}
       \|\Phi_1-\Phi_2\|_{L^{\infty}(\Omega)}\leq C\|\Lambda_{1}-\Lambda_2\|_{D^{3/2}((\partial\Omega)_T)\to L^2((\partial\Omega)_T)}^{\beta},
   \end{equation}
   where $\Lambda_j=\Lambda_{0,\Phi_j}$ for $j=1,2$, $\|\cdot\|_{D^{3/2}((\partial\Omega)_T)\to L^2((\partial\Omega)_T)}$ denotes the operator norm and
   \begin{equation}
       \beta=\frac{\min(2\alpha,1)\alpha}{3(2\alpha+2)(\min(4\alpha,2)+21)}.
   \end{equation}
Beside the article \cite{Kia14}, dealing with \eqref{eq: relativistic schroedinger} in the case $\mathcal{A}=0$, there are at least two other articles studying an inverse problem for a wave equation on an infinite waveguide, namely \cite{cristofol2015determining,MR4835772}. In \cite{cristofol2015determining}, the authors recover the uniformly elliptic coefficient $\gamma\colon \Omega\to\R$ in the wave equation 
\begin{equation}
\label{eq: wave eq on unbounded dom}
    \begin{cases}
        \partial_t^2u-\Div(\gamma\nabla u)=0&\text{ in }\Omega_T,\\
        u=0&\text{ on }(\partial\Omega)_T,\\
        u(0)=u_0,\,\partial_t u(0)=u_1&\text{ in }\Omega,
    \end{cases}
\end{equation}
where $\Omega=\omega\times \R$ is an infinite waveguide. More concretely, they wish to recover $\gamma$ in $\Omega_{\ell} = \omega \times (-\ell,\ell)$ ($\ell>0$ fixed) in a stable way and not on the whole waveguide. Using Carleman estimates for waveguides and linearization techniques, the authors establish H\"older stability estimates. The authors of \cite{kumar2024stable} studied the stable inversion of time-dependent matrix potentials for the wave equation in an infinite waveguide from the input to output map. Their strategy also relies on the use of geometrics optics solutions and a suitable integral identity leading to the light ray transform. Combining this with Vesella's analytic continuation arguments they demonstrated logarithmic type stability estimates. 

\subsection{Organization of the article} The rest of this article is organized as follows. After collecting in Section \ref{preliminaries} some background material and introducing our main notation, we establish in Section \ref{sec: Existence} the unique solvability of the relativistic wave equation \eqref{maineqn}. Afterwards, in Section \ref{sec: geometric optics sol} we prove the existence of geometric optics solutions for \eqref{maineqn} and its formal adjoint. With this preparation at our disposal, we can tackle in Section \ref{sec: Stability estimate} the inverse problem. First, we achieve in Section \ref{subsec: pointwise estimate radon A21} a pointwise estimate for the Radon transform of the vector potential $\mathcal{A}^1-\mathcal{A}^2$. Using this estimate, we prove in Section \ref{subsec: stability for A}  and \ref{subsec: stability of phi} the stability estimates for the vector potential $\mathcal{A}$ and external potential $\Phi$. Finally, in Section \ref{subsec: proof main result} we present the proof of Theorem \ref{main theorem}.

\section{Preliminaries}
\label{preliminaries}

In this section we recall some background material and take this occasion to introduce the main notation used in this work.

\subsection{Function spaces}

Throughout this article, for any open set $\Omega\subset\R^n$, $1\leq p\leq \infty$ and $k\in\N_0\cup\{\infty\}$, we denote by $L^p(\Omega)$ and $W^{k,p}(\Omega)$ the usual (complex) Lebesgue and Sobolev spaces, which carry the norms
\begin{equation}
\label{eq: Lebesgue and Sobolev norms}
    \|u\|_{L^p(\Omega)}=\left(\int_\Omega |u|^p\,dx\right)^{1/p}\,\text{ and }\,\|u\|_{W^{k,p}(\Omega)}=\left(\sum_{|\alpha|\leq k}\|\partial^{\alpha}u\|_{L^p(\Omega)}^p\right)^{1/p}
\end{equation}
(with the usual modifications in the case $p=\infty$). 
Here, $|z|^2=\bar{z}z$ and we write $\bar{z}$ for the complex conjugate of $z\in\C$. In particular, $H^k(\Omega)\vcentcolon = W^{k,2}(\Omega)$, $k\in\N_0$, is a Hilbert space and its norm is induced by the inner product
\begin{equation}
\label{eq: L2 inner product}
    \langle u,v\rangle_{H^k(\Omega)}=\sum_{|\alpha|\leq k}\int_{\Omega}\partial^{\alpha} u\overline{\partial^{\alpha}v}\,dx.
\end{equation}
As usual, we set 
\[
    W^{k,p}_0(\Omega)\vcentcolon = \|\cdot\|_{W^{k,p}(\Omega)}-\text{clos}(C_c^{\infty}(\Omega))\text{ and }H^k_0(\Omega)\vcentcolon = W^{k,2}_0(\Omega)
\]
for all $k\in\N$ and $1\leq p<\infty$. Similarly, if $X$ is a generic Banach space, $-\infty\leq a<b\leq \infty$ and $1\leq p\leq \infty$, then $L^p(a,b;X)$ stands for the space of $p$-summable functions with values in $X$ and $W^{k,p}(a,b;X)$ the space of distributions on $(a,b)$ with values in $X$ such that all derivatives up to order $k$ belong to $L^p(a,b;X)$. In the case $H^k(a,b;X)\vcentcolon = W^{k,2}(a,b;X)$, these are again Hilbert spaces when $X$ is a Hilbert space. These spaces are normed in exactly the same way as in the previous case $L^p(
\Omega)$ and $W^{k,p}(\Omega)$. In particular, one has $L^2(a,b;L^2(\Omega))=L^2(\Omega\times (a,b))$. It is a well-known fact that if $\omega\subset\R^n$, $-\infty\leq a<b\leq \infty$ and $k\in\N_0$, then one has
\begin{equation}
\label{eq: Sobolev spaces on wave guides}
    H^k(\omega\times (a,b))=\{u\,;\partial_y^{\ell} u \in L^2(a,b;H^{k-\ell}(\omega))\text{ for }0\leq\ell\leq k\}
\end{equation}
and
\begin{equation}
    \label{eq: Sobolev space for infinite waveguide}\|u\|^2_{H^k(\omega\times(a,b))}=\sum_{\ell=0}^k\|\partial_y^{\ell}u\|^2_{L^2(a,b;H^{k-\ell}(\omega))},
\end{equation}
where points in $\omega\times (a,b)$ are denoted by $X=(x,y)$. We also make use of the spaces $H^s(\R^n)$ for $s\in\R$. If we denote by $\tempered(\R^n)$ the space of tempered distributions and $\langle D\rangle^s$ the Bessel potential operator of order $s$, whose Fourier symbol is $\langle \xi\rangle^s$ with
\[
    \langle\xi\rangle=(1+|\xi|^2)^{1/2},
\]
then we may introduce the Hilbert spaces
\begin{equation}
\label{eq: fractiona Sobolev}
    H^{s}(\mathbb{R}^n)= \left\{ u \in \tempered(\mathbb{R}^n)\,;\, \|u\|_{H^s(\R^n)}\vcentcolon = \|\langle D\rangle^s u\|_{L^2(\R^n)}<\infty\right\}.
\end{equation}
These spaces are commonly called \emph{fractional Sobolev spaces} and coincide with $H^k(\R^n)$ for $k\in\Z$. If $\Omega\subset\R^n$ is generic open set, then $H^s(\Omega)$ is the space of restrictions of distributions in $H^s(\R^n)$ to $\Omega$. Alternatively, if $\Omega=\R^n$ or $\Omega\Subset\R^n$ is a regular domain, then it can be characterized as the interpolation space
\begin{equation}
\label{eq: def fractional sobolev via interpolation}
    H^s(\Omega)=[H^k(\Omega),L^2(\Omega)]_{\theta}
\end{equation}
with equivalent norms, where $s=(1-\theta)k$ for some $k\in\N$ and $0<\theta<1$, or $[H^{k+1}(\Omega),H^k(\Omega)]_{1-\theta}$ when $s=k+\theta$, $0<\theta<1$ and $k\in\N$. Via a covering of $\partial\Omega$ by charts $(\mathscr{O}_j,\varphi_j)_{j=1,\ldots,N}$ and an associated partition of unity $(\alpha_j)_{j=1,\ldots,N}$, we may define the Hilbert space
\begin{equation}
\label{eq: fractional sobolev for infinite waveguide}
    H^s(\partial\Omega)=\{u\in \distr(\partial\Omega)\,;\,\|u\|^2_{H^s(\partial\Omega)}\vcentcolon =\sum_{j=1}^N\|\varphi_j^*(\alpha_j u)\|^2_{H^s(\R^{n-1})}<\infty\},
\end{equation}
which is independent of the used charts or the partition of unity. In all of the above cases, we have the following interpolation result
\begin{equation}
\label{eq: interpolation}
    [H^{s_2}(X),H^{s_1}(X)]_{\theta}=H^{(1-\theta)s_2+\theta s_1}(X)
\end{equation}
for all $0\leq s_1<s_2<\infty$ and $0<\theta<1$ (with equivalent norms), where $X=\R^n$, $X=\Omega$ or $X=\partial\Omega$ when $\Omega\Subset\R^n$ is a regular domain. The same result holds for negative exponents $s_j$ as long as these numbers and their convex combination $(1-\theta)s_2+\theta s_1$ are not half integers. Motivated by \eqref{eq: Sobolev spaces on wave guides}, we may define
\begin{equation}
H^s(\partial\Omega)= H^s(\R_{y};L^2(\partial\omega))\cap L^2 (\R_{y};H^s(\partial\omega)),
\end{equation}
where $\Omega=\omega\times\R$ with $\omega\Subset\R^n$ being a regular domain. Here, the first space can be defined by Fourier transform (similarly as we did for $H^s(\R^n)$), which is a unitary isomorphism on $L^2(\R;H)$ when $H$ is a Hilbert space. 

More generally, we may define the \emph{space-time Sobolev spaces}
\begin{equation}
\label{eq: more general sobolev}
    H^{r,s}(X\times (0,T))\vcentcolon = H^r(0,T;L^2(\Omega))\cap L^2(0,T;H^s(\Omega))
\end{equation}
for $X=\Omega$ or $X=\partial\Omega$, with $\Omega$ being the whole space, a bound regular domain $\Omega\Subset\R^n$ or an infinite waveguide $\Omega=\omega\times\R$ as above, and all $r,s\geq 0$. These spaces are again Hilbert spaces under the norm 
\begin{equation}
    \|u\|_{H^{r,s}(\Omega_T)}^2=\|u\|_{H^r(0,T;L^2(\Omega))}^2+\|u\|_{L^2(0,T;H^s(\Omega))}^2.
\end{equation}

\section{Well-posedness of the relativistic wave equation}
\label{sec: Existence}

In this section we will show the well-posedness of the Dirichlet problem \eqref{maineqn} in suitable function spaces. For later convenience, let us denote the operator in \eqref{eq: relativistic schroedinger} by
\begin{equation}
\label{eq: op L AP}
     L_{\mathcal{A},\Phi}=(\partial_t +A_0)^2-(\nabla_x + \im A)^2-\partial_y^2+\Phi,
\end{equation}
which can be rewritten as
\begin{equation}
\label{eq: equivalent form L AP}
    L_{\mathcal{A},\Phi}=\partial_t^2-\Delta-(2\im A\cdot \nabla_x + \im \Div_x A-|A|^2-|A_0|^2-\Phi)+2A_0\partial_t,
\end{equation}
whenever the functions $A_0$ and $A$ are sufficiently regular. We have the following well-posedness result. 

\begin{theorem}\label{thm: Existence}
 Let $\Omega=\omega\times\R$ be an infinite waveguide with $\omega\subset\R^2$ being a smoothly bounded domain and $T>0$. Assume that the vector potential  $\mathcal{A}=(A_0,\im A)\in L^{\infty}(\Omega)$ with $\Div_x A\in L^{\infty}(\Omega)$, $A_\nu=\nu\cdot A=0$ on $\partial\Omega$ and $\Phi\in L^{\infty}(\Omega)$. For any boundary data $f \in  D^{3/2}((\partial\Omega)_T) $, there exists a unique solution to 
 \begin{equation}\label{maineqn well-posedness}
    \begin{cases}
     L_{\mathcal{A},\Phi}u = 0 & \text{in } \Omega_T,\\
        u = f & \text{on } (\partial\Omega)_T,\\
          u(0) = 0,\, \partial_{t}u(0) = 0 & \text{in } \Omega
    \end{cases}
\end{equation}
 such that 
 \[
    u\in C ([0,T];  H^{1}(\Omega))\cap C^1 ([0,T]; L^{2}(\Omega)).
 \]
 Furthermore, we have
 $\partial_{\nu}u \in L^{2}((\partial\Omega)_T)$ and there holds
 \begin{equation}
 \label{eq: full estimate well-posedness thm}
\|\partial_\nu u\|_{L^{2}((\partial\Omega)_T)} +\|u\|_{L^{\infty}(0,T;H^1(\Omega))}
+\|\partial_t u\|_{ L^{\infty}(0,T;L^2(\Omega))}
\leq C_0 \|f\|_{D^{3/2}((\partial\Omega)_T)}
\end{equation}
 for some constant $C_0=C_0(\Omega,T,\|\mathcal{A}\|_{L^{\infty}(\Omega)},\|\Div_x A\|_{L^{\infty}(\Omega)},\|\Phi\|_{L^{\infty}(\Omega)})>0$ depending non-decreasingly on the involved norms.
\end{theorem}

To establish this well-posedness result, we will make use of the following lifting lemma for Sobolev functions.

\begin{lemma}[{\cite[Theorem~4.2]{Kia14}}]\label{w_{f}.}
    Let $\Omega=\omega\times\R$ be an infinite waveguide with $\omega\subset\R^2$ being a smoothly bounded domain and $T>0$. For all $f \in D^{3/2}((\partial\Omega)_T)$, there exists $w_{f} \in H^{2,2}(\Omega_T)$ satisfying 
    \begin{equation}
        \begin{cases}
            w_f=f,\,\partial_{\nu}w_f=0 &\text{on } (\partial\Omega)_T,\\
             w_f (0)=\partial_t w_f(0)=0 &\text{in }\Omega
        \end{cases}
    \end{equation}
    and 
    \begin{equation*}
        \|w_{f}\|_{H^{2,2}(\Omega_T)}\leq C \|f\|_{D^{3/2}((\partial\Omega)_T)}
    \end{equation*}
    for some $C>0$ independent of $f$.
\end{lemma}

Next, we prove the above well-posedness result.

\begin{proof}[Proof of theorem \ref{thm: Existence}]
\noindent\textit{Step 1.} In the first step, we show the well-posedness of the following homogeneous problem 
\begin{equation}
\label{eq: well-posedness homogeneous}
     \begin{cases}
     L_{\mathcal{A},\Phi}u = h & \text{in } \Omega_T,\\
        u = 0 & \text{on } (\partial\Omega)_T,\\
          u(0) = 0,\, \partial_{t}u(0) = 0 & \text{on } \Omega,
    \end{cases}
\end{equation}
where $h\in L^2(\Omega_T)$ is a given function. We aim to apply the general results in \cite[Chapter XVIII, \S 5]{DautrayLionsVol5}. 

First, let us recall that on $\Omega$ we have the following Poincar\'e inequality
    \begin{equation}
    \label{eq: Poincare}
        \|u\|_{L^2(\Omega)}\lesssim \|\nabla u\|_{L^2(\Omega)}
    \end{equation}
    for all $u\in H^1_0(\Omega)$. So, we can endow $H^1_0(\Omega)$ with the equivalent norm
    \begin{equation}
    \label{eq: equivalent norm}
        \|u\|_{H^1_0(\Omega)}=\|\nabla u\|_{L^2(\Omega)}.
    \end{equation}
    The spaces
    \[
        H^1_0(\Omega)\hookrightarrow L^2(\Omega)\hookrightarrow (H^1_0(\Omega))'=H^{-1}(\Omega)
    \]
    form a Gelfand triple, i.e. the embeddings are continuous and dense. Following the notation of \cite[Chapter XVIII, \S 5]{DautrayLionsVol5} we set
    \begin{equation}
    \label{eq: sesquilinear fomr a}
        \begin{split}
            a&=a_0+a_1 \text{ with }a_0(u,v)=\langle u,v\rangle_{H^1_0(\Omega)},\\
            a_1(u,v)&=-\langle (2 \im A\cdot \nabla_x + \im \Div_x A-|A|^2-|A_0|^2-\Phi)u,v\rangle_{L^2(\Omega)}
        \end{split}
    \end{equation}
    (see \eqref{eq: equivalent form L AP}).
    Clearly, $a_0$ is continuous, hermitian and coercive. So, $a_0$ fulfills all required properties of \cite[Chapter XVIII, \S 5, Sectiom~1.2]{DautrayLionsVol5} and the related operator is occasionally denoted by $\mathscr{A}_0=-\Delta\in L(H^1_0(\Omega),H^{-1}(\Omega))$\footnote{$L(X,Y)$ stands for the space of continuous linear operators between normed spaces $X$ and $Y$.}. For $a_1$ we need to show that
    \begin{equation}
    \label{eq: continuity estimate a1}
        |a_1(u,v)|\lesssim \|u\|_{H^1_0(\Omega)}\|v\|_{L^2(\Omega)}\text{ for all }u,v\in H^1_0(\Omega).
    \end{equation}
    This implies that the operator associated to $a_1$ satisfies $\mathscr{A}_1\in L(H^1_0(\Omega),H^{-1}(\Omega))$. The continuity estimate \eqref{eq: continuity estimate a1} follows from the Poincar\'e and H\"older inequality. Indeed, we have
    \begin{equation}
    \label{eq: H1 L2 estimate}
        \begin{split}
        &\|(2 \im A\cdot \nabla_x + \im \Div_x A-|A|^2-|A_0|^2-\Phi)u\|_{L^2(\Omega)}\\
        &\lesssim \|A\|_{L^{\infty}(\Omega)}\|\nabla u\|_{L^2(\Omega)}\\
        &\quad +(\|\Div_x A\|_{L^{\infty}(\Omega)}+\|A\|_{L^{\infty}(\Omega)}^2+\|A_0\|^2_{L^{\infty}(\Omega)}+\|\Phi\|_{L^{\infty}(\Omega)})\|u\|_{L^2(\Omega)}\\
        &\lesssim (\|\Div_x A\|_{L^{\infty}(\Omega)}+(1+\|A\|_{L^{\infty}(\Omega)})^2+\|A_0\|^2_{L^{\infty}(\Omega)}+\|\Phi\|_{L^{\infty}(\Omega)})\|\nabla u\|_{L^2(\Omega)}.
    \end{split}
    \end{equation}
    For uniqueness of solutions to \eqref{eq: well-posedness homogeneous} we also need that $\mathscr{A}_1\in L(H^1_0(\Omega),L^2(\Omega)) \cap L(L^2(\Omega),H^{-1}(\Omega))$ (see \cite[Chapter XVIII, \S 5, Theorem~3]{DautrayLionsVol5}). The first assertion is a consequence of the estimate \eqref{eq: H1 L2 estimate}. The second one follows by duality, i.e.
    \[
    \begin{split}
        \langle \mathscr{A}_1 u,v\rangle_{H^{-1}(\Omega)\times H^1_0(\Omega)}&=-\langle u, 2 \im\partial_j (A_j v)\rangle_{L^2(\Omega)}\\
        &\quad+\langle u,( \im\Div_x A-|A|^2-|A_0|^2-\Phi)v\rangle_{L^2(\Omega)}.
    \end{split}
    \]
  
    Arguing as above, we see that there holds
    \[
         \begin{split}
              &|\langle \mathscr{A}_1 u,v\rangle_{H^{-1}(\Omega)\times H^1_0(\Omega)}|\\
              &\lesssim (\|\Div_x A\|_{L^{\infty}(\Omega)}+(1+\|A\|_{L^{\infty}(\Omega)})^2+\|A_0\|^2_{L^{\infty}(\Omega)}+\|\Phi\|_{L^{\infty}(\Omega)})\|u\|_{L^2(\Omega)}\|v\|_{H^1_0(\Omega)}
         \end{split}  
    \]
    for any $u\in L^2(\Omega),v\in H^1_0(\Omega)$ and so
    \[
        \|\mathscr{A}_1u\|_{H^{-1}(\Omega)}\lesssim (\|\Div_x A\|_{L^{\infty}(\Omega)}+(1+\|A\|_{L^{\infty}(\Omega)})^2+\|A_0\|^2_{L^{\infty}(\Omega)}+\|\Phi\|_{L^{\infty}(\Omega)})\|u\|_{L^2(\Omega)}.
    \]
    This yields $\mathscr{A}_1\in L(L^2(\Omega),H^{-1}(\Omega))$.
    
    In our case the sesquilinear form $\beta$ and the related operator $\mathscr{B}$ from \cite[Chapter XVIII, \S 5, Section~5.1]{DautrayLionsVol5} are given by
    \[
        \beta(u,v)=\langle \mathscr{B}u,v\rangle_{L^2(\Omega)}\text{ with }\mathscr{B}u= 2A_0 u
    \]
    and the operator $\mathscr{B}$ clearly maps $L^2(\Omega)$ to itself. As the operator $\mathscr{C}$ in \cite[Chapter XVIII, \S 5, Section~1.2]{DautrayLionsVol5}, we choose the identity map from $L^2(\Omega)$ to itself. Now, we can invoke the existence and uniqueness results \cite[Chapter XVIII, \S 5, Theorem~3 \& 4]{DautrayLionsVol5}, which give the existence of a unique solution $u\in C([0,T];H^1_0(\Omega))\cap C^1([0,T];L^2(\Omega))$ to \eqref{eq: well-posedness homogeneous}. Furthermore, by \cite[Chapter XVIII, \S 5, Lemma~7]{DautrayLionsVol5} the solution $ u$ satisfies the following energy identity
    \begin{equation}
    \label{eq: energy identity}
    \begin{split}
         &\|\partial_t u(t)\|_{L^2(\Omega)}^2+\|\nabla u(t)\|_{L^2(\Omega)}^2+4\int_0^t \langle A_0\partial_t u(\tau),\partial_t u(\tau)\rangle_{L^2(\Omega)}\,d\tau \\
         &-2\re {\int_0^t\langle (2 \im A\cdot \nabla_x + \im \Div_x A-|A|^2-|A_0|^2-\Phi)u(\tau),\partial_t u(\tau)\rangle_{L^2(\Omega)}\,d\tau}\\
         &=2 \re\int_0^t\langle h(\tau),\partial_tu (\tau)\rangle_{L^2(\Omega)}\,d\tau
    \end{split}
    \end{equation}
    for $0\leq t\leq T$. Let us introduce the nonnegative, continuous quantity 
    \[
        \Psi(t)\vcentcolon = \|\partial_t u(t)\|_{L^2(\Omega)}^2+\|\nabla u(t)\|_{L^2(\Omega)}^2
    \]
    for $0\leq t\leq T$. Then, using the Poincar\'e and H\"older inequality we get 
    \[
    \begin{split}
        &\Psi(t)\leq \int_0^t \|h(\tau)\|_{L^2(\Omega)}^2\,d\tau\\
        &+ C\int_0^t ((1+\|A_0\|_{L^{\infty}(\Omega)}+\|A\|_{L^{\infty}(\Omega)})^2+\|\Div_x A\|_{L^{\infty}(\Omega)}+\|\Phi\|_{L^{\infty}(\Omega)})\Psi(\tau)\,d\tau
    \end{split}
    \]
    for $0\leq t\leq T$. Thus, Gronwall's inequality ensures that we have the following energy inequality
    \begin{equation}
    \label{eq: energy inequality}
        \begin{split}
            &\|\partial_t u(t)\|_{L^2(\Omega)}^2+\|\nabla u(t)\|_{L^2(\Omega)}^2\leq \|h\|_{L^2(\Omega_t)}^2\\
            &\quad\cdot\exp\left(Ct\left[(1+\|A_0\|_{L^{\infty}(\Omega)}+\|A\|_{L^{\infty}(\Omega)})^2+\|\Div_x A\|_{L^{\infty}(\Omega)}+\|\Phi\|_{L^{\infty}(\Omega)}\right]\right)
        \end{split}
    \end{equation}
    for $0\leq t\leq T$. In particular, we get
    \begin{equation}
     \label{eq: energy inequality 2}
        \|\partial_t u\|_{L^{\infty}(0,T;L^2(\Omega))}+\|\nabla u\|_{L^{\infty}(0,T;L^2(\Omega))}\leq C\|h\|_{L^2(\Omega_T)}
    \end{equation}
    for some $C>0$ only depending on $\Omega, T$ and the $L^{\infty}(\Omega)-$norms of $A_0,A,\Div_x A$ and $\Phi$.\\
    
\noindent{\textit{Step 2.}} Next, we construct a solution to \eqref{maineqn well-posedness}. If we decompose the seeked solution $u$ as $u= v+w_{f}$, where $w_{f}\in H^{2,2}(\Omega_T)$ is the function from Lemma~\ref{w_{f}.}, then $v$ needs to solve
     \begin{equation}
     \label{eq. for v_{2}}
      \begin{cases}
       L_{\mathcal{A},\Phi}v = -L_{\mathcal{A},\Phi}w_{f}  &\text{in}~~~~~ \Omega_T,\\
    v = 0 & \text{on} ~~~~(\partial\Omega)_T,\\
     v(0)=0,~~ \partial_tv(0) = 0  &\text{in}~~~~\Omega.
    \end{cases}
 \end{equation}
 Since $w_{f} \in H^{2,2}(\Omega_T)$, we have  $L_{\mathcal{A},\Phi}w_{f} \in L^{2}(\Omega_T)$ and we can invoke the results from \textit{Step 1} to deduce the existence of a unique solution $v\in C([0,T];H^1_0(\Omega))\cap C^1([0,T];L^2(\Omega))$ of \eqref{eq. for v_{2}} satisfying the energy estimate
 \begin{equation}
 \label{Energy Estimate for v_{2}}
 \begin{split}
   \|\partial_t v\|_{L^{\infty}(0,T;L^2(\Omega))}+ \|\nabla v\|_{L^{\infty}(0,T;L^2(\Omega))} &\leq C \| L_{\mathcal{A},\Phi}w_{f} \|_{L^{2}(\Omega_T)}\\
   & \leq C \| w_{f} \|_{H^{2,2}(\Omega_T)}
   \end{split}
 \end{equation}
 (see \eqref{eq: energy inequality 2}, \eqref{eq: equivalent form L AP} and \eqref{eq: H1 L2 estimate}). By the above estimates the constant $C>0$ again only depends on $\Omega, T$ and the $L^{\infty}(\Omega)-$norms of $A_0,A,\Div_x A$ and $\Phi$. Next, note that there holds
  \begin{equation}\label{eqn for w}
    \|\partial_t w_{f}\|_{L^{\infty}(0,T;L^2(\Omega))} +\| w_{f}\|_{L^{\infty}(0,T;H^1(\Omega))}\leq C \| w_{f} \|_{H^{2,2}(\Omega_T)}.
 \end{equation}
 The bound for the time derivative is a consequence of the classical embedding $H^1(0,T;L^2(\Omega))\hookrightarrow C([0,T];L^2(\Omega))$ and the estimate for the gradient follows from \cite[Theorem~3.1]{lions2012nonVol1}, which in particular ensures that there holds
 \[
    H^{2,2}(\Omega_T)\hookrightarrow C([0,T];H^{3/2}(\Omega))\hookrightarrow C([0,T];H^{1}(\Omega)).
 \]
 Now, recalling that $u=v+w_f$ and using the Poincar\'e inequality, \eqref{Energy Estimate for v_{2}} and \eqref{eqn for w} we deduce the estimate
\begin{equation}
     \|\partial_t u\|_{L^{\infty}(0,T;L^2(\Omega))} + \|u\|_{L^{\infty}(0,T;H^1(\Omega))} \leq C \| w_{f} \|_{H^{2,2}(\Omega_T)}.
\end{equation}
Taking Lemma~\ref{w_{f}.} into account, we get
\begin{equation}
\label{add1}
    \|\partial_t u\|_{L^{\infty}(0,T;L^2(\Omega))} + \|\nabla u\|_{L^{\infty}(0,T;L^2(\Omega))} \leq C \| f \|_{D^{3/2}((\partial\Omega)_T)}.
\end{equation}
Thus, it remains to show that $\partial_{\nu}u\in L^2((\partial\Omega)_T)$ and there holds
\begin{equation}
\label{eq: estimate normal derivative}
    \|\partial_\nu u\|_{L^2((\partial\Omega)_T)}\leq C\|w_f\|_{H^{2,2}(\Omega_T)}
\end{equation}
as then Lemma~\ref{w_{f}.} and \eqref{add1} implies the estimate \eqref{eq: full estimate well-posedness thm}. First notice that by Lemma~\ref{w_{f}.} we have $\partial_\nu u=\partial_\nu v$ and so by \eqref{Energy Estimate for v_{2}} and the Poincar\'e inequality, it is enough to show $\partial_\nu v\in L^2((\partial\Omega)_T)$ obeying the estimate
\begin{equation}
\label{eq: remaining estimate for normal derivative}
    \|\partial_\nu v\|_{L^2((\partial\Omega)_T)}\lesssim \|w_f\|_{H^{2,2}(\Omega_T)}+\|\partial_t v\|_{L^{\infty}(0,T;L^2(\Omega))}+\|v\|_{L^{\infty}(0,T;H^1(\Omega))}.
\end{equation}
For this let us observe that $v$ solves
\begin{equation}
    \begin{cases}
       (\partial_t^2-\Delta)v = F  &\text{in}~~~~~ \Omega_T,\\
    v = 0 & \text{on} ~~~~(\partial\Omega)_T,\\
     v(0)=0,~~ \partial_tv(0) = 0  &\text{in}~~~~\Omega,
    \end{cases}
\end{equation}
where $F$ is given by
\[
    F= (2 \im A\cdot \nabla_x + \im \Div_x A-|A|^2-|A_0|^2-\Phi)v-2A_0\partial_tv-L_{\mathcal{A},\Phi}w_{f}.
\]
Using \eqref{Energy Estimate for v_{2}} and \eqref{eq: H1 L2 estimate}, we get
\small{\begin{align}
\label{eq: estimate for source}
    &\|F\|_{L^2(\Omega_T)}\leq C(\|A_0\|_{L^{\infty}(\Omega)}\|\partial_t v\|_{L^{\infty}(0,T;L^2(\Omega))}+\|w_f\|_{H^{2,2}(\Omega_T)}) \\
    &+C(\|\Div_x A\|_{L^{\infty}(\Omega)}+(1+\|A\|_{L^{\infty}(\Omega)})^2+\|A_0\|^2_{L^{\infty}(\Omega)}+\|\Phi\|_{L^{\infty}(\Omega)})\|\nabla v\|_{L^{\infty}(0,T;L^2(\Omega))}\\
    &\leq C(\|w_f\|_{H^{2,2}(\Omega_T)}+\|\partial_t v\|_{L^{\infty}(0,T;L^2(\Omega))}+\|v\|_{L^{\infty}(0,T;H^1(\Omega))}).
\end{align}}
From \cite[Proof of Theorem~4.1]{Kia14} we deduce that $\partial_\nu v\in L^2((\partial\Omega)_T)$ and there holds
\[
    \|\partial_\nu v\|_{L^2((\partial\Omega)_T)}\leq C\|F\|_{L^2(\Omega_T)}.
\]
Thus, combining this with \eqref{eq: estimate for source}, we arrive at the estimate \eqref{eq: remaining estimate for normal derivative}. This gives $\partial_\nu u\in L^2((\partial\Omega)_T)$ and we see that \eqref{eq: estimate normal derivative} holds. The assertion on the constant $C_0>0$ in \eqref{eq: full estimate well-posedness thm} is easily seen by keeping track of the previous estimates.

Finally, the uniqueness of solutions to \eqref{maineqn well-posedness} follows from the unique solvability of \eqref{eq: well-posedness homogeneous}. More precisely, if $u,v$ are two solutions to \eqref{maineqn well-posedness}, then $\widetilde{u}=u-w_f$ and $\widetilde{v}=v-w_f$ solve the same problem \eqref{eq. for v_{2}} with homogeneous boundary conditions and thus we necessarily have $\widetilde{u}=\widetilde{v}$. In conclusion, the constructed solution $u$ is the unique solution of \eqref{maineqn well-posedness}.
\end{proof}

 \section{Geometric Optics Solutions} 
 \label{sec: geometric optics sol}
In this section we prove the existence of geometric optics solutions to the relativistic wave equation \eqref{eq: relativistic schroedinger}.

\begin{lemma}\label{lemma: geom opt sol +}
Let $\Omega=\omega\times\R$ be an infinite waveguide with $\omega\subset\R^2$ being a smoothly bounded domain and $T>0$. Assume that the vector potential  $\mathcal{A}=(A_0,\im A)\in C^2_0(\Omega)$ and $\Phi\in L^{\infty}(\Omega)$. Furthermore, let $\theta \in \mathbb{S}^{1}:=\partial B_1(0,\R^2)$, $h\in  H^2(\R)$ and $\varphi\in H^3(\R^{2})$. For any $\lambda > 0$, there exists a solution 
\begin{equation}
\label{eq: CGO sol}
    u^{+}_{\lambda} \in C^{1}([0,T], L^{2}(\Omega))\cap C([0,T], H^{1}(\Omega))
\end{equation}
of the relativistic wave equation
\begin{equation}\label{rel Schroed CGO}
 L_{\mathcal{A},\Phi}u = 0 \text{ in } \Omega_T
 \end{equation}
having the form 
\begin{equation}\label{hypdesoln22}
u^{+}_{\lambda}(X,t) = \varphi(x + t\theta)h(y)A^{+}(X,t)e^{\im\lambda(x\cdot \theta + t)} + r_{\lambda}^{+}(X,t).
\end{equation}
Here, $A^{+}(X,t)$ is given by 
\begin{equation}\label{eq: form of A^(+)}
A^{+}(X,t) = \exp\left(- \int_{0}^{t}{(1,-\theta)\cdot \mathcal{A}(x+s\theta,y)\,ds}\right)    
\end{equation}
and $r_{\lambda}^{+}\in C([0,T];H^1_0(\Omega))\cap C^1([0,T];L^2(\Omega))$ satisfies
\begin{equation}\label{hypdesoln23}
   r_{\lambda}^{+}(0) = \partial_{t}r_{\lambda}^{+}(0) =  0   \text{ in } \Omega.
\end{equation}
Moreover, there exists a constant $C=C\big(\Omega, T, \|\Phi\|_{L^{\infty}(\Omega)},\|\mathcal{A}\|_{W^{2,\infty}(\Omega)}\big) > 0$ such that
\begin{equation}\label{eq: estimate r plus geom opt}
\begin{split}
 &\|\partial_{t}  r_{\lambda}^{+}\|_{L^{\infty}(0,T;L^{2}(\Omega))} + \|\nabla r_{\lambda}^{+}\|_{L^{\infty}(0,T;L^{2}(\Omega))}+\lambda \|r_{\lambda}^{+}\|_{L^{\infty}(0,T;L^{2}(\Omega))}  \\
 &\quad \leq C \|\varphi\|_{H^{3}(\R^{2})}\|h\|_{H^{2}(\R)}.
 \end{split}
\end{equation}
\end{lemma}

\begin{proof}
First note that in light of \eqref{rel Schroed CGO}, \eqref{eq: form of A^(+)} and \eqref{hypdesoln23}, it is enough to prove the existence of a function $r_{\lambda}^{+}\in C([0,T];H^1_0(\Omega))\cap C^1([0,T];L^2(\Omega))$ solving
 \begin{equation}
 \label{eq: PDE for r plus}
\begin{cases}
L_{\mathcal{A},\Phi}r_{\lambda}^{+} = g & \text{in } \Omega_T,\\
r_{\lambda}^{+} = 0 & \text{on } (\partial\Omega)_T,\\
r_{\lambda}^{+}(0)  = \partial_{t}r_{\lambda}^{+}(0)=0 &  \text{in } \Omega
\end{cases}
\end{equation}
and satisfying the estimate \eqref{eq: estimate r plus geom opt}, where  $g\colon\Omega_T\to \C$ is given by
\begin{equation}
\label{eq: def of g}
    g(X,t)= -L_{\mathcal{A},\Phi}\big[ \varphi(x + t\theta)h(y)A^{+}(X,t)e^{\im\lambda( x \cdot \theta + t )}\big]\in L^2(\Omega_T).
\end{equation}

The $L^2(\Omega_T)$ regularity can be easily seen by \eqref{eq: equivalent form L AP} and recalling that Sobolev functions are absolutely continuous on lines (see \cite[Section 11.3]{leoni2024first}). Before constructing the function $r^{+}_\lambda$, let us show the following assertion.
\begin{claim}
\label{claim r plus inhom}
   The function $g$ can be written as
\begin{equation}
\label{eq: equiv form of g}
\begin{split}
    g(X,t) =-e^{\im\lambda( x\cdot \theta + t )}g_{0}(X,t),
\end{split}
\end{equation}
where $g_0$ is given by
\begin{equation}
\label{eq: def of g0}
    g_0(X,t)\vcentcolon =L_{\mathcal{A},\Phi}\left(\varphi(x + t\theta)h(y)A^{+}(X,t)\right)\in L^2(\Omega_T).
\end{equation}
\end{claim}
\begin{proof}
We first show that the function $A^{+}(X, t)$, given by \eqref{eq: form of A^(+)}, solves 
\begin{equation}
\label{eq: transport eq for Aplus}
    (\partial_{t} - \theta\cdot\nabla_{x})A^{+}(X,t) = -\big(A_{0} - \im\theta\cdot A\big)(X) A^{+}(X,t).
\end{equation}
In fact, an integration by parts ensures that there holds
\begin{align*}
 \theta\cdot\nabla_x A^+(X,t)&= A^{+}(X,t) \left(- \int_{0}^{t}{(1,-\im\theta)_{k} \theta_{j}\partial_{{j}}A_{k}(x+s\theta,y)\,ds}\right)\\
 &=A^{+}(X,t) \left( -(1,-\im\theta)_{k}\int_{0}^{t}{ \frac{d}{ds}A_{k}(x+s\theta,y)\,ds}\right)\\
 &= A^{+}(X,t)\big(  -(1,-\theta)\cdot\mathcal{A}(x+t\theta,y)  +(A_0-\im\theta\cdot A)(X)\big)  \\ 
 &= \partial_{t}A^{+}(X,t)+(A_0-\im\theta\cdot A)(X)A^{+}(X,t),
\end{align*}
which is equivalent to \eqref{eq: transport eq for Aplus}. 
To see the identity \eqref{eq: equiv form of g}, let us define the d'Alembertian $\Box_{x,t} =  \partial_t^2 - \Delta_x$ and introduce the functions
\begin{equation}
\label{eq: auxiliary fcts}
   \tilde{\Phi}(X) = (|A_{0}|^2- \im \Div_x A + |A|^{2} + \Phi)(X) \text{ and } f(X,t)=\varphi(x+t\theta)h(y)A^{+}(X, t).
\end{equation}
Then we have
\begin{equation}
L_{\mathcal{A}, \Phi} = \Box_{x,t} - \partial_y^2 + 2A_0 \partial_t - 2 \im A \cdot \nabla_x + \tilde{\Phi}
\end{equation}
(see~\eqref{eq: equivalent form L AP}) and the product rule implies
\begin{equation}\label{eqn for A^{+}(x,y,t)}
\begin{split}
-g&=L_{\mathcal{A}, \Phi} \big(e^{\im\lambda(x \cdot \theta + t )}f\big)\\
&= e^{\im \lambda (x \cdot \theta+ t) } \Box_{x,t} f + f \underbrace{\Box_{x,t} e^{\im \lambda (x \cdot \theta + t )} }_{=0} + 2 \im\lambda e^{\im \lambda ( x \cdot \theta + t )} ( \partial_{t} f - \theta  \cdot \nabla_{x} f ) \\ 
&\quad  - e^{\im \lambda (x \cdot \theta + t)} \partial_y^2 f  + 2 A_{0} e^{\im \lambda (x \cdot \theta + t)} \partial_{t}f  - 2 \im e^{\im \lambda (x \cdot \theta + t)} A \cdot \nabla_{x}f \\  
&\quad  +2\im \lambda e^{\im \lambda (x \cdot \theta + t)}( A_{0}  f  - \im A \cdot \theta f) + \Tilde{\Phi} e^{\im \lambda (x \cdot \theta + t)}f  \\ 
&= e^{\im \lambda (x \cdot \theta + t)} L_{\mathcal{A}, \Phi} f  +2\im \lambda e^{\im \lambda (x \cdot \theta +t)} \left(   \partial_{t} f -  \theta \cdot \nabla_{x} f +  (A_{0}  - \im \theta \cdot A )f \right),
\end{split}
\end{equation}
where the first line comes from the product rule for the d'Alembertian $\Box_{x,t}$. Now, using \eqref{eq: auxiliary fcts}, \eqref{eq: transport eq for Aplus} and $(\partial_t-\theta\cdot\nabla_x)\varphi(x+t\theta)=0$, we get
 \begin{equation}
 \label{eq: help computation}
 \begin{split}
    (A_{0}-  \im \theta\cdot A)(X) f(X,t)
   & =-\varphi(x+t\theta)h(y)(\partial_{t} - \theta\cdot\nabla_{x})A^{+}(X,t)\\
  & =-(\partial_t - \theta \cdot \nabla_x)\varphi(x+t\theta)h(y)A^{+}(X, t)\\
  & =-(\partial_t - \theta \cdot \nabla_x)f(X,t).
  \end{split}
\end{equation}
Therefore, \eqref{eqn for A^{+}(x,y,t)} and \eqref{eq: help computation} imply that $g$ is of the form \eqref{eq: equiv form of g} with $g_0$ given by \eqref{eq: def of g0}.
\end{proof}

Next, we may observe that the existence of $r^{+}_\lambda\in C([0,T];H^1_0(\Omega))\cap C^1([0,T];L^2(\Omega))$ solving \eqref{eq: PDE for r plus} follows from $g\in L^2(\Omega_T)$ and Step 1 of the proof of Theorem \ref{thm: Existence}. Next, we show that $r^+_\lambda$ satisfies the estimate \eqref{eq: estimate r plus geom opt}. For this purpose, let us introduce the time integral transform 
\begin{equation}
\label{eq: time integral r plus}
    R^+_\lambda(X,t) = \int_{0}^{t}{r_{\lambda}^{+}(X,s)\,ds}
\end{equation}
of $r^+_\lambda$. The following claim collects the main properties of $R^+_\lambda$, which are needed to establish \eqref{eq: estimate r plus geom opt}.
\begin{claim} 
\label{claim: properties of R plus}
The function $R^+_\lambda$ has the following properties:
    \begin{enumerate}[(i)]
        \item\label{PDE for time integral r plus} $R^+_\lambda$ solves the PDE 
\begin{equation}
\label{eq: PDE for R plus}
\begin{split}
    \begin{cases}
 L_{\mathcal{A},\Phi}R^+_\lambda  = \int_{0}^{t}g(\tau)\, d\tau & \text{in } \Omega_T,\\
R_{\lambda}^{+} = 0 & \text{on } (\partial\Omega)_T,\\
R_{\lambda}^{+}(0)  = \partial_{t}R_{\lambda}^{+}(0)=0 &  \text{in } \Omega
\end{cases}
\end{split}
\end{equation}
        \item\label{estimate for source of R plus} and the function $F_1^\lambda\vcentcolon =\int_{0}^{t}g(\tau)\, d\tau$ satisfies the estimate
        \begin{equation} 
\label{eq: estimate F1}
    \norm{F_{1}^{\lambda}}_{L^{2}(\Omega_T)} \leq  \dfrac{C}{\lambda}\norm{\varphi}_{H^{3}(\R^2)}  \norm{h}_{H^{2}(\R)} .
\end{equation}
    \end{enumerate}
\end{claim}
\begin{proof}
    \ref{PDE for time integral r plus}: First, recall that $r^+_\lambda$ solves \eqref{eq: PDE for r plus}, if $r^+_\lambda\in C([0,T];H^1_0(\Omega))\cap C^1([0,T];L^2(\Omega))$ satisfies $r^+_\lambda(0)=\partial_t r^+_\lambda(0) = 0 $ and
    \begin{equation}
    \label{eq: weak form for r plus}
    \begin{split}
        &-\int_0^T\langle \partial_t r^+_\lambda,v\rangle_{L^2(\Omega)}\partial_t \psi\, dt+2\int_0^T \langle A_0\partial_t r^+_\lambda,v\rangle_{L^2(\Omega)}\psi\,dt
 +\int_0^T a(r^+_\lambda,v)\psi\,dt\\
 &=\int_0^T\langle g(t),v\rangle_{L^2(\Omega)} \psi\,dt
 \end{split}
    \end{equation}
    for all $v\in H^1_0(\Omega)$ and $\psi\in C_c^{\infty}((0,T))$, where the sesquilinear form $a$ is given by \eqref{eq: sesquilinear fomr a}. Next, we assert that the identity indeed holds for all $\psi\in C_c^{\infty}([0,T))$ as the initial values of $r^+_\lambda$ are zero. This can be proved similarly as  \cite[Claim 4.2]{Semilinear-nonlocal-WEQ}. Indeed, let us consider for $\eps>0$ the parabolically regularized problem
    \begin{equation}
         \label{eq: parabolic reg PDE for r plus}
\begin{cases}
L^{\eps}_{\mathcal{A},\Phi}r = g & \text{in } \Omega_T,\\
r = 0 & \text{on } (\partial\Omega)_T,\\
r(0)  = \partial_{t}r(0)=0 &  \text{in } \Omega,
\end{cases}
    \end{equation}
    where $L^{\eps}_{\mathcal{A},\Phi}$ is given by
    \begin{equation}
        \label{eq: parabolic reg of L}L^{\eps}_{\mathcal{A},\Phi}=L_{\mathcal{A},\Phi}-\eps \Delta \partial_t. 
    \end{equation}
    Thanks to the additional viscosity term $-\eps \Delta\partial_t$ the regularized problem \eqref{eq: parabolic reg PDE for r plus} falls into the framework of \cite[Chapter XVIII, \S 5, Problem (P1)]{DautrayLionsVol5} and we deduce from \cite[Chapter XVIII, \S 5, Theorem 1]{DautrayLionsVol5} that for any $\eps>0$ there exists a unique solution $r^+_{\lambda,\eps}\in H^1(0,T;H^1_0(\Omega))$ with $\partial_t^2r^+_{\lambda,\eps}\in L^2(0,T;H^{-1}(\Omega))$. Furthermore, using \cite[Chapter XVIII, \S 5, Section 5.3.2]{DautrayLionsVol5}, we infer the following convergence results as $\eps\to 0$:
    \begin{enumerate}[(a)]
        \item\label{conv time der} $\partial_t r^+_{\lambda,\eps}\weakstar \partial_t r^+_{\lambda}$ in $L^{\infty}(0,T;L^2(\Omega))$,
        \item\label{conv 2nd time der} $\partial_t^2 r^+_{\lambda,\eps}\weak \partial_t^2 r^+_{\lambda}$ in $L^2(0,T;H^{-1}(\Omega))$.
    \end{enumerate}
    Now, using the usual integration by parts formula and that the initial conditions for $r^+_{\lambda}$ as well as $r^+_{\lambda,\eps}$ are zero, we may compute
    \[
    \begin{split}
        -\int_0^T\langle \partial_t r^+_\lambda,v\rangle_{L^2(\Omega)}\partial_t \psi\, dt&\overset{\ref{conv time der}}{=}-\lim_{\eps\to 0}\int_0^T\langle \partial_t r^+_{\lambda,\eps},v\rangle_{L^2(\Omega)}\partial_t \psi\, dt\\
        &\,=\lim_{\eps\to 0}\int_0^T\langle \partial_t^2 r^+_{\lambda,\eps},v\rangle_{H^{-1}(\Omega)\times H^1_0(\Omega)}\psi\, dt\\
        &\overset{\ref{conv 2nd time der}}{=}\int_0^T\langle \partial_t^2 r^+_{\lambda},v\rangle_{H^{-1}(\Omega)\times H^1_0(\Omega)}\psi\, dt
    \end{split}
    \]
    for any $v\in H^1_0(\Omega)$ and $\psi\in C_c^{\infty}((0,T))$. Hence, using the density of the inclusion $C_c^{\infty}((0,T))\otimes H^1_0(\Omega)\hookrightarrow L^2(0,T;H^1_0(\Omega))$ and \eqref{eq: weak form for r plus}, we see that
    \begin{equation}
    \label{eq: weak sol in duality}
    \begin{split}
       & \int_0^T\langle \partial_t^2 r^+_{\lambda},v\rangle_{H^{-1}(\Omega)\times H^1_0(\Omega)}\, dt+2\int_0^T \langle A_0\partial_t r^+_\lambda,v\rangle_{L^2(\Omega)}\,dt
 +\int_0^T a(r^+_\lambda,v)\psi\,dt \\ 
 &=\int_0^T\langle g(t),v\rangle_{L^2(\Omega)} \,dt
    \end{split}
    \end{equation}
    for all $v\in L^2(0,T;H^1_0(\Omega))$. Next, note that by an integration by parts and \eqref{eq: parabolic reg PDE for r plus} we have
    \[
    \begin{split}
        \int_0^T\langle \partial_t^2 r^+_{\lambda},v\rangle_{H^{-1}(\Omega)\times H^1_0(\Omega)}\psi\, dt&\overset{\ref{conv 2nd time der}}{=}\lim_{\eps\to 0}\int_0^T\langle \partial_t^2 r^+_{\lambda,\eps},v\rangle_{H^{-1}(\Omega)\times H^1_0(\Omega)}\, dt\\
        &\,=-\lim_{\eps\to 0}\int_0^T\langle \partial_t v, \partial_t r^+_{\lambda,\eps}\rangle_{H^{-1}(\Omega)\times H^1_0(\Omega)}\, dt\\
        &\,=-\lim_{\eps\to 0}\int_0^T\langle  \partial_t r^+_{\lambda,\eps},\partial_t v\rangle_{L^2(\Omega)}\, dt\\
        &\overset{\ref{conv time der}}{=}-\int_0^T\langle  \partial_t r^+_{\lambda},\partial_t v\rangle_{L^2(\Omega)}\, dt
    \end{split}
    \]
    for any $v\in C_c^{\infty}([0,T))\otimes H^1_0(\Omega)$. Hence, from \eqref{eq: weak sol in duality} we deduce that
    \begin{equation}
    \label{eq: distributional sols}
        \begin{split}
            & -\int_0^T\langle  \partial_t r^+_{\lambda},\partial_t v\rangle_{L^2(\Omega)}\, dt+2\int_0^T \langle A_0\partial_t r^+_\lambda,v\rangle_{L^2(\Omega)}\,dt
 +\int_0^T a(r^+_\lambda,v)\psi\,dt\\
 &=\int_0^T\langle g(t),v\rangle_{L^2(\Omega)} \,dt
        \end{split}
    \end{equation}
    for any $v\in C_c^{\infty}([0,T))\otimes H^1_0(\Omega)$.
    
    Now, since $R^+_{\lambda}$ is given by \eqref{eq: time integral r plus}, we clearly have $R_{\lambda}^+\in C([0,T];H^1_0(\Omega))\cap C^1([0,T];L^2(\Omega))$ and the initial values vanish. Thus, it remains to show that $R^+_\lambda$ satisfies \eqref{eq: weak form for r plus} with $g$ replaced by $\int_0^tg(\tau)\,d\tau\in L^2(\Omega_T)$. A direct computation shows that for any $\varphi\in C_c^{\infty}((0,T))$ and $v\in H^1_0(\Omega)$ there holds
    \begin{equation}
        \begin{split}
            &-\int_0^T\langle \partial_t R^+_\lambda,v\rangle_{L^2(\Omega)}\partial_t \varphi\, dt+2\int_0^T \langle A_0\partial_t R^+_\lambda,v\rangle_{L^2(\Omega)}\varphi\,dt
 +\int_0^T a(R^+_\lambda,v)\varphi\,dt\\
 &=-\int_0^T\langle r^+_\lambda,v\rangle_{L^2(\Omega)}\partial_t \varphi\, dt+2\int_0^T \langle A_0 r^+_\lambda,v\rangle_{L^2(\Omega)}\varphi\,dt
 +\int_0^T \int_0^t a(r^+_\lambda(\tau),v)\,d\tau\varphi(t)\,dt\\
 &=\int_0^T\langle \partial_t r^+_\lambda,v\rangle_{L^2(\Omega)} \varphi\, dt+2\int_0^T \langle A_0 r^+_\lambda,v\rangle_{L^2(\Omega)}\varphi\,dt
 +\int_0^T \int_0^t a(r^+_\lambda(\tau),v)\,d\tau\varphi(t)\,dt.
        \end{split}
    \end{equation}
    By introducing the function
    \[
        \psi(t)=\int_t^T \varphi(\tau)\,d\tau\in C_c^{\infty}([0,T))
    \]
    and using the previous calculation, an integration by parts, the time independence of $A_0$, $r^+_{\lambda}(0)=\partial_t r^+_{\lambda}(0)=0$, \eqref{eq: distributional sols} and Fubini's theorem, we get
    \[
    \begin{split}
         &-\int_0^T\langle \partial_t R^+_\lambda,v\rangle_{L^2(\Omega)}\partial_t \varphi\, dt+2\int_0^T \langle A_0\partial_t R^+_\lambda,v\rangle_{L^2(\Omega)}\varphi\,dt
 +\int_0^T a(R^+_\lambda,v)\varphi\,dt\\
 &=-\int_0^T\langle \partial_t r^+_\lambda,v\rangle_{L^2(\Omega)} \partial_t \psi\, dt-2\int_0^T \langle A_0 r^+_\lambda,v\rangle_{L^2(\Omega)}\partial_t\psi\,dt\\
 &\quad 
 -\int_0^T \int_0^t a(r^+_\lambda(\tau),v)\,d\tau\,\partial_t \psi(t)\,dt\\
 &=-\int_0^T\langle \partial_t r^+_\lambda,v\rangle_{L^2(\Omega)} \partial_t\psi\, dt+2\int_0^T \langle A_0 \partial_tr^+_\lambda,v\rangle_{L^2(\Omega)}\psi\,dt
 +\int_0^T a(r^+_\lambda,v)\, \psi\,dt\\
 &=\int_0^T \langle g,v\rangle_{L^2(\Omega)}\psi\,dt\\
 &=\int_0^T\left\langle\int_0^{t} g(\tau)\,d\tau,v\right\rangle_{L^2(\Omega)}\varphi\,dt
    \end{split}
    \]
    for any $v\in H^1_0(\Omega)$ and $\varphi\in C^{\infty}_c((0,T))$. Thus, $R^+_{\lambda}$ indeed solves \eqref{eq: PDE for R plus}.\\

    \noindent\ref{estimate for source of R plus}: Note that by Claim \ref{claim r plus inhom} we can write
    \[
       \begin{split}
           F_{1}^{\lambda}(X,t)&=-\int_{0}^{t}e^{\im\lambda(x\cdot \theta + \tau)}g_{0}(X,\tau)\, d\tau= \dfrac{\im}{\lambda}\int_{0}^{t}\partial_{\tau}\big(e^{\im\lambda(x\cdot \theta + \tau)}\big)g_{0}(X,\tau)\,d\tau.
       \end{split}
    \]
    As the potentials $(\mathcal{A}, \Phi)$ are time independent, an integration by parts gives
\begin{equation}
    \begin{split}
        &F_1^{\lambda}(X,t)= \dfrac{\im}{\lambda}\left(e^{\im \lambda(x\cdot\theta+t)}g_0(X,t)- e^{\im\lambda x\cdot\theta}g_0(X,0)- \int_0^te^{\im\lambda(x\cdot\theta+\tau)}\partial_\tau g_0(X,\tau)\,d\tau\right)\\
        &\quad= \dfrac{\im}{\lambda}e^{\im \lambda(x\cdot\theta+t)}L_{\mathcal{A},\Phi}\left(\varphi(x + t\theta)h(y)A^{+}(X,t)\right)\\
        &\quad\quad - \dfrac{\im}{\lambda} e^{\im\lambda x\cdot\theta}L_{\mathcal{A},\Phi}\left(\varphi(x + t\theta)h(y)A^{+}(X,t)\right)\big|_{t=0}\\
        &\quad\quad - \dfrac{\im}{\lambda}\int_0^t e^{\im\lambda(x\cdot\theta+\tau)}L_{\mathcal{A},\Phi}(\theta\cdot\nabla_x\varphi(x+\tau\theta)h(y)A^+(X,\tau))\,d\tau\\
        &\quad\quad - \dfrac{\im}{\lambda}\int_0^t e^{\im\lambda(x\cdot\theta+\tau)} L_{\mathcal{A},\Phi}(\varphi(x+\tau\theta)h(y)A^+(X,\tau)(-(1,-\theta)\cdot\mathcal{A}(x+\tau\theta,y)))\,d\tau.  
    \end{split}
\end{equation}
Using that $L_{\mathcal{A},\Phi}$ is a second order differential operator with bounded coefficients and the estimate
\[
    \left\|\int_0^t f(\cdot,\tau)\,d\tau\right\|_{L^2(\Omega_T)}\lesssim \|f\|_{L^2(\Omega_T)},
\]
which follows for example from Jensen's inequality,
we obtain the bound \eqref{eq: estimate F1}. Furthermore, we see that the constant in \eqref{eq: estimate F1} only depends on $T>0$, $\|\mathcal{A}\|_{W^{2,\infty}(\Omega)}$ and $\|\Phi\|_{L^{\infty}(\Omega)}$, but is independent of $\theta$. 
\end{proof}

Now, by the statement \ref{PDE for time integral r plus} of Claim \ref{claim: properties of R plus},
we can use the energy estimate \eqref{eq: energy inequality 2} to get
\[
    \|\partial_t R^+_\lambda(t)\|_{L^2(\Omega)}\leq C \|F_1^{\lambda}\|_{L^2(\Omega_T)}.
\]
Then Claim \ref{claim: properties of R plus}, \ref{estimate for source of R plus} implies
\begin{equation}
\label{eau1}
    \norm{r_{\lambda}^{+}(t)}_{L^2(\Omega)}= \|\partial_t R^+_\lambda(t)\|_{L^2(\Omega)}\leq C\  \norm{F_{1}^{\lambda}}_{L^{2}(\Omega_T)} \leq  \dfrac{C}{\lambda}\norm{\varphi}_{H^{3}(\R^2)}  \norm{h}_{H^{2}(\R)}.
\end{equation} 
Thus, the energy estimate \eqref{eq: energy inequality 2} for the equation \eqref{eq: PDE for r plus} and the upper bound \eqref{eau1} ensure that 
\begin{equation}
\begin{split}
     &\|\partial_{t}  {r_{\lambda}^{+}(t)}\|_{L^{2}(\Omega)} + \norm{\nabla r_{\lambda}^{+}(t)}_{L^{2}(\Omega)}+ \lambda \norm{r_{\lambda}^{+}(t)}_{L^2(\Omega)}\\
     &\leq C(\|g\|_{L^2(\Omega_T)}+\norm{\varphi}_{H^{3}(\R^2)}  \norm{h}_{H^{2}(\R)})\\
     &= C(\|\varphi\|_{H^2(\R^2)}\|h\|_{H^2(\R)}+\norm{\varphi}_{H^{3}(\R^2)}  \norm{h}_{H^{2}(\R)})\\
     &\leq C \norm{\varphi}_{H^{3}(\R^2)}  \norm{h}_{H^{2}(\R)}
\end{split}
\end{equation}
for any $t\in [0,T]$. In the above computation we again used Claim \ref{claim r plus inhom} and that $L_{\mathcal{A},\Phi}$ is a 2nd order differential operator with bounded coefficients. Hence, we can conclude the proof.
\end{proof}

Next, for any given vector potential $\mathcal{A}=(A_0,\im A)$ and external potential $\Phi$ we construct geometric optics solutions to the adjoint equation
\begin{equation}
\label{eq: adjoint eq}
    L_{\mathcal{A},\Phi}^* u=0\text{ in }\Omega_T
\end{equation}
such that the remainder term has vanishing final conditions. Note that the adjoint operator $L_{\mathcal{A},\Phi}^*$ is given by
\begin{equation}
\label{eq: adjoint op}
\begin{split}
    L^{*}_{\mathcal{A},\Phi}= L_{\mathcal{A}_r,\Phi},
\end{split}
\end{equation}
where $\mathcal{A}_r=(-A_0,\im A)$ and $L_{\mathcal{A}_r,\Phi}$ is defined by \eqref{eq: op L AP}. We have the following result.




\begin{lemma}\label{lemma: geom opt sol -}
Let $\Omega=\omega\times\R$ be an infinite waveguide with $\omega\subset\R^2$ being a smoothly bounded domain and $T>0$. Assume that the vector potential  $\mathcal{A}=(A_0, \im A)\in  C^2_0(\Omega)$ and $\Phi\in L^{\infty}(\Omega)$. Furthermore, let $\theta \in \mathbb{S}^{1}, h\in  H^2(\R)$ and $\varphi\in H^3(\R^{2})$. Then  for any $\lambda > 0$, there exists a solution 
\begin{equation} 
\label{eq: CGO sol 2 }
    u^{-}_{\lambda} \in C^{1}([0,T], L^{2}(\Omega))\cap C([0,T], H^{1}(\Omega))
\end{equation}
of the adjoint equation \eqref{eq: adjoint eq}
having the form 
\begin{equation}\label{eq: form of CGO sol 2}
u^{-}_{\lambda}(X,t) = \varphi(x + t\theta)h(y)A^{-}(X,t)e^{\im\lambda(x\cdot \theta + t)} + r_{\lambda}^{-}(X,t).
\end{equation}
Here, $A^{-}(X,t)$ is given by 
\begin{equation}\label{eq: form of A^(-)}
A^{-}(X,t) = \exp\left( \int_{0}^{t}{(1,\theta)\cdot \mathcal{A}(x+s\theta,y)\,ds}\right)
\end{equation}
and $r_{\lambda}^{-}\in C([0,T];H^1_0(\Omega))\cap C^1([0,T];L^2(\Omega))$ satisfies
\begin{equation}\label{eq: condition on  r^(-) }  
   r_{\lambda}^{-}(T) = \partial_{t}r_{\lambda}^{-}(T) =  0   \text{ in } \Omega.
\end{equation}
Moreover, there exists a constant $C{(\Omega, T, \|\Phi\|_{L^{\infty}(\Omega)},\|A\|_{W^{2,\infty}(\Omega)}})  > 0$ such that
\begin{equation}\label{Energy Estimate  for r^(-)}
\begin{split}
 &\|\partial_{t}  r_{\lambda}^{-}\|_{L^{\infty}(0,T;L^{2}(\Omega))} +\lambda \|r_{\lambda}^{-}\|_{L^{\infty}(0,T;L^{2}(\Omega))} + \|\nabla r_{\lambda}^{-}\|_{L^{\infty}(0,T;L^{2}(\Omega))} \\
 &\quad \leq C \|\varphi\|_{H^{3}(\R^{2})}\|h\|_{H^{2}(\R)}.
 \end{split}
\end{equation}
\end{lemma}

\begin{proof}

 As in the proof of Lemma \ref{lemma: geom opt sol +}, it is enough to prove the existence of a function $r_{\lambda}^{-}\in C([0,T];H^1(\Omega))\cap C^1([0,T];L^2(\Omega))$ solving  
\begin{equation}
    \label{eq: PDE for r minus}
\begin{cases}
L_{\mathcal{A},\Phi}^* r_{\lambda}^{-} = \tilde{g} \hspace{1.75cm} \text{in }  \Omega_{T}    , \\
r_{\lambda}^{-} = 0 \hspace{2.5cm} \text{on } \partial \Omega_{T}, \\
r_{\lambda}^{-}(T) = \partial_{t}r_{\lambda}^{-}(T) = 0 \quad \text{in } \Omega,
\end{cases}
\end{equation}
and satisfying the estimate \eqref{Energy Estimate  for r^(-)}, where  $\Tilde{g}\colon\Omega_T\to \C$ is given by
\begin{equation}
\label{eq: def of tilde(g)}
    \tilde{g}(X,t)= -L^{*}_{\mathcal{A},\Phi}\big[ \varphi(x + t\theta)h(y)A^{-}(X,t)e^{\im\lambda( x \cdot \theta + t )}\big]\in L^2(\Omega_T).
\end{equation}
Next, let us observe that using the notation $\mathcal{A}_r=(-A_0,\im A)$, we see that
\begin{equation}
\label{eq: potential -}
    A^{-}(X,t)=A_r^+(X,t),
\end{equation}
where $A_r^+$ is given by \eqref{eq: form of A^(+)} with $\mathcal{A}$ replaced by $\mathcal{A}_r$. Thus, by \eqref{eq: potential -}, \eqref{eq: adjoint op} and Claim \ref{claim r plus inhom} it is immediate that we can write
\begin{equation}
\label{eq: equiv form of tilde(g)}
\begin{split}
    \tilde{g}(X,t) =-e^{\im\lambda( x\cdot \theta + t )}\tilde{g}_{0}(X,t),
\end{split}
\end{equation}
where $\tilde{g}_0$ is given by
\begin{equation}
\label{eq: def of tilde(g)0}
    \tilde{g}_0(X,t)\vcentcolon =L^{*}_{\mathcal{A},\Phi}\left(\varphi(x + t\theta)h(y)A^{-}(X,t)\right)\in L^2(\Omega_T).
\end{equation}
On the other hand, if $F\in L^2(\Omega_T)$, then $v$ solves
\[
\begin{cases}
L^{*}_{\mathcal{A}, \Phi}v = F &\text{ in }  \Omega_{T}    , \\
v = 0  &\text{ on } \partial \Omega_{T}, \\
v(T) = \partial_{t} v(T) = 0 & \text{ in } \Omega,
\end{cases}
\]
if and only if $u(X,t)=v(X,T-t)$ solves
\begin{equation}
\label{eq: time reversed eq}
\begin{cases}
L_{\mathcal{A}, \Phi}u = F^* &\text{ in }  \Omega_{T}    , \\
u= 0  &\text{ on } \partial \Omega_{T}, \\
u(0)= \partial_{t} u(0)= 0 &\text{ in } \Omega,
\end{cases}
\end{equation}
where $F^*(X,t)=F(X,T-t)$.
 
Recall from Step 1 of the proof of Theorem \ref{thm: Existence} that for any $F \in L^2(\Omega_{T})$ the problem \eqref{eq: time reversed eq} has a unique solution  $u \in C([0,T]; H^1_0(\Omega)) \cap C^1([0,T], L^2(\Omega))$ and it satisfies the energy estimate
\[
\| \partial_t u \|_{L^\infty(0,T; L^2(\Omega))} + \| \nabla u \|_{L^\infty(0,T; L^2(\Omega))} \leq C \| F \|_{L^2(\Omega_{T})}.
\]

Thus, by \eqref{eq: def of tilde(g)}, we immediately get the existence of $r^{-}_\lambda\in C([0,T];H^1_0(\Omega))\cap C^1([0,T];L^2(\Omega))$ solving \eqref{eq: PDE for r minus}. Next, we introduce the time integral transform 
\begin{equation}
\label{eq: time integral r minus}
  R_{\lambda}^-(X,t)  = \int_t^T r_{\lambda}^-(X,s)  ds
\end{equation}
of $r^-_\lambda$. Let us note that one has
\[
    R^{-}_\lambda(X,T-t)=\int_0^t r^{-}_\lambda(X,T-\tau)\,d\tau
\]
and by the above equivalence the function $r_\lambda^{-}(X,T-t)$ solves \eqref{eq: time reversed eq} with $F=\widetilde{g}$. So, from Claim \ref{claim: properties of R plus}, we deduce that $R^-_\lambda$ has the following properties:
    \begin{enumerate}[(i)]
        \item\label{PDE for time integral r minus} $R^-_\lambda$ solves the PDE 
\begin{equation}
\label{eq: PDE for R minus}
\begin{split}
    \begin{cases}
 L^{*}_{\mathcal{A},\Phi}R^-_\lambda  = \int_{t}^{T}\tilde{g}(\tau)\, d\tau & \text{in } \Omega_T,\\
R_{\lambda}^{-} = 0 & \text{on } (\partial\Omega)_T,\\
R_{\lambda}^{-}(T)  = \partial_{t}R_{\lambda}^{-}(T)=0 &  \text{in } \Omega
\end{cases}
\end{split}
\end{equation}
 \item\label{estimate for source of R minus} and the function $F_2^\lambda\vcentcolon =\int_{t}^{T}\tilde{g}(\tau)\, d\tau$ satisfies the estimate
        \begin{equation} 
\label{eq: estimate F2}
    \norm{F_{2}^{\lambda}}_{L^{2}(\Omega_T)} \leq  \dfrac{C}{\lambda}\norm{\varphi}_{H^{3}(\R^2)}  \norm{h}_{H^{2}(\R)} .
\end{equation}
    \end{enumerate}

We can now finish the proof in exactly the same way as Lemma \ref{lemma: geom opt sol +}.

    
\end{proof}

\section{Stability estimate}
\label{sec: Stability estimate}

In this section, we finally establish the stability estimate asserted in Theorem \ref{main theorem}. For later convenience, we introduce here some notation:
\begin{itemize}
    \item The norm $\|\cdot\|_{D^{3/2}((\partial\Omega)_T)\to L^2((\partial\Omega)_T)}$ is denote by $\|\cdot\|_*$.
    \item For any $C\subset \R^2$ and $\rho>0$, we define $C^\rho=C+\overline{B_{\rho}(0,\R^2)}$.
    \item For two given vector potentials $\mathcal{A}^1, \mathcal{A}^2$ and external potentials $\Phi_1,\Phi_2$, we set $\mathcal{A}^{21}=\mathcal{A}^{2}-\mathcal{A}^{1}$ and $\Phi_{21}=\Phi_{2}-\Phi_{1}$.
    \item For any pair $(\mathcal{A},\Phi)\in L^{\infty}(\Omega)\times L^{\infty}(\Omega)$ we denote by $A^{\pm}(X,t)$ the functions
    \begin{equation}
    \label{eq: def A plus}
      A^{+}(X,t)=\exp\left(-\int_0^t(1,-\theta)\cdot \mathcal{A}(x+s\theta,y)\,ds\right)
    \end{equation}
    and 
    \begin{equation}
    \label{eq: def A minus}
        A^{-}(X,t)=\exp\left(\int_0^t(1,\theta)\cdot \mathcal{A}(x+s\theta,y)\,ds\right)
    \end{equation}
    (see eq. \eqref{eq: form of A^(+)} and \eqref{eq: form of A^(-)}). In particular, we also use these formulas for the difference of two potentials $\mathcal{A}^{21}=\mathcal{A}^2-\mathcal{A}^1$ and in this case the resulting functions in  \eqref{eq: def A plus} and  \eqref{eq: def A minus} are denoted by $A^{21,+}$ and $A^{21,-}$, respectively.
    \item If $v\colon \R^2\to \R^2$ is a given vector field, then its rotation is defined by
    \begin{equation}
        \nabla\wedge v\vcentcolon = \partial_1 v_2-\partial_2 v_1.
    \end{equation}
\end{itemize}

\subsection{Pointwise estimate for the Radon transform of \texorpdfstring{$\mathcal{A}^{21}$}{A21}}
\label{subsec: pointwise estimate radon A21}

The main goal of this section is to provide a pointwise estimate of the form
\begin{equation}
\label{eq: pointwise radon trafo est}
    \left|\int_\R (1,-\theta)\cdot \mathcal{A}^{21}(x-s\theta,y)\,ds\right|\lesssim \lambda^\beta \|\Lambda_{\mathcal{A}^2,\Phi_2}-\Lambda_{\mathcal{A}^1,\Phi_1}\|_*+\lambda^{-\gamma}
\end{equation}
for some constants $\beta,\gamma>0$ and $\lambda\geq 1$ (Lemma \ref{lemma: integral estimate of A}). We call \eqref{eq: pointwise radon trafo est} pointwise estimate for the Radon transform of $\mathcal{A}^{21}$ as in the special case $x=\rho \theta^\perp$ the left hand side coincides with the Radon transform of $\mathcal{A}^{21}$ up to a scaling factor. 

To achieve this goal we first prove the following auxiliary lemma, which relies on an $H^{2,2}(\Omega_T)$ regularity result for the geometric optics solutions $u^{\pm}_\lambda$ that is established also in the proof below (Claim \ref{claim: regularity geometric optics}).

\begin{lemma}\label{intid2 2}
    Let $\Omega=\omega\times\R$ be an infinite waveguide with $\omega\subset\R^2$ being a smoothly bounded domain and $T>0$. Suppose that $(\mathcal{A}^{j},\Phi_j)\in \mathscr{A}(R)$, $j=1,2$, for some $R>0$. Then there exists a constant $C=C(\omega,T,R) > 0$ such that for any $\theta \in \mathbb{S}^1$, $h\in H^2(\R;\R)$ and $\varphi \in H^3(\R^2;\R)\cap C^{1,1}((\partial\omega)^T)$ satisfying $\varphi=\nabla_x \varphi =0 $ on $\overline{\omega}\cup(\overline{\omega}+T\theta)$, there holds
  \small{\begin{align}
  \label{eq: integral estimate}
      & \abs{\int_{\R^{3}_T}(1,-\theta)\cdot \mathcal{A}^{21}(x - t\theta,y){\varphi^{2}(x)h^{2}(y)\exp\left(-\int_0^t (1,-\theta)\cdot\mathcal{A}^{21}(x-\rho\theta,y)\,d\rho\right)\,dXdt}} \\ 
      &\quad\leq C ( \lambda^{2} \norm{\Lambda_{\mathcal{A}^{2},\Phi_{2}}-\Lambda_{\mathcal{A}^{1},\Phi_{1}}}_*+ \lambda^{-1})\norm{\varphi}^{2}_{H^{3}(\R^{2})}\|h\|^{2}_{H^{2}(\R)}
  \end{align}}
for all $\lambda\geq 1$. 
\end{lemma}

\begin{proof}
    For any $\lambda>0$, let us denote by $u^{+}_{2,\lambda}$, $u^{-}_{1,\lambda}$ the geometric optics solutions constructed in Lemma \ref{lemma: geom opt sol +} corresponding to the potentials $\mathcal{A}=\mathcal{A}^2$, $\Phi=\Phi_2$ and Lemma \ref{lemma: geom opt sol -}  corresponding to the potentials $\mathcal{A}=\mathcal{A}^1$, $\Phi=\Phi_1$, respectively. We use below the following assertion.
\begin{claim}
\label{claim: regularity geometric optics}
    For any $\lambda>0$, one has $u^{+}_{2,\lambda},u^{-}_{1,\lambda}\in H^{2,2}(\Omega_T)$ and thus they are strong solutions of the corresponding PDEs.
\end{claim}
\begin{proof}
    We only give the proof of the assertion for $u^{+}_{2,\lambda}$ as the one for $u^{-}_{1,\lambda}$ is similar. First of all, notice that we have
    \[
        \varphi(x+t\theta)h(y)A^{2,+}(X,t)e^{\im \lambda(x\cdot\theta +t)}\in H^{2,2}(\Omega_T)
    \]
    and so it is enough to prove that $r_{2,\lambda}^+\in H^{2,2}(\Omega_T)$. The later function solves
    \begin{equation}
    \label{eq: r plus int id}
    \begin{cases}
L_{\mathcal{A}^2,\Phi_2}r_{2,\lambda}^{+} = g & \text{in } \Omega_T,\\
r_{2,\lambda}^{+} = 0 & \text{on } (\partial\Omega)_T,\\
r_{2,\lambda}^{+}(0)  = \partial_{t}r_{\lambda}^{+}(0)=0 &  \text{in } \Omega,
\end{cases}
    \end{equation}
    where 
    \[
        g(X,t)= -L_{\mathcal{A}^2,\Phi_2}\big[ \varphi(x + t\theta)h(y)A^{2,+}(X,t)e^{\im\lambda( x \cdot \theta + t )}\big]\in H^1(0,T;L^2(\Omega)).
    \] 
    Let us introduce the function
    \[
        \mathcal{R}_{2,\lambda}^+\vcentcolon = e^{A^2_0 t}r_{2,\lambda}^+\in C([0,T];H^1_0(\Omega))\cap C^1([0,T];L^2(\Omega))
    \]    and observe that by our regularity assumptions it is enough to show  $\mathcal{R}_{2,\lambda}^+\in H^{2,2}(\Omega_T)$. As explained below, the function $\mathcal{R}_{2,\lambda}^+$ solves the problem
    \begin{equation}
    \label{eq: PDE for mathcal R plus}
     \begin{cases}
\mathcal{L}_{\mathcal{A}^2,\Phi_2}\mathcal{R}_{2,\lambda}^{+} = G & \text{in } \Omega_T,\\
\mathcal{R}_{2,\lambda}^{+} = 0 & \text{on } (\partial\Omega)_T,\\
\mathcal{R}_{2,\lambda}^{+}(0)  = \partial_{t}\mathcal{R}_{2,\lambda}^{+}(0)=0 &  \text{in } \Omega,
\end{cases}
    \end{equation}
    where
    \begin{equation}
    \label{eq: Help diff op}
    \begin{split}
        \mathcal{L}_{\mathcal{A}^2,\Phi_2}&=\partial_t^2-(\nabla_x+\im A^2)^2-\partial_y^2+2t\nabla_X A_0^2\cdot\nabla_X\\
        &\quad+t(\Delta_X A_0^2)-t^2 |\nabla_X A_0^2|^2  +2\im tA^2\cdot\nabla_x A_0^2+\Phi_2
    \end{split}
    \end{equation}
    and $G=e^{A^2_0 t}g$. In fact, if the function $r_{2,\lambda}^+$ is sufficiently smooth, then we may compute
    \begin{equation}
    \label{eq: time der identity}
    \begin{split}
    \partial_t^2\mathcal{R}_{2,\lambda}^+=e^{A^2_0 t}(\partial_t +A^2_0)^2 r_{2,\lambda}^+,
    \end{split}
    \end{equation}
    \begin{equation}
    \label{eq: gradient of mathcal R +}
    \begin{split}
         \nabla_x \mathcal{R}_{2,\lambda}^{+}&=t(\nabla_x A_0^2) \mathcal{R}_{2,\lambda}^{+}+e^{A_0^2 t}\nabla_x r_{2,\lambda}^+,
    \end{split}
    \end{equation}
    \begin{equation}
    \label{eq: y der of mathcal R +}
    \begin{split}
         \partial_y \mathcal{R}_{2,\lambda}^{+}&=t(\partial_y A_0^2) \mathcal{R}_{2,\lambda}^{+}+e^{A_0^2 t}\partial_y r_{2,\lambda}^+,
    \end{split}
    \end{equation}
    \begin{equation}
     \label{eq: laplacian of mathcal R +}
    \begin{split}
        \Delta_x  \mathcal{R}_{2,\lambda}^{+}&= t(\Delta_x A_0^2)\mathcal{R}_{2,\lambda}^+ +t \nabla_x A_0^2\cdot\nabla_x\mathcal{R}_{2,\lambda}^+ + te^{A_0^2 t}\nabla_x A_0^2\cdot\nabla_x r_{2,\lambda}^{+} +e^{A_0^2 t}\Delta_x r_{2,\lambda}^{+}\\
        &\overset{\eqref{eq: gradient of mathcal R +}}{=}t(\Delta_x A_0^2)\mathcal{R}_{2,\lambda}^+ -t^2 |\nabla_x A_0^2|^2 \mathcal{R}_{2,\lambda}^+ +2t \nabla_x A_0^2\cdot\nabla_x\mathcal{R}_{2,\lambda}^+ +e^{A_0^2 t}\Delta_x r_{2,\lambda}^{+}
    \end{split}
    \end{equation}
    and
    \begin{equation}
     \label{eq: y laplacian of mathcal R +}
    \begin{split}
        \partial_y^2  \mathcal{R}_{2,\lambda}^{+}&= t(\partial_y^2 A_0^2)\mathcal{R}_{2,\lambda}^+ -t^2 (\partial_y A_0^2)^2 \mathcal{R}_{2,\lambda}^+ +2t \partial_y A_0^2\cdot\partial_y\mathcal{R}_{2,\lambda}^+ +e^{A_0^2 t}\partial_y^2 r_{2,\lambda}^{+}
    \end{split}
    \end{equation}
    By \eqref{eq: gradient of mathcal R +}, \eqref{eq: laplacian of mathcal R +} and
    \[
        (\nabla_x+\im A^2)^2=\Delta_x +2\im A^2\cdot \nabla_x + \im \Div_x A^2-|A^2|^2,
    \]
    we get
    \begin{equation}
    \label{eq: magnetic part}
    \begin{split}
          (\nabla_x+\im A^2)^2 \mathcal{R}_{2,\lambda}^{+}&=e^{A_0^2 t}(\nabla_x+\im A^2)^2 r_{2,\lambda}^+ +2t\nabla_x A_0^2\cdot\nabla_x \mathcal{R}_{2,\lambda}^+\\
          &\quad+(t(\Delta_x A_0^2)-t^2 |\nabla_x A_0^2|^2  +2\im tA^2\cdot\nabla_x A_0^2)\mathcal{R}_{2,\lambda}^+.
    \end{split}
    \end{equation}
    Using \eqref{eq: Help diff op}, \eqref{eq: time der identity}, \eqref{eq: magnetic part}, \eqref{eq: y laplacian of mathcal R +} and \eqref{eq: r plus int id}, we deduce that $\mathcal{R}_{2,\lambda}^+$ satisfies
    \[
    \begin{split}
        \mathcal{L}_{\mathcal{A}^2,\Phi_2}\mathcal{R}_{2,\lambda}^+&=e^{A_0^2 t}L_{\mathcal{A}^2,\Phi_2}r_{2,\lambda}^+=G.
    \end{split}
    \]
    As $\mathcal{R}_{2,\lambda}^+(X,t)=0$ for $(X,t)\in (\partial\Omega)_T\cup (\Omega\times \{0\})$ and $\partial_t \mathcal{R}_{2,\lambda}^+(X,0)=0$ for $X\in \Omega$, $\mathcal{R}_{2,\lambda}^+$ indeed solves \eqref{eq: PDE for mathcal R plus} as asserted.
    
    The previous formal computations can be made rigorous by doing the previous calculations on the level of test functions and using integration by parts and approximation arguments as in the proof of Claim \ref{claim: properties of R plus}.
    
    Now, recalling our smoothness assumptions on the potentials $(\mathcal{A}^j,\Phi_j)$, we may generalize \cite[Theorem 2.1 and Lemma 2.1 in Chapter 4]{lions2012nonVol2} to our setting and conclude that $\mathcal{R}^{+}_{2,\lambda}\in H^{2,2}(\Omega_T)$ (see, for example, \cite[Section 9.6]{brezis2011functional} for the necessary elliptic regularity results). As pointed out above this ensures $u_{2,\lambda}^{+}\in H^{2,2}(\Omega_T)$.
\end{proof}
Moreover, let us set $f\vcentcolon =u^+_{2,\lambda}|_{(\partial\Omega)_T}$ and define $u_{1,\lambda}$ to be the unique solution of the problem
\begin{equation}
\label{eq: sol u1}
     \begin{cases}
     L_{\mathcal{A}^{1},\Phi_{1}}u = 0 & \text{ in } \Omega_T,\\
        u =f & \text{ on } (\partial\Omega)_T,\\
          u(0) = 0,\, \partial_{t}u(0) = 0 & \text{ on } \Omega.
    \end{cases}
\end{equation}
Indeed, the existence of $u_{1,\lambda}$ follows from Theorem \ref{thm: Existence} and the fact that
\[
    f(X,t)=\varphi(x+t\theta)h(y)A^{2,+}(X,t)e^{\im \lambda(x\cdot \theta+t)}\in D^{3/2}((\partial\Omega)_T)
\]
(see Lemma \ref{lemma: geom opt sol +}). The asserted inclusion can be seen as follows. First of all, the regularity $H^{3/2,3/2}((\partial\Omega)_T)$ follows by the assumptions on $\varphi,h$ and $\mathcal{A}$. Secondly, the function $f$ vanishes for $t=0$ by Lemma \ref{lemma: geom opt sol +} and $\varphi=0$ on $\partial\omega$. Thirdly, one has
\[
\begin{split}
     \partial_tf(X,t) &= \theta\cdot\nabla_x \varphi(x+t\theta)h(y)A^{2,+}(X,t)e^{\im\lambda(x\cdot\theta+t)}\\
     &\quad - \varphi(x+t\theta)h(y)(1,-\theta)\cdot \mathcal{A}^2(x+t\theta,y)A^{2,+}(X,t)e^{\im\lambda(x\cdot\theta+t)}\\
     &\quad +\im\lambda f(X,t)
\end{split}
\]
and hence we may estimate
\[
\begin{split}
   & \int_{(\partial\Omega)_T}\frac{|\partial_t f|^2}{t}\,dSdt\leq C \int_{(\partial\Omega)_T}\frac{|\nabla_x \varphi(x+t\theta)|^2+|\varphi(x+t\theta)|^2}{t}|h(y)|^2\,dSdt.
\end{split}
\]
As $\varphi=\nabla_x\varphi=0$ on $\partial\omega$ and $\varphi\in C^{1,1}((\partial\omega)^T)$, we see that the right hand side is finite and therefore $\partial_t f\in L^2((\partial\Omega)_T;t^{-1}dSdt)$. In a similar way, one can show that $\partial_\tau f, \partial_y f\in L^2((\partial\Omega)_T;t^{-1}dSdt)$. 

Furthermore, we define $u_\lambda\vcentcolon = u_{2,\lambda}^{+}-u_{1,\lambda}$. Using the alternative definition \eqref{eq: equivalent form L AP}, \eqref{eq: sol u1} and $L_{\mathcal{A}^2,\Phi_2}u_{2,\lambda}^+=0$ in $\Omega_T$, we may compute
\[
\begin{split}
    L_{\mathcal{A}^{1}, \Phi_{1}}u_{\lambda}&= L_{\mathcal{A}^{1}, \Phi_{1}}u_{2,\lambda}^{+}\\
    &=-(L_{\mathcal{A}^{2}, \Phi_{2}}-  L_{\mathcal{A}^{1}, \Phi_{1}})u_{2,\lambda}^{+}\\
    & = -(2\mathcal{A}^{21}\cdot(\partial_{t},-\nabla_{x}) + \tilde{\Phi}^{21})u_{2,\lambda}^{+}\in L^2(\Omega_T),
\end{split}
\]
where 
\begin{equation}
\label{eq: def phi int id 1}
\begin{split}
    \tilde{\Phi}^{21}\vcentcolon =-\im\Div_x A^{21} +|\mathcal{A}^2|^2-|\mathcal{A}^1|^2+\Phi_{21}\in L^{\infty}(\Omega).
\end{split}
\end{equation}
Therefore $u_\lambda$ solves 
\begin{equation}\label{IBP}
     \begin{cases}
     L_{\mathcal{A}^{1},\Phi_{1}}u =  -(2\mathcal{A}^{21}\cdot(\partial_{t},-\nabla_{x}) + \tilde{\Phi}^{21})u_{2,\lambda}^{+} & \text{ in } \Omega_T,\\
        u = 0 & \text{ on } (\partial\Omega)_T,\\
          u(0) = 0,\, \partial_{t}u(0) = 0 & \text{ on } \Omega.
    \end{cases}
\end{equation}
Note that for the initial conditions we used Lemma \ref{lemma: geom opt sol +} and the assumption $\varphi=\nabla_x\varphi=0$ on $\omega$.
On the one hand, using  \eqref{IBP}, $\partial_\nu u_\lambda\in L^2((\partial\Omega)_T)$, Lemma \ref{lemma: geom opt sol +} and \ref{lemma: geom opt sol -}, $\varphi=\nabla_x \varphi=0$ on $\overline{\omega}\cup(\overline{\omega}+T\theta)$, $\mathcal{A}^1,\mathcal{A}^2\in C^2_0(\Omega)$ and Claim \ref{claim: regularity geometric optics}, we may deduce using an integration by parts argument that there holds
\begin{equation}\label{inteqn}
\begin{split}
\int_{\Omega_{T}}\big(2\mathcal{A}^{21}\cdot(\partial_{t},-\nabla_{x})u_{2,\lambda}^{+} + \tilde{\Phi}^{21}u_{2,\lambda}^{+}\big)\overline{u_{1,\lambda}^{-}}\,dXdt 
&= \int_{(\partial\Omega)_T}{\partial_{\nu}u_\lambda\overline{u_{1,\lambda}^{-}}\,dSdt}. 
\end{split}
\end{equation}
On the other hand, noting that for any $\mathcal{A}=(A_0,\im A)$ and $(X,t)\in\Omega_T$ there holds
\[
    \overline{A^{-}(X,t)}=(-A)^+(X,t)=\exp\left(\int_0^t (1,-\theta)\cdot \mathcal{A}(x+s\theta,y)\,ds\right).
\]
 and by using Lemma \ref{lemma: geom opt sol +} and \ref{lemma: geom opt sol -}, we get
\begin{equation}
\label{eq: decomposition remainder integral identity 1}
\begin{split}
 [2\mathcal{A}^{21}\cdot(\partial_{t},-\nabla_{x})u_{2,\lambda}^{+}]\overline{u_{1,\lambda}^{-}}&=2\im\lambda\mathcal{A}^{21}\cdot(1,-\theta)\varphi^{2}(x+t\theta)h^{2}(y)  A^{21,+}(X,t) + P_{\lambda},
\end{split}
\end{equation}
where $P_\lambda$ satisfies
\begin{equation}
\label{eq: remainder estimate lemma int identity 1}
    \norm{ P_{\lambda}}_{L^{1}(\Omega_{T})} \leq C\|\mathcal{A}^{21}\|_{L^{\infty}(\Omega)}\norm{\varphi}^{2}_{H^{3}{(\R^{2})}}\norm{h}^{2}_{H^{2}{(\R)}}
\end{equation}
for some $C>0$ only depending on $\omega,T,R$.
Using \eqref{eq: def phi int id 1} and Lemmas \ref{lemma: geom opt sol +},\ref{lemma: geom opt sol -}, we obtain
\begin{equation}\label{1.1}
\begin{split}
   \bigg| \int_{\Omega_{T}}{\Tilde{\Phi}^{21}u_{2,\lambda}^+\overline{u_{1,\lambda}^{-}}\,dXdt}\bigg| & \leq C \|u_{2,\lambda}^+\|_{L^2(\Omega_T)}\|u_{1,\lambda}^{-}\|_{L^2(\Omega_T)}\\
   &\leq C \norm{\varphi}^{2}_{H^{3}{(\R^{2})}}\norm{h}^{2}_{H^{2}{(\R)}}
\end{split}
\end{equation}
where the constant $C>0$ only depends on $\Omega,T,R$.
Furthermore, there holds
\begin{equation}\label{1.11}
\begin{split}
&\bigg| \int_{(\partial\Omega)_T} {\partial_{\nu}u_{\lambda}\overline{u_{1,\lambda}^{-}}\,dSdt}  \bigg|  =\bigg|\int_{(\partial\Omega)_T}\left[\big( \Lambda_{\mathcal{A}^{2},\Phi_{2}}-\Lambda_{\mathcal{A}^{1},\Phi_{1}} \big)f\right]\overline{u_{1,\lambda}^{-}} \,dSdt \bigg| \\
& \quad\leq C\norm{\Lambda_{\mathcal{A}^{2},\Phi_{2}}-\Lambda_{\mathcal{A}^{1},\Phi_{1}}}_{*}\norm{f}_{D^{3/2}((\partial\Omega)_T)}\norm{u_{1,\lambda}^{-}}_{L^{2}{(\partial\Omega)_T}}  \\
& \quad\leq C \norm{\Lambda_{\mathcal{A}^{2},\Phi_{2}}-\Lambda_{\mathcal{A}^{1},\Phi_{1}}}_{*}\norm{\varphi(x+t\theta)h(y)A^{2,+}e^{\im \lambda(x\cdot \theta+t)}}_{H^{2}(\Omega_{T})}\\
&\quad\quad\times\norm{\varphi(x+t\theta)h(y)A^{1,-}e^{\im \lambda(x\cdot \theta+t)}}_{H^{1}(\Omega_{T})} \\
&\quad \leq C\lambda^{3}  \norm{\Lambda_{\mathcal{A}^{2},\Phi_{2}}-\Lambda_{\mathcal{A}^{1},\Phi_{1}}}_{*}  \norm{\varphi}^{2}_{H^{3}{(\R^{2})}}\norm{h}^{2}_{H^{2}{(\R)}}.
\end{split}
\end{equation}
for $\lambda \geq 1$. The first equality follows from the definition $u_\lambda$ and of the DN maps, in the second inequality we are using that the DN maps $\Lambda_{\mathcal{A}^j,\Phi_j}$, $j=1,2$, are bounded linear operators from $D^{3/2}((\partial\Omega)_T)$ to $L^2((\partial\Omega)_T)$ and in the third line we applied the the trace theorem and \cite[Chapter 4, Theorem 2.1]{lions2012nonVol2}.
By rearranging the terms in  \eqref{inteqn}, we get with the help of \eqref{eq: decomposition remainder integral identity 1}, \eqref{1.1} and \eqref{1.11} the bound
\small{
\begin{align}
\label{eq: prelim int id 1}
& \quad\,\,\,\,\left|\int_{\Omega_{T}} (1,-\theta)\cdot\mathcal{A}^{21}(X)\varphi^{2}(x+t\theta)h^{2}(y) \exp\left( -\int_{0}^{t}{(1,-\theta)\cdot \mathcal{A}^{21}(x+s\theta,y)\,ds} \right) \,dXdt\right| \\ 
& \quad \leq C \bigg( \lambda^{2}\norm{\Lambda_{\mathcal{A}^{2},\Phi_{2}}-\Lambda_{\mathcal{A}^{1},\Phi_{1}}} + \dfrac{1}{\lambda} \bigg)  \norm{\varphi}^{2}_{H^{3}{(\R^{2})}}\norm{h}^{2}_{H^{2}{(\R)}}.
\end{align}}
By the change of variables $x'=x+t\theta$ and $\rho=t-s$, we obtain
\small{\begin{align}
    &\int_{\Omega_{T}} (1,-\theta)\cdot\mathcal{A}^{21}(X)\varphi^{2}(x+t\theta)h^{2}(y) \exp\left( -\int_{0}^{t}{(1,-\theta)\cdot \mathcal{A}^{21}(x+s\theta,y)\,ds} \right) \,dXdt\\
    &\quad = \int_0^T\int_{\R}\int_{\omega+t\theta}(1,-\theta)\cdot\mathcal{A}^{21}(x'-t\theta,y)\varphi^{2}(x')h^{2}(y)\\
    &\quad \quad\times\exp\left( -\int_{0}^{t}{(1,-\theta)\cdot \mathcal{A}^{21}(x'-\rho\theta,y)\,d\rho} \right) \,dx'dydt.
\end{align}}
As $\mathcal{A}^j\in C^2_0(\Omega)$ for $j=1,2$ implies $\mathcal{A}^{21}(x-t\theta,y)=0$ for $x\notin \omega+t\theta$ and $y\in \R$, we may write
\small{
\begin{align}
    & \int_{\Omega_{T}} (1,-\theta)\cdot\mathcal{A}^{21}(X)\varphi^{2}(x+t\theta)h^{2}(y) \exp\left( -\int_{0}^{t}{(1,-\theta)\cdot \mathcal{A}^{21}(x+s\theta,y)\,ds} \right) \,dXdt\\
    &\quad = \int_{\R^3_T}(1,-\theta)\cdot\mathcal{A}^{21}(x-t\theta,y)\varphi^{2}(x)h^{2}(y)\exp\left( -\int_{0}^{t}{(1,-\theta)\cdot \mathcal{A}^{21}(x-\rho\theta,y)\,d\rho} \right) \,dXdt.
\end{align}}
Thus, from \eqref{eq: prelim int id 1} we deduce that there holds
 \small{\begin{align}
       &\abs{\int_{\R^3_T}(1,-\theta)\cdot\mathcal{A}^{21}(x-t\theta,y)\varphi^{2}(x)h^{2}(y) \exp\left( -\int_{0}^{t}{(1,-\theta)\cdot \mathcal{A}^{21}(x-\rho\theta,y)\,d\rho} \right) \,dXdt} \\ 
        &\quad\leq C \left( \lambda^{2} \norm{\Lambda_{\mathcal{A}^{2},\Phi_{2}}-\Lambda_{\mathcal{A}^{1},\Phi_{1}}}_* + \dfrac{1}{\lambda}\right)\norm{\varphi}^{2}_{H^{3}(\R^{2})}\|h\|^{2}_{H^{2}(\R)}
\end{align}}
for any $\lambda \geq 1$.
     This completes the proof of the lemma.
     \end{proof}
     
\begin{lemma}
\label{lemma: integral estimate of A}
    Let $\Omega=\omega\times\R$ be an infinite waveguide with $\omega\subset\R^2$ being a smoothly bounded domain, $T>\text{diam}(\omega)$ and let
    \begin{equation}
    \label{eq: def omega alpha}
        \omega_\alpha \vcentcolon = \{x\in \R^2\setminus\overline{\omega}\,;\,\dist(x,\omega)<\alpha\}
    \end{equation}
    for some $\alpha\in(0,\min\{1, (T-\text{diam}(\omega))/3\})$.
    Suppose that $(\mathcal{A}^{j},\Phi_j)\in \mathscr{A}(R)$, $j=1,2$, for some $R >0$ and there holds
    \begin{equation}
    \label{eq: smallness condition on vector potential}
        \|\mathcal{A}^1-\mathcal{A}^2\|_{L^{\infty}(\Omega)}<\frac{2\pi}{T}.
    \end{equation}
    Then there exists $C>0$ only depending on $\omega,T,R$ such that for all $0<\gamma\leq 1/11$ and $\beta>0$ satisfying $\beta\geq 2+10\gamma$, we have
    \[
      \left|\int_{\R} (1,-\theta)\cdot\mathcal{A}^{21}(x-s \theta,y)\, d s\right| \leq C\left(\lambda^{\beta}\|\Lambda_{\mathcal{A}^{2}, \Phi_{2}}-\Lambda_{\mathcal{A}^{1},\Phi_{1}}\|_* +\frac{1}{\lambda^{\gamma}}\right)
    \]
    for all $\theta \in \mathbb{S}^{1}$, $X=(x,y)\in \R^3$ and $\lambda\geq 1$.
\end{lemma}

\begin{proof}
    First, we observe that for all $\varphi\in L^2(\R^2;\R)$ and $h\in L^2(\R;\R)$, one has
\begin{equation}
\label{eq: integration by parts formula for pointwise estimate}
\begin{split}
&\int_{\R^{3}_T}(1,-\theta)\cdot\mathcal{A}^{21}(x - t\theta,y){\varphi^{2}(x)h^{2}(y)e^{-\int_{0}^{t}{(1,-\theta)\cdot\mathcal{A}^{21}(x - s\theta,y)\,ds}}\,dxdydt} \\
 &= -\int_{\R^3_T} \varphi^{2}(x)h^{2}(y)\partial_{t} e^{-\int_{0}^{t}{(1,-\theta)\cdot\mathcal{A}^{21}(x - s\theta,y)\,ds}}\,dxdydt \\ 
 & = - \int_{\R^{3}} \varphi^{2}(x)h^{2}(y)\left[ e^{-\int_{0}^{T}{(1,-\theta)\cdot\mathcal{A}^{21}(x - s\theta,y)\,ds}} -1\right]\,dxdy.
 \end{split}
\end{equation}
Next, for a given point $X_0=(x_0,y_0)\in \omega_\alpha\times \R$ we denote by $(h_\eps)_{\eps>0}\subset C_c^{\infty}(\R;\R)$, $(\varphi_\eps)_{\eps>0}\subset C_c^{\infty}(\R^2;\R)$ the functions
\[
    h_\eps(y)=\eps^{-1/2}h\left(\frac{y-y_0}{\eps}\right)\text{ and } \varphi_\eps(x)=\eps^{-1}\varphi\left(\frac{x-x_0}{\eps}\right),
\]
where $h\in C_c^{\infty}((-1,1);\R)$ and $\varphi\in C_c^{\infty}(B_1(0);\R)$ are the square roots of the standard mollifiers in dimension one and two, respectively. Let us note that one has $\supp(\varphi_{\epsilon}) \subset \omega_{\alpha}$ for sufficiently small $\eps>0$ and 
\begin{equation}
\label{eq: support condition}
    (\omega_\alpha \pm S\theta)\cap \omega = \emptyset\text{ for all }S\geq T.
\end{equation}
The latter condition is a consequence of $ x\pm S\theta\notin \omega_\alpha\text{ for all }x\in\omega$, which in turn follows from $\dist(x\pm S\theta,\omega)\geq S-\text{diam}(\omega)\geq T-\text{diam}(\omega)>\alpha$. Moreover, by using $\int_{\R^2}\varphi^2_\eps dx=\int_\R h_\eps^2dy=1$, we get
\begin{equation}
\label{eq: auxiliary estimate lemma 5.1}
\begin{split}
& \abs{e^{-\int_{0}^{T}{(1,-\theta)\cdot \mathcal{A}^{21}(x_0-s \theta,y_0)\,ds}}-1}\\
&= \abs{\int_{\R^{3}}{\varphi_{\epsilon}^{2}(x)h_{\eps}^{2}(y)\left[e^{-\int_{0}^{T}{(1,-\theta)\cdot \mathcal{A}^{21}(x_0-s \theta,y_0)\,ds}}-1\right]\,dxdy}}\\
&\leq \abs{\int_{\R^{3}}{\varphi_{\epsilon}^{2}(x)h_{\eps}^{2}(y)\left[e^{- \int_{0}^{T}{(1,-\theta)\cdot \mathcal{A}^{21}(x_0-s \theta,y_0)\,ds}} -e^{-\int_{0}^{T}{(1,-\theta)\cdot \mathcal{A}^{21}(x-s \theta,y)\,ds}}\right]\,dxdy}} \\
& \quad + \abs{\int_{\R^{3}}{\varphi_{\epsilon}^{2}(x)h_{\eps}^{2}(y)\left[e^{-\int_{0}^{T}{(1,-\theta)\cdot \mathcal{A}^{21}(x-s \theta,y)\,ds}}-1\right]\,dxdy}}.  
\end{split}
\end{equation}
If $a,b\in \C$ satisfy $|a|,|b|\leq M$ for some $M>0$, then we may estimate 
\begin{equation}
\label{eq: simple estimate exponential}
    |e^a-e^b|\leq e^{M}|a-b|
\end{equation}
and for $M<2\pi$ we have
\begin{equation}
\label{eq: simple estimate exponential 2}
    |a|\leq \frac{M}{1-e^{-M}} |e^a-1|.
\end{equation}
The estimate \eqref{eq: simple estimate exponential} is a simple consequence of the fundamental theorem of calculus. For the estimate \eqref{eq: simple estimate exponential 2}, we use $(e^z-1)/z \neq 0$ on $B_{2\pi}$ together with the minimum principle to obtain
\[
    \left|\frac{e^a-1}{a}\right|\geq \frac{\min_{|z|=M}|e^z-1|}{M}=\frac{1-e^{-M}}{M}. 
\]
Thus, by using the uniform bound
\begin{equation}
\label{eq: uniform bound radon trafo vector pot}
    \left|\int_0^T (1,-\theta)\cdot \mathcal{A}^{21}(x-s \theta,y) \,ds\right|\leq T\|\mathcal{A}^{21}\|_{L^{\infty}(\Omega)}\leq 2TR
\end{equation}
for all $(x,y)\in\R^3$, \eqref{eq: simple estimate exponential}, $\mathcal{A}^{21}\in C^2_0(\Omega)\cap W^{2,\infty}(\Omega)$ and the fundamental theorem of calculus, we have
\begin{equation}
\label{eq: estimate difference of exponentials}
    \begin{split}
        & \left[e^{- \int_{0}^{T}{(1,-\theta)\cdot \mathcal{A}^{21}(x_0-s \theta,y_0)\,ds}} -e^{-\int_{0}^{T}{(1,-\theta)\cdot \mathcal{A}^{21}(x-s \theta,y)\,ds}}\right]\\
        &\leq C \abs{\int_{0}^{T} \bigg((1,-\theta)\cdot \mathcal{A}^{21}(x_0-s \theta,y_0)-(1,-\theta)\cdot \mathcal{A}^{21}(x-s \theta,y)\bigg)\,ds}\\
        &\leq C|X-X_0|
    \end{split}
\end{equation}
for all $X,X_0\in\R^3$, where $C>0$ only depends on $T$, $\omega$ and $R$. 
Next, let us choose $\eps>0$ so small that $\supp\varphi_\eps\Subset \omega_\alpha$, then by taking into account \eqref{eq: support condition} we deduce from Lemma \ref{intid2 2} with $\varphi=\varphi_{\epsilon}$, $h=h_\eps$, \eqref{eq: integration by parts formula for pointwise estimate}, \eqref{eq: auxiliary estimate lemma 5.1} and \eqref{eq: estimate difference of exponentials} the estimate
\begin{equation}
\begin{split}
     &\abs{ e^{- \int_{0}^{T} (1,-\theta)\cdot \mathcal{A}^{21}(x_0-s \theta,y_0)\,ds}-1}\leq C \int_{\R^{3}} \varphi_{\epsilon}^2(x)h^{2}_{\eps}(y)|X-X_0| \,d xdy\\
     &\quad +C\left(\lambda^2\norm{\Lambda_{\mathcal{A}^2, \Phi_2}-\Lambda_{\mathcal{A}^{1}, \Phi_1}}_* +\frac{1}{\lambda}\right)\left\|\varphi_{\epsilon}\right\|_{H^3\left(\R^{2}\right)}^2\|h_{\eps}\|^{2}_{H^{2}(\R)}.
\end{split}
\end{equation}
Combining this with the bounds
\begin{equation}
\label{eq: bounds for mollifiers}
    \left\|h_{\epsilon}\right\|_{H^2\left(\R\right)} \leq \eps^{-2}\|h\|_{H^2(\R)}, \quad        \left\|\varphi_{\epsilon}\right\|_{H^3\left(\R^2\right)} \leq \eps^{-3}\|\varphi\|_{H^3(\R^2)}
\end{equation}
for $0<\eps\leq 1$ yields
\[
    \left|e^{- \int_0^T (1,-\theta)\cdot \mathcal{A}^{21}(x_0-s \theta,y_0 ) \,ds}-1\right| \leq C \eps + C\left(\lambda^2\left\|\Lambda_{\mathcal{A}^2, \Phi_2}- \Lambda_{\mathcal{A}^1, \Phi_1}\right\|_* +\frac{1}{\lambda}\right) \eps^{-10}
\]
for $0<\eps\leq 1$. Next, let us choose $\eps=\lambda^{-\gamma}$ for $0<\gamma\leq 1/11$ and let $\beta\geq 2+10\gamma$. Then for $\lambda\geq 1$, we have $0<\eps\leq 1$, $\lambda^2\eps^{-10}=\lambda^{2+10\gamma}\leq \lambda^\beta$ and
\[
    \eps+\frac{1}{\lambda \eps^{10}}=\lambda^{-\gamma}+\frac{1}{\lambda^{1-10\gamma}}\leq 2\lambda^{-\gamma},
\]
as $0<\gamma\leq 1/11$ is equivalent to $1-10\gamma \geq \gamma$. Thus, there holds
\[
    \abs{e^{-\int_{0}^{T}{(1,-\theta)\cdot \mathcal{A}^{21}(x_0-s \theta,y_0)\,ds}}-1} \leq C\left[\lambda^{\beta}\norm{\Lambda_{\mathcal{A}^2, \Phi_2}-\Lambda_{\mathcal{A}^1,\Phi_1}}_* + \frac{1}{\lambda^\gamma}\right].
\]
Now, using \eqref{eq: simple estimate exponential 2} with $M=T\|\mathcal{A}^{21}\|_{L^{\infty}(\Omega)}<2\pi$ (see \eqref{eq: smallness condition on vector potential} and \eqref{eq: uniform bound radon trafo vector pot}) together with the fact that $[0,\infty)\ni \rho\mapsto \frac{\rho}{1-e^{-\rho}}\in (0,\infty)$ is increasing and $M\leq 2TR$, we deduce that
\begin{equation}
\label{eq: almost final estimate for radon}
    \begin{split}
        \left|\int_0^T (1,-\theta)\cdot\mathcal{A}^{21}(x_0-s \theta,y_0) \,ds\right|& \leq \frac{M}{1-e^{-M}}\left| e^{- \int_0^T (1,-\theta)\cdot\mathcal{A}^{21}(x_0-s \theta,y_{0}, s) d s}-1\right|\\
        & \leq \frac{2TR}{1-e^{-2TR}}\left| e^{- \int_0^T (1,-\theta)\cdot\mathcal{A}^{21}(x_0-s \theta,y_{0}, s) d s}-1\right|\\
        &\leq C\left[\lambda^{\beta}\norm{\Lambda_{\mathcal{A}^2, \Phi_2}-\Lambda_{\mathcal{A}^1,\Phi_1}}_* + \frac{1}{\lambda^\gamma}\right]
    \end{split}
\end{equation}
for some $C>0$ only depending on $\omega, T$ and $R$.
As \eqref{eq: almost final estimate for radon} also holds for $-\theta$ in place of $\theta$, we get by a change of variables the estimate
\begin{equation}
\label{eq: almost final estimate for radon 3}
    \begin{split}
        \left|\int_{-T}^T (1,-\theta)\cdot\mathcal{A}^{21}(x_0-s \theta,y_0) \,ds\right|& \leq C\left[\lambda^{\beta}\norm{\Lambda_{\mathcal{A}^2, \Phi_2}-\Lambda_{\mathcal{A}^1,\Phi_1}}_* + \frac{1}{\lambda^\gamma}\right].
    \end{split}
\end{equation}
Let us observe that the conditions $x_0\in \omega_\alpha$, \eqref{eq: support condition}  and $\mathcal{A}^{21}\in C^2_0(\Omega)$ imply $\mathcal{A}^{21}(x_0-s\theta)=0$ for all $s\geq T$. Therefore, we deduce that 
\begin{equation}
\label{eq: almost final estimate for radon 4}
    \begin{split}
        \left|\int_\R (1,-\theta)\cdot\mathcal{A}^{21}(x_0-s \theta,y_0) \,ds\right|& \leq C\left[\lambda^{\beta}\norm{\Lambda_{\mathcal{A}^2, \Phi_2}-\Lambda_{\mathcal{A}^1,\Phi_1}}_* + \frac{1}{\lambda^\gamma}\right].
    \end{split}
\end{equation}
Next, we show that the above estimate holds for any $x_0 \in \R^{2}$. If for a given point $x_0\in \omega_\alpha^c$ there exists  $s_0\in\R$ such that $x_0-s_0\theta\in \omega_\alpha$, then we may use \eqref{eq: almost final estimate for radon 4} and a change of variables to get
    \[
        \begin{split}
            \left|\int_\R (1,-\theta)\cdot\mathcal{A}^{21}(x_0-s \theta,y_0) \,ds\right|&=\left|\int_\R (1,-\theta)\cdot\mathcal{A}^{21}(x_0-s_0\theta-(s-s_0) \theta,y_0) \,ds\right|\\
            &=\left|\int_\R (1,-\theta)\cdot\mathcal{A}^{21}(x_0-s_0\theta+\tau \theta,y_0) \,d\tau\right|\\
            &\leq C\left[\lambda^{\beta}\norm{\Lambda_{\mathcal{A}^2, \Phi_2}-\Lambda_{\mathcal{A}^1,\Phi_1}}_* + \frac{1}{\lambda^\gamma}\right].
        \end{split}
    \]
If there is no such $s_0$, then the line $x_0-s\theta$, $s\in\R$, never hits $\omega$ and so $\mathcal{A}^{21}(x_0-s\theta,y)=0$ for all $s\in \R$. Thus, we get in both cases the estimate \eqref{eq: almost final estimate for radon 4} and we can conclude the proof.
\end{proof}

\subsection{Stability estimate for vector potential \texorpdfstring{$\mathcal{A}$}{A}}
\label{subsec: stability for A}

In this section we infer from the pointwise estimate for the Radon transform of $\mathcal{A}^{21}$ the stability estimate for the vector potential $\mathcal{A}^{21}$ (Lemma \ref{lemma: stability estimate for vector potential}). Before giving this proof, we show in the next lemma that we can appropriately control the Fourier transform of $A_0^{21}$ and $\nabla\wedge\mathcal{A}^{21}$.

 \begin{lemma}
 \label{Lemma: estimates for fourier transform}
    Suppose that the assumptions of Lemma \ref{lemma: integral estimate of A} hold and let $\beta,\gamma>0$ be the constants from that lemma. There exists $C> 0$ only depending on $\omega,T,R$ such that one has
\begin{equation}
\label{eq: fourier transform estimates for A0 and A}
\begin{split}
      |\widehat{A_0^{21}}(\xi,y)|+ \big| \theta\cdot\widehat{A^{21}}(\xi,y) \big| &\leq C\left( \lambda^{\beta} \norm{\Lambda_{\mathcal{A}^{2},\Phi_{2}}- \Lambda_{\mathcal{A}^{1},\Phi_{1}} }_{*} + \dfrac{1}{\lambda^{\gamma}}\right)
\end{split}
\end{equation}
for all pairs $(\theta,\xi)\in \mathbb{S}^1\times \R^2$ with $\theta\perp \xi$ , $y\in \R$ and $\lambda\geq 1$, where the Fourier transform of $\mathcal{A}^{21}$ is only performed in the $x$ variable. Furthermore, we have
\begin{equation}
\label{eq: estimate sigma matrix}
  |\widehat{A_0^{21}}(\xi,y)|+\frac{|\widehat{\nabla\wedge A^{21}}(\xi,y)|}{|\xi|} \leq C\left(\lambda^{\beta} \norm{\Lambda_{\mathcal{A}^{2},\Phi_{2}}- \Lambda_{\mathcal{A}^{1},\Phi_{1}} }_{*} + \dfrac{1}{\lambda^{\gamma}}\right)
\end{equation}
for all $\xi\in\R^2$.
 \end{lemma}
 \begin{proof}
  Let $(\theta,\xi)\in \mathbb{S}^1\times \R^2$ satisfy $\theta\perp \xi$ and let us denote by $\theta^\perp$ the unit vector obtained from $\theta$ by a counterclockwise rotation of 90 degrees. By a change of variables, we have
    \begin{equation}
    \label{change of variables fourier of radon for A}
        \begin{split}
            &\int_{\R^2}e^{-\im \rho\theta^{\perp}\cdot \xi}(1,-\theta)\cdot\mathcal{A}^{21}(\rho\theta^\perp-s\theta,y)\,ds\,d\rho\\
            &\quad \overset{\xi\perp \theta}{=}\int_{\R^2}e^{-\im (\rho\theta^{\perp}-s\theta)\cdot \xi}(1,-\theta)\cdot\mathcal{A}^{21}(\rho\theta^\perp-s\theta,y)\,ds\,d\rho\\
          &\quad\overset{x=\rho\theta^{\perp}-s\theta}{=}\int_{\R^2}e^{-\im \rho x\cdot\xi}(1,-\theta)\cdot\mathcal{A}^{21}(x,y)\,dx\\ 
            &\quad = (1,-\theta)\cdot\widehat{\mathcal{A}^{21}}(\xi,y).
        \end{split}
    \end{equation}
    Therefore, using Lemma \ref{lemma: integral estimate of A}, $\mathcal{A}^{21}\in C^2_0(\Omega)$ and the boundedness of $\omega$, we get
    \begin{equation}
    \label{eq: estimate fourier of full vector potential}
    \begin{split}
       \big|  (1,-\theta)\cdot\widehat{\mathcal{A}^{21}}(\xi,y) \big| &\leq\int_{-r}^r\left|\int_{\R}
(1,-\theta)\cdot\mathcal{A}^{21}(\rho\theta^\perp-s\theta,y)\,ds\right|\,d\rho \\
         &\leq C \bigg( \lambda^{\beta} \norm{\Lambda_{\mathcal{A}^{2},\Phi_{2}}- \Lambda_{\mathcal{A}^{1},\Phi_{1}} }_{*} + \dfrac{1}{\lambda^{\gamma}} \bigg)
    \end{split}
    \end{equation}
    for some suitable $r>0$. Note that for deriving this estimate the only important point was that $\theta,\xi $ are perpendicular and thus we can replace $\theta$ by $-\theta$ while keeping the same $\xi$. This leads to
    \[ 
    \big|  (1,\theta)\cdot\widehat{\mathcal{A}^{21}}(\xi,y) \big| 
    \leq C \bigg( \lambda^{\beta} \norm{\Lambda_{\mathcal{A}^{2},\Phi_{2}}- \Lambda_{\mathcal{A}^{1},\Phi_{1}} }_{*} + \dfrac{1}{\lambda^{\gamma}} \bigg).
    \]
    By adding up these two estimates, we get
    \[
    \begin{split}
         |\widehat{A_0^{21}}(\xi,y)|&\leq C\left( \big|  (1,-\theta)\cdot\widehat{\mathcal{A}^{21}}(\xi,y) \big|+\big|  (1,\theta)\cdot\widehat{\mathcal{A}^{21}}(\xi,y) \big|\right)\\
         &\leq C \bigg( \lambda^{\beta} \norm{\Lambda_{\mathcal{A}^{2},\Phi_{2}}- \Lambda_{\mathcal{A}^{1},\Phi_{1}} }_{*} + \dfrac{1}{\lambda^{\gamma}} \bigg)
    \end{split}
    \]
    and by taking the difference one has
    \[
     \big| \theta\cdot\widehat{A^{21}}(\xi,y) \big| \leq C \bigg( \lambda^{\beta} \norm{\Lambda_{\mathcal{A}^{2},\Phi_{2}}- \Lambda_{\mathcal{A}^{1},\Phi_{1}} }_{*} + \dfrac{1}{\lambda^{\gamma}} \bigg).
    \]
    This concludes the proof of \eqref{eq: fourier transform estimates for A0 and A}. Next, for a fixed $\xi\in\R^2$, let us define $\theta\in\mathbb{S}^1$ by 
    \[
        \theta \vcentcolon = \dfrac{\xi_1 e_2-\xi_2 e_1}{|\xi|}\perp \xi,
    \] 
    where $(e_1,e_2)\in \R^2$ is the standard basis. Inserting this into \eqref{eq: fourier transform estimates for A0 and A}, we get \eqref{eq: estimate sigma matrix}. Therefore, we can conclude the proof.
\end{proof}

\begin{lemma}
\label{lemma: stability estimate for vector potential}
    Let $\Omega=\omega\times\R$ be an infinite waveguide with $\omega\subset\R^2$ being a smoothly bounded domain, $T>\text{diam}(\omega)$. Suppose that $\alpha\in(0,\min\{1, (T-\text{diam}(\omega))/3\})$ and $R_1,R_2>0$, $s_0,s_1\geq 2$ and $s_2\geq 0$.
     For any $\bar{s}_0,\bar{s}_1\in\R$ satisfying
     \begin{equation}
     \label{eq: def coeff stability lemma pot}
         \mathfrak{b}_0\vcentcolon =s_0-\bar{s}_0,\,\mathfrak{b}_1\vcentcolon = s_1-\bar{s}_1\in (0,1),
     \end{equation}
     there exist $C=C(T,\omega,R_1,R_2)>0$ such that we have   
     \begin{equation}
     \label{eq: main stability estimate for vector potential}
     \begin{split}
         &\| A^{1}_{0} -A^{2}_{0}\|^2_{L^{\infty}(\R_{y}; {H^{\bar{s}_0}(\R^2)})}+\| \nabla\wedge(A^1-A^2)\|^2_{L^{\infty}(\R_{y}; {H^{\bar{s}_1-1}(\R^2)})}\\
         &\quad\leq  C\frac{4^{\mathfrak{a}}}{\mathfrak{d}}\left(\frac{1+\mathfrak{a}}{\mathfrak{c}}\right)^{\frac{\mathfrak{c}}{1+\mathfrak{a}+\mathfrak{c}}} \norm{\Lambda_{\mathcal{A}^{2},\Phi_{2}}- \Lambda_{\mathcal{A}^{1},\Phi_{1}}}_{*}^{\mu}
     \end{split}
     \end{equation}
     for all $(\mathcal{A}^{j},\Phi_{j})\in \mathscr{A}_{s_0,s_1,s_2}(R_1,R_2)$, $j=1,2$,  satisfying condition \eqref{eq: smallness condition on vector potential}. Here, the coefficients $\mathfrak{a},\mathfrak{c},\mathfrak{d}$ and the H\"older exponent $\mu\in (0,1)$ are given by 
     \begin{equation}
     \label{eq: coeff a,c,d}
         \mathfrak{a}\vcentcolon =\max(s_0,s_1),\, \mathfrak{c}\vcentcolon = \min(\mathfrak{b}_0,\mathfrak{b}_1),\,\mathfrak{d}\vcentcolon = \Theta(\mathfrak{b}_1)\Theta(\mathfrak{b}_2)
     \end{equation}
     and
    \begin{equation}
    \label{eq: Holder exponent stability A}
        \mu\vcentcolon =\frac{2\gamma\mathfrak{c}}{(1+\mathfrak{a}+\mathfrak{c})(\gamma+\beta)}\in (0,1),
    \end{equation}
    where $\Theta (\rho)=1-\rho$ and $(\gamma,\beta)$ is any pair fulfilling the conditions of Lemma \ref{lemma: integral estimate of A}.
\end{lemma}
\begin{proof}
Throughout the proof let us set $B_\rho=B_\rho (0;\R^2)$ for all $\rho>0$. Using Plancherel's theorem, for any $r\geq 1$ and a.e. $y\in \R$, we can estimate
\begin{equation}
    \begin{split}
        &\|A_0^{21}(\cdot,y)\|^2_{H^{\bar{s}_0}(\R^2)}= C\|\langle \xi\rangle^{\bar{s}_0}\widehat{A_0^{21}}(\cdot,y)\|_{L^2(\R^2)}^2\\
        &\quad = C\left(\|\langle \xi\rangle^{\bar{s}_0}\widehat{A_0^{21}}(\cdot,y)\|_{L^2(B_r)}^2+\|\langle \xi\rangle^{\bar{s}_0}\widehat{A_0^{21}}(\cdot,y)\|_{L^2(B_r^c)}^2\right)\\
        &\quad \leq C\left(\|\widehat{A^{21}_0}(\cdot,y)\|^2_{L^{\infty}(B_r)}\|\langle \xi\rangle^{\bar{s}_0}\|_{L^2(B_r)}^2+\|\langle\xi\rangle^{\bar{s}_0}\widehat{A^{21}_0}(\cdot,y)\|_{L^2(B_r^c)}^2\right).
    \end{split}
\end{equation}
Using $\langle \xi\rangle^{\bar{s}_0}=\langle \xi\rangle^{\bar{s}_0-s_0}\langle \xi\rangle^{s_0}$, $\bar{s}_0<s_0$ and $r\geq 1$, we may bound
\begin{equation}
\label{eq: bound japanese bracket}
    \langle \xi\rangle^{\bar{s}_0}\leq
    \begin{cases}
        (2r)^{s_0}\langle\xi\rangle^{\bar{s}_0-s_0}&\text{ for }|\xi|\leq r,\\
        r^{\bar{s}_0-s_0}\langle\xi\rangle^{s_0}&\text{ for }|\xi|>r.
    \end{cases}
\end{equation}
Taking Lemma \ref{Lemma: estimates for fourier transform} into account we arrive at the estimate
\begin{equation}
\label{eq: estimate for A0}
\begin{split}
    &\|A_0^{21}(\cdot,y)\|^2_{H^{\bar{s}_0}(\R^2)}\\
    & \leq 4^{s_0}C\left(r^{2s_0}\left(\int_0^r\frac{\rho\,d\rho}{(1+\rho^2)^{s_0-\bar{s}_0}}\right)\|\widehat{A^{21}_0}(\cdot,y)\|^2_{L^{\infty}(B_r)}+r^{2(\bar{s}_0-s_0)}\|A_0^{21}\|^2_{H^{s_0}(\R^2)}\right)\\
    &\leq 4^{\bar{s}_0}C\left(\frac{r^{2s_0}r^{2(1-(s_0-\bar{s}_0))}}{1-(s_0-\bar{s}_0)} 
\|\widehat{A^{21}_0}(\cdot,y)\|^2_{L^{\infty}(B_r)}+r^{2(\bar{s}_0-s_0)}\|A_0^{21}\|^2_{H^{s_0}(\R^2)}\right) \\
 & \leq 4^{\bar{s}_0}C\left(\frac{r^{2(1+\bar{s}_0)}}{1-(s_0-\bar{s}_0)}
\left(\lambda^{2\beta}\norm{\Lambda_{\mathcal{A}^{2},\Phi_{2}}- \Lambda_{\mathcal{A}^{1},\Phi_{1}} }_{*}^2 + \dfrac{1}{\lambda^{2\gamma}}\right)+r^{2(\bar{s}_0-s_0)}\|A_0^{21}\|^2_{H^{s_0}(\R^2)}\right)
\end{split}
\end{equation}
as $\mathfrak{b}_0=s_0-\bar{s}_0<1$.
Thus, using $(\mathcal{A}^j,\Phi_j)\in \mathscr{A}_{s_0,s_1,s_2}(R_1,R_2)$ we get
\begin{equation}
\label{eq: estimate for A0 3}
    \begin{split}
         &\|A_0^{21}(\cdot,y)\|^2_{H^{\bar{s}_0}(\R^2)}\\
         &\leq C\frac{4^{\bar{s}_0}}{\Theta(s_0-\bar{s}_0)}\left(r^{2(1+\bar{s}_0)}
\left(\lambda^{2\beta}\norm{\Lambda_{\mathcal{A}^{2},\Phi_{2}}- \Lambda_{\mathcal{A}^{1},\Phi_{1}} }_{*}^2 + \dfrac{1}{\lambda^{2\gamma}}\right)+r^{2(\bar{s}_0-s_0)}\right),
    \end{split}
\end{equation}
where $\Theta(\rho)=1-\rho$. Similarly, we can bound
\begin{equation}
\label{eq: estimate for A}
\begin{split}
    &\|\nabla\wedge A^{21}(\cdot,y)\|^2_{H^{\bar{s}_1-1}(\R^2)}\\
    &= C\left(\|\langle \xi\rangle^{\bar{s}_1-1}\widehat{\nabla\wedge A^{21}}(\cdot,y)\|_{L^2(B_r)}^2+\|\langle \xi\rangle^{\bar{s}_1-1}\widehat{\nabla\wedge A^{21}}(\cdot,y)\|_{L^2(B_r^c)}^2\right)\\
    &\leq C\left(\|\langle\xi\rangle^{\bar{s}_1}\|_{L^2(B_r)}^2\|\widehat{\nabla\wedge A^{21}}(\cdot,y)/|\xi|\|_{L^{\infty}(B_r)}^2+r^{2(\bar{s}_1-s_1)}\|\nabla\wedge A^{21}(\cdot,y)\|_{H^{s_1-1}(\R^2)}\right)\\
    &\leq C\frac{4^{\bar{s}_1}}{\Theta(s_1-\bar{s}_1)}\left(r^{2(1+\bar{s}_1)}\|\widehat{\nabla\wedge A^{21}}(\cdot,y)/|\xi|\|_{L^{\infty}(B_r)}^2+r^{2(\bar{s}_1-s_1)}\|A^{21}(\cdot,y)\|_{H^{s_1}(\R^2)}\right)\\
    & \leq C\frac{4^{\bar{s}_1}}{\Theta(s_1-\bar{s}_1)}\left(r^{2(1+\bar{s}_1)}\left(\lambda^{2\beta}\norm{\Lambda_{\mathcal{A}^{2},\Phi_{2}}- \Lambda_{\mathcal{A}^{1},\Phi_{1}} }_{*}^2 + \dfrac{1}{\lambda^{2\gamma}}\right)+r^{2(\bar{s}_1-s_1)}\right)
\end{split}
\end{equation}
for $y\in\R$. Recalling the definition of $\mathfrak{a},\mathfrak{b}_1,\mathfrak{b}_2, \mathfrak{c}$ and $\mathfrak{d}$, we get
\begin{equation}
\label{eq: almost stability for vector potential}
\begin{split}
     &\|A_0^{21}(\cdot,y)\|^2_{H^{\bar{s}_0}(\R^2)}+\|\nabla\wedge A^{21}(\cdot,y)\|^2_{H^{\bar{s}_1-1}(\R^2)}\\
     &\quad\leq C\frac{4^{\mathfrak{a}}}{\mathfrak{d}}\left(r^{2(1+\mathfrak{a})}\left(\lambda^{2\beta}\norm{\Lambda_{\mathcal{A}^{2},\Phi_{2}}- \Lambda_{\mathcal{A}^{1},\Phi_{1}} }_{*}^2 + \dfrac{1}{\lambda^{2\gamma}}\right)+r^{-2\mathfrak{c}}\right)
\end{split}
\end{equation}
for $r\geq 1$. Next, observe that
\[
    r=\left(\frac{\mathfrak{c}}{1+\mathfrak{a}}\right)^{\frac{1}{2(1+\mathfrak{a}+\mathfrak{c})}}\lambda^{\frac{\gamma}{1+\mathfrak{a}+\mathfrak{c}}}
\]
minimizes the remainder term. Choosing this particular $r$ in \eqref{eq: almost stability for vector potential} and using $\mathfrak{a}\geq 0$ as well as $0<\mathfrak{c}< 1$, we deduce
\begin{equation}
    \label{eq: almost stability estimate for vector potential}
    \begin{split}
     &\|A_0^{21}(\cdot,y)\|^2_{H^{\bar{s}_0}(\R^2)}+\|\nabla\wedge A^{21}(\cdot,y)\|^2_{H^{\bar{s}_1-1}(\R^2)}\\
     &\quad \leq C\frac{4^{\mathfrak{a}}}{\mathfrak{d}}\bigg(\underbrace{\left(\frac{\mathfrak{c}}{1+\mathfrak{a}}\right)^{\frac{1+\mathfrak{a}}{1+\mathfrak{a}+\mathfrak{c}}}}_{\leq 1}\lambda^{\frac{2\gamma(1+\mathfrak{a})}{1+\mathfrak{a}+\mathfrak{c}}+2\beta}\norm{\Lambda_{\mathcal{A}^{2},\Phi_{2}}- \Lambda_{\mathcal{A}^{1},\Phi_{1}} }_{*}^2 \\
     &\quad\quad +\underbrace{\frac{1+\mathfrak{a}+\mathfrak{c}}{1+\mathfrak{a}}}_{\leq 2}\left(\frac{1+\mathfrak{a}}{\mathfrak{c}}\right)^{\frac{\mathfrak{c}}{1+\mathfrak{a}+\mathfrak{c}}}\lambda^{-\frac{2\gamma\mathfrak{c}}{1+\mathfrak{a}+\mathfrak{c}}}\bigg)\\
     &\quad \leq C\frac{4^{\mathfrak{a}}}{\mathfrak{d}}\left(\frac{1+\mathfrak{a}}{\mathfrak{c}}\right)^{\frac{\mathfrak{c}}{1+\mathfrak{a}+\mathfrak{c}}}\left(\lambda^{\frac{2\gamma(1+\mathfrak{a})}{1+\mathfrak{a}+\mathfrak{c}}+2\beta}\norm{\Lambda_{\mathcal{A}^{2},\Phi_{2}}- \Lambda_{\mathcal{A}^{1},\Phi_{1}} }_{*}^2+\lambda^{-\frac{2\gamma\mathfrak{c}}{1+\mathfrak{a}+\mathfrak{c}}}\right)
\end{split}
\end{equation}
for $\lambda\geq \lambda_0$, where
\begin{equation}
\label{eq: lower bound lambda 0}
    \lambda_0\vcentcolon = \left(\frac{1+\mathfrak{a}}{\mathfrak{c}}\right)^{\frac{1}{2\gamma}}\geq 1.
\end{equation}
Next, we want to take 
\begin{equation}
\label{eq: choice of lambda}
    \lambda=\|\Lambda_{\mathcal{A}^{2},\Phi_{2}}- \Lambda_{\mathcal{A}^{1},\Phi_{1}}\|_{*}^{-\rho}
\end{equation}
for some $\rho>0$. With that choice, equation \eqref{eq: almost stability estimate for vector potential} would imply
\[
\begin{split}
    &\|A_0^{21}(\cdot,y)\|^2_{H^{\bar{s}_0}(\R^2)}+\|\nabla\wedge A^{21}(\cdot,y)\|^2_{H^{\bar{s}_1-1}(\R^2)}\\
    &\quad \leq  C\frac{4^{\mathfrak{a}}}{\mathfrak{d}}\left(\frac{1+\mathfrak{a}}{\mathfrak{c}}\right)^{\frac{\mathfrak{c}}{1+\mathfrak{a}+\mathfrak{c}}}\left(\nu^{2-\rho\left(\frac{2\gamma(1+\mathfrak{a})}{1+\mathfrak{a}+\mathfrak{c}}+2\beta\right)}+\nu^{\rho \frac{2\gamma\mathfrak{c}}{1+\mathfrak{a}+\mathfrak{c}}}\right),
\end{split}
\]
where $\nu\vcentcolon = \|\Lambda_{\mathcal{A}^{2},\Phi_{2}}- \Lambda_{\mathcal{A}^{1},\Phi_{1}}\|_{*}$. Note that the powers of $\nu$ are equal if
\begin{equation}
\label{eq: choice of rho}
    \rho\vcentcolon = \frac{1}{\gamma + \beta}
\end{equation}
and in this case we get the estimate
\begin{equation}
\label{eq: stability estimate vector potential lemma}
    \begin{split}
        &\|A_0^{21}(\cdot,y)\|^2_{H^{\bar{s}_0}(\R^2)}+\|\nabla\wedge A^{21}(\cdot,y)\|^2_{H^{\bar{s}_1-1}(\R^2)}\\
        &\quad \leq C\frac{4^{\mathfrak{a}}}{\mathfrak{d}}\left(\frac{1+\mathfrak{a}}{\mathfrak{c}}\right)^{\frac{\mathfrak{c}}{1+\mathfrak{a}+\mathfrak{c}}}\|\Lambda_{\mathcal{A}^{2},\Phi_{2}}- \Lambda_{\mathcal{A}^{1},\Phi_{1}}\|_{*}^{\mu},
    \end{split}
\end{equation}
where $\mu$ is given by
\begin{equation}
\label{eq: coeff mu}
    \mu=\frac{2\gamma\mathfrak{c}}{(1+\mathfrak{a}+\mathfrak{c})(\gamma+\beta)}\in (0,1),
\end{equation}
which is increasing in $\mathfrak{c}$ and decreasing in $\mathfrak{a}$. The previous argument requires
\[
    \|\Lambda_{\mathcal{A}^{2},\Phi_{2}}- \Lambda_{\mathcal{A}^{1},\Phi_{1}}\|_{*}\leq \lambda_0^{-1/\rho},
\]
where $\lambda_0$ and $\rho$ are given by \eqref{eq: lower bound lambda 0} and \eqref{eq: choice of rho}, respectively. So, to conclude the proof it remains to show the estimate \eqref{eq: stability estimate vector potential lemma} in the case 
\begin{equation}
\label{eq: large DN data}
    \nu =\|\Lambda_{\mathcal{A}^{2},\Phi_{2}}- \Lambda_{\mathcal{A}^{1},\Phi_{1}}\|_{*}> \lambda_0^{-1/\rho}.
\end{equation}
For this let us recall that for all $s,t\in\R$ with $s\leq t$ one has the Sobolev embedding
\begin{equation}
\label{eq: general Sobolev embedding}
    H^t(\R^n)\hookrightarrow H^s(\R^n)\text{ with }\|u\|_{H^s(\R^n)}\leq \|u\|_{H^t(\R^n)}
\end{equation}
(see e.g. \cite[Theorem 12.5]{WO-pseudodifferential-operatros}). This in turn together with \eqref{eq: large DN data}, \eqref{eq: lower bound lambda 0} and \eqref{eq: choice of rho} gives
\[
\begin{split}
    &\|A_0^{21}(\cdot,y)\|^2_{H^{\bar{s}_0}(\R^2)}+\|\nabla\wedge A^{21}(\cdot,y)\|^2_{H^{\bar{s}_1-1}(\R^2)}\\
    &\leq \|A_0^{21}(\cdot,y)\|^2_{H^{s_0}(\R^2)}+\|\nabla\wedge A^{21}(\cdot,y)\|^2_{H^{s_1-1}(\R^2)}\\
    &\leq C( \|A_0^{21}(\cdot,y)\|^2_{H^{s_0}(\R^2)}+\|A^{21}(\cdot,y)\|^2_{H^{s_1}(\R^2)}) \\
    &\leq C\lambda_0^{\mu/\rho}\|\Lambda_{\mathcal{A}^{2},\Phi_{2}}- \Lambda_{\mathcal{A}^{1},\Phi_{1}}\|_{*}^{\mu}\\
    &= C\left(\frac{1+\mathfrak{a}}{\mathfrak{c}}\right)^{\frac{(\gamma+\beta)}{2\gamma}\frac{2\gamma\mathfrak{c}}{(1+\mathfrak{a}+\mathfrak{c})(\gamma+\beta)}}\|\Lambda_{\mathcal{A}^{2},\Phi_{2}}- \Lambda_{\mathcal{A}^{1},\Phi_{1}}\|_{*}^{\mu}\\
    &=C\left(\frac{1+\mathfrak{a}}{\mathfrak{c}}\right)^{\frac{\mathfrak{c}}{1+\mathfrak{a}+\mathfrak{c}}}\|\Lambda_{\mathcal{A}^{2},\Phi_{2}}- \Lambda_{\mathcal{A}^{1},\Phi_{1}}\|_{*}^{\mu}\\
    &\leq C\frac{4^{\mathfrak{a}}}{\mathfrak{d}}\left(\frac{1+\mathfrak{a}}{\mathfrak{c}}\right)^{\frac{\mathfrak{c}}{1+\mathfrak{a}+\mathfrak{c}}}\|\Lambda_{\mathcal{A}^{2},\Phi_{2}}- \Lambda_{\mathcal{A}^{1},\Phi_{1}}\|_{*}^{\mu}.
\end{split}
\]
Thus, in both cases we have shown the estimate \eqref{eq: main stability estimate for vector potential} and can conclude the proof.
\end{proof}
	
\subsection{Stability estimate for external potential \texorpdfstring{$\Phi$}{Phi}}
\label{subsec: stability of phi}

In this section we prove the analogous results for the external potential that we did above for the vector potential and then eventually deduce in Lemma \ref{lemma: stability estimate for potential} the stability estimate for $\Phi_{21}$. The pointwise estimate for the Radon transform and the Fourier transform of $\Phi_{21}$ reads as follows:

\begin{lemma}
\label{lemma: integral estimate of Phi}
    Let the assumptions of Lemma \ref{lemma: stability estimate for vector potential} hold. Furthermore, assume that the vector field $A^{21}\vcentcolon = A^2-A^1$ satisfies $\Div_x A^{21}=0$ and there holds $\mathfrak{b}_j\in (0,1)$ for $j=0,1$. If the DN maps satisfy
    \begin{equation}
        \|\Lambda_{\mathcal{A}^2,\Phi_2}-\Lambda_{\mathcal{A}^1,\Phi_1}\|_{*}\leq \lambda_0^{-1/\rho},
    \end{equation}
    where $\lambda_0$ and $\rho>0$ are the constants in \eqref{eq: lower bound lambda 0} and \eqref{eq: choice of rho}, respectively, then
    there exists a constant $C=C(\omega,T,R_1,R_2)>0$ such that for all $0<\delta<1/11$, $\kappa\geq 3+10\delta$ and $\theta \in \mathbb{S}^{1}$, we have
 \begin{equation}
        \label{eq: integral estimate for external potential}
       \begin{split}
    &\left|\int_{\R} \Phi_{21}(x - t\theta, y) \, dt\right| \\
    &\leq   C\Bigg(\lambda^{\kappa}\frac{4^{\mathfrak{a}}}{\mathfrak{c}\mathfrak{d}}\left(\frac{1+\mathfrak{a}}{\mathfrak{c}}\right)^{\frac{\mathfrak{c}\mathfrak{b}_1}{1+\mathfrak{a}+\mathfrak{c}}} \| \Lambda_{\mathcal{A}^2, \Phi_2} - \Lambda_{\mathcal{A}^1, \Phi_1} \|_{*}^{\mathfrak{b}_1\mu/2}+\lambda^{-\delta}\Bigg),
    \end{split}
  \end{equation}
for $(x,y)\in\R^3$ and $\lambda\geq 1$, where the constants are the ones introduced in Lemma \ref{lemma: stability estimate for vector potential}. Furthermore, under the above conditions there exists $C=C(\omega,T,R_1,R_2)>0$ such that there holds
\begin{equation}
\label{eq: estimate fourier transform external potential}
   \begin{split}
   &\big|\widehat{\Phi_{21}}(\xi,y) \big| \\
   &\leq C\Bigg(\lambda^{\kappa}\frac{4^{\mathfrak{a}}}{\mathfrak{c}\mathfrak{d}}\left(\frac{1+\mathfrak{a}}{\mathfrak{c}}\right)^{\frac{\mathfrak{c}\mathfrak{b}_1}{1+\mathfrak{a}+\mathfrak{c}}} \| \Lambda_{\mathcal{A}^2, \Phi_2} - \Lambda_{\mathcal{A}^1, \Phi_1} \|_{*}^{\mathfrak{b}_1\mu/2}+\lambda^{-\delta}\Bigg),
   \end{split}
\end{equation}
whenever $\lambda\geq 1$ and $\xi\perp\theta$.
\end{lemma}
\begin{proof}
    Throughout this proof, we use the same notation as in Lemma \ref{intid2 2}. First of all, note that equation \eqref{inteqn} ensures that we can write
    \begin{equation}
    \label{inteqn for phi}
    \begin{split}
        &\int_{\Omega_{T}} \tilde{\Phi}^{21}u_{2,\lambda}^{+}\overline{u_{1,\lambda}^{-}}\,dXdt \\
        &\quad = \int_{(\partial\Omega)_T}{\partial_{\nu}u_\lambda\overline{u_{1,\lambda}^{-}}\,dSdt}- \int_{\Omega_{T}}2\mathcal{A}^{21}\cdot(\partial_{t},-\nabla_{x})u_{2,\lambda}^{+} \overline{u_{1,\lambda}^{-}}\,dXdt, 
    \end{split}
\end{equation}
where the definition of $\widetilde{\Phi}^{21}$ is given in \eqref{eq: def phi int id 1} and $u_{j,\lambda}^{\pm}$, $j=1,2$, are the geometric optics solutions constructed in Lemma \ref{lemma: geom opt sol +} and \ref{lemma: geom opt sol -}. By \eqref{eq: decomposition remainder integral identity 1} and \eqref{eq: remainder estimate lemma int identity 1} we already know that there holds
\begin{equation}
\label{eq: remainder estimate lemma int identity phi}
 \left|\int_{\Omega_{T}}2\mathcal{A}^{21}\cdot(\partial_{t},-\nabla_{x})u_{2,\lambda}^{+} \overline{u_{1,\lambda}^{-}}\,dXdt\right| \leq C \lambda \| \mathcal{A}^{21} \|_{L^{\infty}(\Omega)}\|\varphi \|^2_{H^{3}(\R^2)} \| h\|^2_{H^2(\R)}
\end{equation}
for $\lambda \geq 1$ and $\varphi,h$ satisfying the conditions in Lemma \ref{intid2 2}. Hence, by \eqref{inteqn for phi} and \eqref{1.11}, we  obtain 
\begin{equation}
\label{eq: decay estimate tilde potential term}
\begin{split}
&\left|\int_{\Omega_{T}} \tilde{\Phi}^{21}u_{2,\lambda}^{+}\overline{u_{1,\lambda}^{-}}\,dXdt\right| \\
&\quad\leq C \, \bigg( \lambda^3 \| \Lambda_{\mathcal{A}^2, \Phi_2} - \Lambda_{\mathcal{A}^1, \Phi_1} \|_{*} + \lambda \| \mathcal{A}^{21} \|_{L^{\infty}(\Omega)} \bigg) \| \varphi \|^2_{H^{3}(\R^2)} \| h\|^2_{H^2(\R)}
\end{split}
\end{equation}
for all $\lambda\geq 1$. Next, notice that Lemma \ref{lemma: geom opt sol +} and \ref{lemma: geom opt sol -} ensure that we can write
\begin{equation}
\label{eq: decomposition for potential term}
    u_{2,\lambda}^{+}\overline{u_{1,\lambda}^{-}} = \varphi^2(x+t\theta) h^2(y) A^{21,+}(X,t) + Q_{\lambda},  
\end{equation}
and the remainder $Q_\lambda$ satisfies the estimate
\begin{equation}
\label{eq: remainder estimate for decomposition potential term}
    \|Q_{\lambda}\|_{L^{1}({\Omega_{T}})}\leq \frac{C}{\lambda} \|\varphi \|^2_{H^{3}(\R^2)} \| h\|^2_{H^2(\R)}
\end{equation}
Therefore, using \eqref{eq: def phi int id 1}, \eqref{eq: decay estimate tilde potential term}, \eqref{eq: decomposition for potential term}, \eqref{eq: remainder estimate for decomposition potential term} and $(\mathcal{A}^j,\Phi_j)\in\mathscr{A}_{s_0,s_1,s_2}(R_1,R_2)$, we get
\begin{equation}
\label{eq: potential term estimate 1}
\begin{split}
    &\left|\int_{\Omega_{T}} \Phi_{21}(X) \varphi^2(x+t\theta) h^2(y) A^{21,+}(X, t) \, dX dt\right| \\
&\quad \leq C \, \bigg( \lambda^3 \| \Lambda_{\mathcal{A}^2, \Phi_2} - \Lambda_{\mathcal{A}^1, \Phi_1} \|_{*} + \lambda \| \mathcal{A}^{21} \|_{L^{\infty}(\Omega)} + \frac{1}{\lambda} \bigg) \| \varphi \|^2_{H^{3}(\R^2)} \| h\|^2_{H^2(\R)}.
\end{split}
\end{equation}
for $\lambda\geq 1$. From now on, we extend $\Phi_{21} \in L^{\infty}(\Omega)$ by zero to all of $\R^3$. Then using the change of variables $x + t\theta \to x$ and $t - s\to s$, we derive from \eqref{eq: potential term estimate 1} the bound
\begin{equation}
\begin{split}
   &\left|\int_{\R^3_{T}} \Phi_{21}(x - t\theta, y) \varphi^2(x) h^2(y) \exp\left(-\int_0^t (1,-\theta) \cdot \mathcal{A}^{21}(x - s\theta, y) \, ds\right) \, dX dt \right| \\
    &\leq C \, \bigg( \lambda^3 \| \Lambda_{\mathcal{A}^2, \Phi_2} - \Lambda_{\mathcal{A}^1, \Phi_1} \|_{*} + \lambda \| \mathcal{A}^{21} \|_{L^{\infty}(\Omega)} + \frac{1}{\lambda} \bigg) \| \varphi \|^2_{H^{3}(\R^2)} \| h\|^2_{H^2(\R)}.
\end{split}
\end{equation}
This implies
\begin{equation}
\label{eq: eqn for Phi}
\begin{split}
   &\left|\int_{\R^3_T} \Phi_{21}(x - t\theta, y) \varphi^2(x) h^2(y)  \, dX dt\right| \\
    &\leq\left| \int_{\R^3_T} \Phi_{21}(x - t\theta, y)  \varphi^2(x) h^2(y)  \left[ 1 - \exp\left(-\int_0^t (1-\theta) \cdot \mathcal{A}^{21}(x - s\theta, y) \, ds\right) \right] \, dX  dt\right|\\
    &\quad + C \, \bigg( \lambda^3 \| \Lambda_{\mathcal{A}^2, \Phi_2} - \Lambda_{\mathcal{A}^1, \Phi_1} \|_{*} + \lambda \| \mathcal{A}^{21} \|_{L^{\infty}(\Omega)} + \frac{1}{\lambda} \bigg) \| \varphi \|^2_{H^{3}(\R^2)} \| h\|^2_{H^2(\R)}.
\end{split}
\end{equation}
Next, we apply \eqref{eq: simple estimate exponential} to bound the first term on the right hand side as
\begin{equation}
\begin{split}
    &\left|\int_{\R^3_T} \Phi_{21}(x - t\theta, y)  \varphi^2(x) h^2(y)  \left[1 - \exp\left(-\int_0^t (1-\theta) \cdot \mathcal{A}^{21}(x - s\theta, y) \, ds\right) \right] \, dX dt\right|\\
    &\quad\leq C \| \mathcal{A}^{21} \|_{L^{\infty}(\Omega)} \| \varphi \|^2_{L^2(\R^2)} \| h\|^2_{L^2(\R)}. 
\end{split}
\end{equation}
Thus, we have shown that there holds
\begin{equation}
\label{eq: potential estimate 2}
\begin{split}
&\left|\int_{\R^3_{T}} \Phi_{21}(x-t\theta, y) \varphi^2(x) h^2(y) \, dX dt\right| \\
&\quad \leq C \, \bigg( \lambda^3 \| \Lambda_{\mathcal{A}^2, \Phi_2} - \Lambda_{\mathcal{A}^1, \Phi_1} \|_{*} +\lambda\| \mathcal{A}^{21} \|_{L^{\infty}(\Omega)} + \frac{1}{\lambda} \bigg) \| \varphi \|^2_{H^{3}(\R^2)} \| h\|^2_{H^2(\R)} 
\end{split}
\end{equation}
for all $\lambda\geq 1$.

Now, let us fix some $x_0 \in \omega_\alpha$, $y_0\in\R$ and denote by $(\varphi_\epsilon)_{\eps>0}$ and $(h_{\epsilon})_{\eps>0}$ the mollifiers in $\R^2$ and $\R$, respectively, from Lemma \ref{lemma: integral estimate of A}. Then, the estimate \eqref{eq: potential estimate 2} implies
\begin{equation}
\label{eq: epsilon estimate external potential}
\begin{split}
&\left|\int_0^T \Phi_{21}(x_0 - t\theta, y_0) \, dt\right| =\left|\int_{\R^3_T} \Phi_{21}(x_0 - t\theta, y_0) \varphi_\epsilon^2(x) h_{\epsilon}^2(y) \, dxdy dt\right|\\
&\quad\leq\left|\int_{\R^3_T} \Phi_{21}(x - t\theta, y) \varphi_\epsilon^2(x) h_{\epsilon}^2(y) \, dxdy dt\right|\\ 
&\quad\quad +\left|\int_{\R^3_T} \left[ \Phi_{21}(x - t\theta, y) - \Phi_{21}(x_0 - t\theta, y_0) \right] 
\varphi_\epsilon^2(x) h_{\epsilon}^2(y) \, dxdydt\right| \\
&\quad \leq C \bigg( \lambda^3 \| \Lambda_{\mathcal{A}^2, \Phi_2} - \Lambda_{\mathcal{A}^1, \Phi_1} \|_{*} + \lambda \| \mathcal{A}^{21} \|_{L^{\infty}(\Omega)} + \frac{1}{\lambda} \bigg) \| \varphi_{\epsilon} \|^2_{H^{3}(\R^2)} \| h_{\eps}\|^2_{H^2(\R)}\\
&\quad\quad + C \int_{\R^3}|X-X_0|\varphi_\epsilon^2(x) h_{\epsilon}^2(y) \, dxdy\\
&\quad\leq C \left(\left( \lambda^3 \| \Lambda_{\mathcal{A}^2, \Phi_2} - \Lambda_{\mathcal{A}^1, \Phi_1} \|_{*} + \lambda \| \mathcal{A}^{21} \|_{L^{\infty}(\Omega)} + \frac{1}{\lambda} \right) \eps^{-10}+\eps\right)
\end{split}
\end{equation}
for all $\lambda\geq 1$ and $0<\eps\leq 1$, where we used in the last step that $\varphi$ and $h$ are supported in the unit ball, the squares of $\varphi,h$ integrate to one and the estimate \eqref{eq: bounds for mollifiers}. Next, let us recall that Morrey's embedding implies $H^{r}(\R^2)\hookrightarrow L^{\infty}(\R^2)$ for any $r>1$ and thus
\begin{equation}
\label{eq: morrey estimate A0 1}
     \|A_0^{21}(\cdot,y)\|_{L^{\infty}(\omega)}\leq C\|A^{21}_0(\cdot,y)\|_{H^{1+\eta}(\R^2)}
\end{equation}
for any $0<\eta<1$. Note that the constant in the previous estimate can be taken to be 
\begin{equation}
\label{eq: constant morrey embedding}
    C=\|\langle \xi\rangle^{-(1+\eta)}\|_{L^2(\R^2)}=(2\eta)^{-1/2}.
\end{equation}
As $\mathfrak{b}_0<1$ we can choose $\eta=\mathfrak{b}_0\in (0,1)$. Then, by invoking \eqref{eq: general Sobolev embedding}, we get
\begin{equation}
\label{eq: morrey estimate A0 2}
     \|A_0^{21}(\cdot,y)\|_{L^{\infty}(\omega)}\leq \frac{C}{\mathfrak{b}_0^{1/2}}\|A^{21}_0(\cdot,y)\|_{H^{\bar{s}_0}(\R^2)}.
\end{equation}
Similarly, by
\begin{equation}
\label{eq: morrey estimate A0 3}
\begin{split}
    \|A^{21}(\cdot,y)\|_{L^{\infty}(\omega)}&\leq \frac{C}{\eta^{1/2}}\|A^{21}(\cdot,y)\|_{H^{1+\eta}(\R^2)}\\
    &\leq \frac{C}{\eta^{1/2}}\|A^{21}(\cdot,y)\|^{1-\eta}_{H^{1}(\R^2)}\|A^{21}(\cdot,y)\|_{H^{2}(\R^2)}^{\eta}\\
    &\leq \frac{C}{\eta^{1/2}}\|A^{21}(\cdot,y)\|^{1-\eta}_{H^{1}(\R^2)},
\end{split}
\end{equation}
where we additionally used the classical interpolation in $H^t(\R^n)$ spaces and our boundedness assumption on $A^1,A^2$. By taking a $R>0$ sufficiently large such that $\omega\Subset B_R$, we deduce from \cite[Remark 3.5]{girault2012finite}, $\Div_x A^{21}=0$ and $A^{21}=0$ outside $\omega$ that there holds
\begin{equation}
\begin{split}
    \|A^{21}(\cdot,y)\|_{L^{\infty}(\omega)}&\leq \frac{C}{\eta^{1/2}}\|\nabla \wedge A^{21}(\cdot,y)\|^{1-\eta}_{L^2(B_R)}\\
    &\leq \frac{C}{\eta^{1/2}}\|\nabla \wedge A^{21}(\cdot,y)\|^{1-\eta}_{H^{\bar{s}_1-1}(\R^2)}
\end{split}
\end{equation}
for any $0<\eta<1$, where we used $\bar{s}_1-1=(\bar{s}_1-s_1)+s_1-1> 0$. We choose $\eta=1-\mathfrak{b}_1^{1/2}\in (0,1)$ to obtain
\begin{equation}
    \|A^{21}(\cdot,y)\|_{L^{\infty}(\omega)}\leq \frac{C}{(1-\mathfrak{b}_1^{1/2})^{1/2}}\|\nabla \wedge A^{21}(\cdot,y)\|^{\mathfrak{b}_1^{1/2}}_{H^{\bar{s}_1-1}(\R^2)}.
\end{equation}
Noting that
\[
    \frac{1}{(1-\mathfrak{b}_1^{1/2})^{1/2}}\leq \frac{(1+\mathfrak{b}_1^{1/2})^{1/2}}{(1-\mathfrak{b}_1)^{1/2}}\leq \frac{2}{(1-\mathfrak{b}_1)^{1/2}}.
\]
we get
\begin{equation}
\label{eq: morrey estimate A}
    \|A^{21}(\cdot,y)\|_{L^{\infty}(\omega)}\leq\frac{C}{(1-\mathfrak{b}_1)^{1/2}}\|\nabla \wedge A^{21}(\cdot,y)\|^{\mathfrak{b}_1^{1/2}}_{H^{\bar{s}_1-1}(\R^2)}.
\end{equation}
Thus, from Lemma \ref{lemma: stability estimate for vector potential}, \eqref{eq: morrey estimate A0 3} and \eqref{eq: morrey estimate A} we may infer the estimate
\begin{equation}
\label{eq: Stability in L infinity}
    \begin{split}
        &\|A^{21}_0(\cdot,y)\|^2_{L^{\infty}(\omega)}+\|A^{21}(\cdot,y)\|^2_{L^{\infty}(\omega)}\\
        &\leq \frac{C}{\mathfrak{b}_0}\|A^{21}_0(\cdot,y)\|^2_{H^{\bar{s}_0}(\R^2)}+\frac{C}{1-\mathfrak{b}_1}\|\nabla \wedge A^{21}(\cdot,y)\|^{\mathfrak{b}_1}_{H^{\bar{s}_1-1}(\R^2)}\\
        &\leq C\frac{4^{\mathfrak{a}}}{\mathfrak{c}\mathfrak{d}}\left(\frac{1+\mathfrak{a}}{\mathfrak{c}}\right)^{\frac{\mathfrak{c}}{1+\mathfrak{a}+\mathfrak{c}}} \|\Lambda_{\mathcal{A}^{2},\Phi_{2}}- \Lambda_{\mathcal{A}^{1},\Phi_{1}}\|_{*}^{\mu}\\
        &\quad +\frac{C}{\mathfrak{d}}\left(\frac{4^{\mathfrak{a}}}{\mathfrak{c}\mathfrak{d}}\left(\frac{1+\mathfrak{a}}{\mathfrak{c}}\right)^{\frac{\mathfrak{c}}{1+\mathfrak{a}+\mathfrak{c}}}\right)^{\mathfrak{b}_1}\|\Lambda_{\mathcal{A}^{2},\Phi_{2}}- \Lambda_{\mathcal{A}^{1},\Phi_{1}}\|_{*}^{\mathfrak{b}_1\mu}.
    \end{split}
\end{equation}
Next, let us observe that if $\lambda_0$, $\rho$ are defined by \eqref{eq: lower bound lambda 0}, \eqref{eq: choice of rho}, respectively, and
\begin{equation}
\label{eq: case small DN map}
    \|\Lambda_{\mathcal{A}^2,\Phi_2}-\Lambda_{\mathcal{A}^1,\Phi_1}\|_{*}\leq \lambda_0^{-1/\rho},
\end{equation}
then we have
\begin{equation}
\label{eq: bound DN map with larger exponent}
\begin{split}
     \|\Lambda_{\mathcal{A}^2,\Phi_2}-\Lambda_{\mathcal{A}^1,\Phi_1}\|_{*}^\mu&=\lambda_0^{-\mu/\rho}(\lambda_0^{1/\rho} \|\Lambda_{\mathcal{A}^2,\Phi_2}-\Lambda_{\mathcal{A}^1,\Phi_1}\|_{*})^\mu\\
     &\leq \lambda_0^{-\mu/\rho}(\lambda_0^{1/\rho} \|\Lambda_{\mathcal{A}^2,\Phi_2}-\Lambda_{\mathcal{A}^1,\Phi_1}\|_{*})^{\mathfrak{b}_1\mu}\\
     &=\lambda_0^{\mu(\mathfrak{b}_1-1)/\rho}\|\Lambda_{\mathcal{A}^2,\Phi_2}-\Lambda_{\mathcal{A}^1,\Phi_1}\|_{*}^{\mathfrak{b}_1\mu}\\
     &=\left(\frac{1+\mathfrak{a}}{\mathfrak{c}}\right)^{\frac{\mathfrak{c}(\mathfrak{b}_1-1)}{1+\mathfrak{a}+\mathfrak{c}}}\|\Lambda_{\mathcal{A}^2,\Phi_2}-\Lambda_{\mathcal{A}^1,\Phi_1}\|_{*}^{\mathfrak{b}_1\mu}.
\end{split}
\end{equation}
Thus, from \eqref{eq: Stability in L infinity} and \eqref{eq: bound DN map with larger exponent} we deduce that there holds
\begin{equation}
\label{eq: small control of L infinity norm of vector potential by DN map}
    \|\mathcal{A}^{21}(\cdot,y)\|_{L^{\infty}(\omega)}\leq C\frac{4^{\mathfrak{a}}}{\mathfrak{c}\mathfrak{d}}\left(\frac{1+\mathfrak{a}}{\mathfrak{c}}\right)^{\frac{\mathfrak{c}\mathfrak{b}_1}{1+\mathfrak{a}+\mathfrak{c}}}\|\Lambda_{\mathcal{A}^{2},\Phi_{2}}- \Lambda_{\mathcal{A}^{1},\Phi_{1}}\|_{*}^{\mathfrak{b}_1 \mu/2},
\end{equation}
when the DN maps satisfy \eqref{eq: case small DN map}.
For later, we record here the following stability estimate for large difference of DN data:
\begin{equation}
    \label{eq: large control of L infinity norm of vector potential by DN map}
    \|\mathcal{A}^{21}(\cdot,y)\|_{L^{\infty}(\omega)}\leq C\frac{4^{\mathfrak{a}}}{\mathfrak{c}\mathfrak{d}}\left(\frac{1+\mathfrak{a}}{\mathfrak{c}}\right)^{\frac{\mathfrak{c}}{1+\mathfrak{a}+\mathfrak{c}}}\|\Lambda_{\mathcal{A}^{2},\Phi_{2}}- \Lambda_{\mathcal{A}^{1},\Phi_{1}}\|_{*}^{\mu}.
\end{equation}
Thus, in the case of a small difference of DN maps in the sense of \eqref{eq: case small DN map} we infer from \eqref{eq: epsilon estimate external potential} and \eqref{eq: small control of L infinity norm of vector potential by DN map} the bound
\begin{equation}
\label{eq: estimate external potential with eps and lambda 2}
\begin{split}
    &\left|\int_0^T \Phi_{21}(x_0 - t\theta, y_0) \, dt\right| \\
    &\leq   C \frac{\lambda^3}{\eps^{10}} \| \Lambda_{\mathcal{A}^2, \Phi_2} - \Lambda_{\mathcal{A}^1, \Phi_1} \|_{*} + C\frac{\lambda}{\eps^{10}}\frac{4^{\mathfrak{a}}}{\mathfrak{c}\mathfrak{d}}\left(\frac{1+\mathfrak{a}}{\mathfrak{c}}\right)^{\frac{\mathfrak{c}\mathfrak{b}_1}{1+\mathfrak{a}+\mathfrak{c}}}\| \Lambda_{\mathcal{A}^2, \Phi_2} - \Lambda_{\mathcal{A}^1, \Phi_1} \|_{*}^{\mathfrak{b}_1\mu/2}\\
    &\quad +\frac{C}{\lambda\eps^{10}}+C\eps
\end{split}
\end{equation}
Arguing as in \eqref{eq: bound DN map with larger exponent}, we see that under the smallness assumption \eqref{eq: case small DN map} there holds
\begin{equation}
    \begin{split}
          \|\Lambda_{\mathcal{A}^2,\Phi_2}-\Lambda_{\mathcal{A}^1,\Phi_1}\|_{*}&=\lambda_0^{-1/\rho}(\lambda_0^{1/\rho} \|\Lambda_{\mathcal{A}^2,\Phi_2}-\Lambda_{\mathcal{A}^1,\Phi_1}\|_{*})\\
     &\leq \lambda_0^{-1/\rho}(\lambda_0^{1/\rho} \|\Lambda_{\mathcal{A}^2,\Phi_2}-\Lambda_{\mathcal{A}^1,\Phi_1}\|_{*})^{\mathfrak{b}_1\mu/2}\\
     &=\lambda_0^{(\mathfrak{b}_1\mu/2-1)/\rho}\|\Lambda_{\mathcal{A}^2,\Phi_2}-\Lambda_{\mathcal{A}^1,\Phi_1}\|_{*}^{\mathfrak{b}_1\mu/2}\\
     &\leq\lambda_0^{(\mathfrak{b}_1\mu-1)/\rho}\|\Lambda_{\mathcal{A}^2,\Phi_2}-\Lambda_{\mathcal{A}^1,\Phi_1}\|_{*}^{\mathfrak{b}_1\mu/2}\\
     &=\left(\frac{1+\mathfrak{a}}{\mathfrak{c}}\right)^{\frac{\mathfrak{b}_1\mathfrak{c}}{1+\mathfrak{a}+\mathfrak{c}}-\frac{\gamma+\beta}{2\gamma}}\|\Lambda_{\mathcal{A}^2,\Phi_2}-\Lambda_{\mathcal{A}^1,\Phi_1}\|_{*}^{\mathfrak{b}_1\mu/2}\\
     &\leq \left(\frac{1+\mathfrak{a}}{\mathfrak{c}}\right)^{\frac{\mathfrak{b}_1\mathfrak{c}}{1+\mathfrak{a}+\mathfrak{c}}}\|\Lambda_{\mathcal{A}^2,\Phi_2}-\Lambda_{\mathcal{A}^1,\Phi_1}\|_{*}^{\mathfrak{b}_1\mu/2}.
    \end{split}
\end{equation}
Using $\lambda\geq 1$ and \eqref{eq: estimate external potential with eps and lambda 2}, we get
\begin{equation}
    \begin{split}
    &\left|\int_0^T \Phi_{21}(x_0 - t\theta, y_0) \, dt\right| \\
    &\leq   C\frac{\lambda^3}{\eps^{10}}\frac{4^{\mathfrak{a}}}{\mathfrak{c}\mathfrak{d}}\left(\frac{1+\mathfrak{a}}{\mathfrak{c}}\right)^{\frac{\mathfrak{c}\mathfrak{b}_1}{1+\mathfrak{a}+\mathfrak{c}}} \| \Lambda_{\mathcal{A}^2, \Phi_2} - \Lambda_{\mathcal{A}^1, \Phi_1} \|_{*}^{\mathfrak{b}_1\mu/2}\\
    &\quad +C\frac{\lambda}{\eps^{10}}\frac{4^{\mathfrak{a}}}{\mathfrak{c}\mathfrak{d}}\left(\frac{1+\mathfrak{a}}{\mathfrak{c}}\right)^{\frac{\mathfrak{c}\mathfrak{b}_1}{1+\mathfrak{a}+\mathfrak{c}}}\| \Lambda_{\mathcal{A}^2, \Phi_2} - \Lambda_{\mathcal{A}^1, \Phi_1} \|_{*}^{\mathfrak{b}_1\mu/2}+\frac{C}{\lambda\eps^{10}}+C\eps
    \end{split}
\end{equation}
Next, we argue similarly as in the proof of Lemma \ref{lemma: integral estimate of A}. Thus, let us set $\eps=\lambda^{-\delta}$ for some $0<\delta<1/11$. Then for any $\kappa\geq 3+10\delta$, we have $0<\eps\leq 1$, 
\[
    \lambda^3/\eps^{10}+\lambda/\eps^{10}\leq 2\lambda^{3+10\delta}\leq \lambda^\kappa
\]
and
\[
    \eps+\frac{1}{\lambda\eps^{10}}=\lambda^{-\delta}+\frac{1}{\lambda^{1-10\delta}}\leq 2\lambda^{-\delta}.
\]
This choice implies
\begin{equation}
    \begin{split}
    &\left|\int_0^T \Phi_{21}(x_0 - t\theta, y_0) \, dt\right| \\
    &\leq   C\Bigg(\lambda^{\kappa}\frac{4^{\mathfrak{a}}}{\mathfrak{c}\mathfrak{d}}\left(\frac{1+\mathfrak{a}}{\mathfrak{c}}\right)^{\frac{\mathfrak{c}\mathfrak{b}_1}{1+\mathfrak{a}+\mathfrak{c}}} \| \Lambda_{\mathcal{A}^2, \Phi_2} - \Lambda_{\mathcal{A}^1, \Phi_1} \|_{*}^{\mathfrak{b}_1\mu/2}+\lambda^{-\delta}\Bigg).
    \end{split}
\end{equation}
Replacing $\theta$ by $-\theta$ and using a change of variables, we deduce 
\begin{equation}
\label{eq: Radon transform estimate for Phi}
    \begin{split}
    &\left|\int_{\R} \Phi_{21}(x_0 - t\theta, y_0) \, dt\right| \\
    &\leq   C\Bigg(\lambda^{\kappa}\frac{4^{\mathfrak{a}}}{\mathfrak{c}\mathfrak{d}}\left(\frac{1+\mathfrak{a}}{\mathfrak{c}}\right)^{\frac{\mathfrak{c}\mathfrak{b}_1}{1+\mathfrak{a}+\mathfrak{c}}} \| \Lambda_{\mathcal{A}^2, \Phi_2} - \Lambda_{\mathcal{A}^1, \Phi_1} \|_{*}^{\mathfrak{b}_1\mu/2}+\lambda^{-\delta}\Bigg),
    \end{split}
\end{equation}
where we also used that $\Phi(x_0-s\theta,y_0)=0$ for $x_0\in \omega$ and $s\geq T$. The previous identity can be extended to all $x_0\in \R^2$ like in the proof of Lemma \ref{lemma: integral estimate of A}. Thus, the estimate \eqref{eq: Radon transform estimate for Phi} holds under the condition \eqref{eq: case small DN map} for all $(x_0,y_0)\in \R^3$, $\theta\in\mathbb{S}^1$ and $\lambda\geq \lambda_0$. This concludes the proof of \eqref{eq: integral estimate for external potential}.\\

Eventually, we show the estimate \eqref{eq: estimate fourier transform external potential}. By the very same computation as in the beginning of the proof of Lemma \ref{Lemma: estimates for fourier transform}, we have
\begin{equation}
    \label{change of variables fourier of radon for q}
        \begin{split}
            \widehat{\Phi_{21}}(\xi,y)=\int_{\R^2}e^{-\im \rho\theta^{\perp}\cdot \xi}\Phi_{21}(\rho\theta^\perp-s\theta,y)\,ds\,d\rho
        \end{split}
    \end{equation}
    for any $\xi\perp\theta$ and $y\in\R$. Therefore, the estimate \eqref{eq: Radon transform estimate for Phi} yields
    \begin{equation}
    \label{eq: estimate fourier of full  potential}
    \begin{split}
       \big| \widehat{\Phi_{21}}(\xi,y) \big| &\leq\int_{-r}^r\left|\int_{\R}
{\Phi}_{21}(\rho\theta^\perp-s\theta,y)\,ds\right|\,d\rho \\
         &\leq C\left(\lambda^{\kappa}\frac{4^{\mathfrak{a}}}{\mathfrak{c}\mathfrak{d}}\left(\frac{1+\mathfrak{a}}{\mathfrak{c}}\right)^{\frac{\mathfrak{c}\mathfrak{b}_1}{1+\mathfrak{a}+\mathfrak{c}}} \| \Lambda_{\mathcal{A}^2, \Phi_2} - \Lambda_{\mathcal{A}^1, \Phi_1} \|_{*}^{\mathfrak{b}_1\mu/2}+\lambda^{-\delta}\right)
    \end{split}
    \end{equation}
    for some suitable $r>0$ and we can conclude the proof.
\end{proof}
Note that the proof of the previous lemma shows the following corollary.
\begin{corollary}[$L^\infty$ stability estimate for the vector potential]
\label{cor: L infinity stability}
    Suppose the assumptions of Lemma \ref{lemma: integral estimate of Phi} hold and let the forthcoming constants be as in Lemma \ref{lemma: stability estimate for vector potential}. Then the following stability estimates hold
    \begin{equation}
     \|A^1-A^2\|_{L^{\infty}(\Omega)}\leq C\frac{4^{\mathfrak{a}}}{\mathfrak{c}\mathfrak{d}}\left(\frac{1+\mathfrak{a}}{\mathfrak{c}}\right)^{\frac{\mathfrak{c}\mathfrak{b}_1}{1+\mathfrak{a}+\mathfrak{c}}}\|\Lambda_{\mathcal{A}^{2},\Phi_{2}}- \Lambda_{\mathcal{A}^{1},\Phi_{1}}\|_{*}^{\mathfrak{b}_1 \mu/2},
    \end{equation}
when the DN maps satisfy 
\begin{equation}
    \|\Lambda_{\mathcal{A}^2,\Phi_2}-\Lambda_{\mathcal{A}^1,\Phi_1}\|_{*}\leq \lambda_0^{-1/\rho},
\end{equation}
and 
\begin{equation}
       \|A^1-A^2\|_{L^{\infty}(\Omega)}\leq C\frac{4^{\mathfrak{a}}}{\mathfrak{c}\mathfrak{d}}\left(\frac{1+\mathfrak{a}}{\mathfrak{c}}\right)^{\frac{\mathfrak{c}}{1+\mathfrak{a}+\mathfrak{c}}}\|\Lambda_{\mathcal{A}^{2},\Phi_{2}}- \Lambda_{\mathcal{A}^{1},\Phi_{1}}\|_{*}^{\mu}
\end{equation}
otherwise. Furthermore, the temporal part of the vector potential satisfies the following $L^{\infty}$ stability estimate
 \begin{equation}
        \|A^1_0-A^2_0\|_{L^{\infty}(\Omega)}\leq C\frac{4^{\mathfrak{a}}}{\mathfrak{c}\mathfrak{d}}\left(\frac{1+\mathfrak{a}}{\mathfrak{c}}\right)^{\frac{\mathfrak{c}}{1+\mathfrak{a}+\mathfrak{c}}} \|\Lambda_{\mathcal{A}^{2},\Phi_{2}}- \Lambda_{\mathcal{A}^{1},\Phi_{1}}\|_{*}^{\mu}.
    \end{equation}
\end{corollary}
Now, we can again follow a similar argument as for the vector potential (Lemma \ref{lemma: stability estimate for vector potential}) to deduce the desired stability estimate for external potential.
\begin{lemma}
\label{lemma: stability estimate for potential}
    Let the assumptions of Lemma \ref{lemma: stability estimate for vector potential} hold. Furthermore, assume that the vector field $A^{21}\vcentcolon = A^2-A^1$ satisfies $\Div_x A^{21}=0$ and there holds $\mathfrak{b}_j\in (0,1)$ for $j=0,1$. For any $\bar{s}_2\in\R$ satisfying
    \begin{equation}
        \mathfrak{b}_2\vcentcolon = s_2-\bar{s}_2\in (0,1)
    \end{equation}
    there exists $C=C(\omega,T,R_1,R_2)>0$ such that
\begin{equation}
\label{eq: main stability estimate for external potential}
\begin{split}
    &\|\Phi_1-\Phi_2\|_{L^{\infty}(\R_{y}; {H^{\bar{s}_2}(\R^2)})} \\
    &\leq C\frac{ 4^{3\mathfrak{A}}}{(\mathfrak{C}\mathfrak{D})^2}\left(\frac{1+\mathfrak{A}}{\mathfrak{C}}\right)^{\frac{\gamma+\beta}{\gamma}+\frac{2\mathfrak{c}\mathfrak{b}_1}{1+\mathfrak{a}+\mathfrak{c}}+\frac{\mathfrak{b}_2}{1+\bar{s}_2+\mathfrak{b}_2}} \|\Lambda_{\mathcal{A}^{2},\Phi_{2}}-\Lambda_{\mathcal{A}^{1},\Phi_{1}}\|_{*}^{\zeta}
\end{split}
\end{equation}
    for all $(\mathcal{A}^{j},\Phi_{j})\in \mathscr{A}_{s_0,s_1,s_2}(R_1,R_2)$, $j=1,2$,  satisfying condition \eqref{eq: smallness condition on vector potential}. Here, the coefficients and H\"older exponent are given by \eqref{eq: coeff a,c,d}, 
     \[
    \mathfrak{A}\vcentcolon = \max(s_0,s_1,s_2),\,\mathfrak{C}\vcentcolon = \min(\mathfrak{b}_0,\mathfrak{b}_1,\mathfrak{b}_2),\,\mathfrak{D}\vcentcolon = \Theta(\mathfrak{b}_0)\Theta(\mathfrak{b}_1)\Theta(\mathfrak{b}_2).
    \]
    and
    \begin{equation}
    \label{eq: Holder exponent stability Phi}
        \zeta\vcentcolon = \frac{\mathfrak{b}_1\mathfrak{b}_2\delta\mu}{(1+\bar{s}_2+\mathfrak{b}_2)(\delta+\kappa)}\in (0,1),
    \end{equation}
    where $(\delta,\kappa)$ is any pair fulfilling the conditions of Lemma \ref{lemma: integral estimate of Phi}.
\end{lemma}
\begin{proof} 
We will use a similar idea as in the proof of Lemma \ref{lemma: stability estimate for vector potential}. For any $r\geq 1$ there holds
\begin{equation}
\label{eq: estimate for phi}
    \begin{split}
        &\|\Phi_{21}(\cdot,y)\|^2_{H^{\bar{s}_2}(\R^2)}= C\|\langle \xi\rangle^{\bar{s}_2}\widehat{\Phi_{21}}(\cdot,y)\|_{L^2(\R^2)}^2\\
        &\quad \leq C\left(\|\widehat{\Phi_{21}}(\cdot,y)\|^2_{L^{\infty}(B_r)}\|\langle \xi\rangle^{\bar{s}_2}\|_{L^2(B_r)}^2+\|\langle\xi\rangle^{\bar{s}_2}\widehat{\Phi_{21}}(\cdot,y)\|_{L^2(B_r^c)}^2\right)\\
         & \quad\leq 4^{s_2}C\left(r^{2s_2}\left(\int_0^r\frac{\rho\,d\rho}{(1+\rho^2)^{s_2-\bar{s}_2}}\right)\|\widehat{\Phi_{21}}(\cdot,y)\|^2_{L^{\infty}(B_r)}+r^{2(\bar{s}_2-s_2)}\|\Phi_{21}\|^2_{H^{s_2}(\R^2)}\right)\\
    &\quad\leq 4^{\bar{s}_2}C\left(\frac{r^{2(1+\bar{s}_2)}}{1-(s_2-\bar{s}_2)}
\|\widehat{\Phi_{21}}(\cdot,y)\|^2_{L^{\infty}(B_r)}+r^{2(\bar{s}_2-s_2)}\right),
    \end{split} 
\end{equation}
where we used \eqref{eq: bound japanese bracket}. Thus, if 
\begin{equation}
\label{eq: smallness condition stability phi}
       \|\Lambda_{\mathcal{A}^2,\Phi_2}-\Lambda_{\mathcal{A}^1,\Phi_1}\|_{*}\leq \lambda_0^{-1/\rho},
\end{equation}
we get from Lemma \ref{lemma: integral estimate of Phi} the estimate
\[
\begin{split}
     &\|\Phi_{21}(\cdot,y)\|^2_{H^{\bar{s}_2}(\R^2)}= C\|\langle \xi\rangle^{\bar{s}_2}\widehat{\Phi_{21}}(\cdot,y)\|_{L^2(\R^2)}^2\\
    &\leq \frac{C 4^{\bar{s}_2}r^{2(1+\bar{s}_2)}}{\Theta(\mathfrak{b}_2)}\Bigg(\lambda^{2\kappa}\frac{4^{2\mathfrak{a}}}{(\mathfrak{c}\mathfrak{d})^2}\left(\frac{1+\mathfrak{a}}{\mathfrak{c}}\right)^{\frac{2\mathfrak{c}\mathfrak{b}_1}{1+\mathfrak{a}+\mathfrak{c}}} \| \Lambda_{\mathcal{A}^2, \Phi_2} - \Lambda_{\mathcal{A}^1, \Phi_1} \|_{*}^{\mathfrak{b}_1\mu}+\lambda^{-2\delta}\Bigg)\\
    &\quad +\frac{C 4^{\bar{s}_2}}{\Theta(\mathfrak{b}_2)}r^{-2\mathfrak{b}_2}.
\end{split}
\]
for any $0<\delta<1/11$ and $\kappa\geq 3+10\delta$. Next, notice that the remainder is minimized by taking
\[
    r=\left(\frac{\mathfrak{b}_2}{1+\bar{s}_2}\right)^{\frac{1}{2(1+\mathfrak{b}_2+\bar{s}_2)}}\lambda^{\frac{\delta}{1+\mathfrak{b}_2+\bar{s}_2}}
\]
and for this choice one easily calculates
\[
    r^{2(1+\bar{s}_2)}\lambda^{-2\delta}+r^{-2\mathfrak{b}_2}=\frac{1+\bar{s}_2+\mathfrak{b}_2}{1+\bar{s}_2}\left(\frac{1+\bar{s}_2}{\mathfrak{b}_2}\right)^{\frac{\mathfrak{b}_2}{1+\bar{s}_2+\mathfrak{b}_2}}\lambda^{-\frac{2\mathfrak{b}_2\delta}{1+\bar{s}_2+\mathfrak{b}_2}}.
\]
Thus, we obtain
\begin{equation}
\label{eq: almost final estimate for stability potential}
\begin{split}
    &\|\Phi_{21}(\cdot,y)\|^2_{H^{\bar{s}_2}(\R^2)}\\
    &\leq C\frac{ 4^{\bar{s}_2}}{\Theta(\mathfrak{b}_2)}\frac{4^{2\mathfrak{a}}}{(\mathfrak{c}\mathfrak{d})^2}\left(\frac{1+\mathfrak{a}}{\mathfrak{c}}\right)^{\frac{2\mathfrak{c}\mathfrak{b}_1}{1+\mathfrak{a}+\mathfrak{c}}} \lambda^{2\left(\frac{\delta(1+\bar{s}_2)
    }{1+\bar{s}_2+\mathfrak{b}_2}+\kappa\right)}\| \Lambda_{\mathcal{A}^2, \Phi_2} - \Lambda_{\mathcal{A}^1, \Phi_1} \|_{*}^{\mathfrak{b}_1\mu}\\
    &\quad+\frac{C 4^{\bar{s}_2}}{\Theta(\mathfrak{b}_2)}\left(\frac{1+\bar{s}_2}{\mathfrak{b}_2}\right)^{\frac{\mathfrak{b}_2}{1+\bar{s}_2+\mathfrak{b}_2}}\lambda^{-\frac{2\mathfrak{b}_2\delta}{1+\bar{s}_2+\mathfrak{b}_2}}.
\end{split}
\end{equation}
Now, we again want to take
\begin{equation}
\label{eq: choice of lambda external potential}
    \lambda=\|\Lambda_{\mathcal{A}^{2},\Phi_{2}}- \Lambda_{\mathcal{A}^{1},\Phi_{1}}\|_{*}^{-\tau}
\end{equation}
for a suitable $\tau>0$. Inserting this into \eqref{eq: almost final estimate for stability potential} gives the bound
\begin{equation}
    \begin{split}
        &\|\Phi_{21}(\cdot,y)\|^2_{H^{\bar{s}_2}(\R^2)}\leq C\frac{ 4^{\bar{s}_2}}{\Theta(\mathfrak{b}_2)}\\
        &\times\Bigg(\frac{4^{2\mathfrak{a}}}{(\mathfrak{c}\mathfrak{d})^2}\left(\frac{1+\mathfrak{a}}{\mathfrak{c}}\right)^{\frac{2\mathfrak{c}\mathfrak{b}_1}{1+\mathfrak{a}+\mathfrak{c}}} \nu^{\mathfrak{b}_1\mu-2\tau\left(\frac{\delta(1+\bar{s}_2)
    }{1+\bar{s}_2+\mathfrak{b}_2}+\kappa\right)}+\left(\frac{1+\bar{s}_2}{\mathfrak{b}_2}\right)^{\frac{\mathfrak{b}_2}{1+\bar{s}_2+\mathfrak{b}_2}}\nu^{\frac{2\tau\mathfrak{b}_2\delta}{1+\bar{s}_2+\mathfrak{b}_2}}\Bigg).
    \end{split}
\end{equation}
where we put $\nu=\|\Lambda_{\mathcal{A}^{2},\Phi_{2}}- \Lambda_{\mathcal{A}^{1},\Phi_{1}}\|_{*}$. A direct calculation shows that the exponents of $\nu$ are equal if we choose
\begin{equation}
\label{eq: choice of rho for phi}
    \tau=\frac{\mathfrak{b_1}\mu}{2(\delta+\kappa)}
\end{equation} 
 and this yields
 \begin{equation}
 \label{eq: cond for det right exponent}
 \begin{split}
      \|\Phi_{21}(\cdot,y)\|^2_{H^{\bar{s}_2}(\R^2)}&\leq C\frac{ 4^{3\mathfrak{A}}}{(\mathfrak{C}\mathfrak{D})^2}\left(\frac{1+\mathfrak{A}}{\mathfrak{C}}\right)^{\frac{2\mathfrak{c}\mathfrak{b}_1}{1+\mathfrak{a}+\mathfrak{c}}+\frac{\mathfrak{b}_2}{1+\bar{s}_2+\mathfrak{b}_2}} \|\Lambda_{\mathcal{A}^{2},\Phi_{2}}-\Lambda_{\mathcal{A}^{1},\Phi_{1}}\|_{*}^{\zeta},
 \end{split}
 \end{equation}
 where $\zeta$ is given by
 \begin{equation}
 \label{eq: def of zeta}
     \zeta\vcentcolon = \frac{\mathfrak{b}_1\mathfrak{b}_2\delta \mu}{(1+\bar{s}_2+\mathfrak{b}_2)(\delta+\kappa)}.
 \end{equation}
 Notice that the above argument also requires $r\geq 1$ and so it only holds when
 \[
    \lambda\geq \widetilde{\lambda}_0\vcentcolon = \left(\frac{1+\bar{s}_2}{\mathfrak{b}_2}\right)^{1/2\delta}\geq 1
 \]
 and thus applies for DN maps satisfying
 \[
    \|\Lambda_{\mathcal{A}^{2},\Phi_{2}}-\Lambda_{\mathcal{A}^{1},\Phi_{1}}\|_{*}\leq \chi\vcentcolon =\min\left(\left(\frac{\mathfrak{b}_2}{1+\bar{s}_2}\right)^{\frac{\delta+\kappa}{\delta\mathfrak{b}_1\mu}},\left(\frac{\mathfrak{c}}{1+\mathfrak{a}}\right)^{\frac{\gamma+\beta}{2\gamma}}\right)
 \]
 (see \eqref{eq: smallness condition stability phi}).
 If 
 \[
    \|\Lambda_{\mathcal{A}^{2},\Phi_{2}}-\Lambda_{\mathcal{A}^{1},\Phi_{1}}\|_{*}> \chi,
 \] 
 then using \eqref{eq: general Sobolev embedding} and
 \[
    \chi\geq \left(\frac{\mathfrak{C}}{1+\mathfrak{A}}\right)^{\max(\frac{\delta+\kappa}{\delta\mathfrak{b}_1\mu},\frac{\gamma+\beta}{2\gamma})}
 \]
 we may estimate
 \begin{equation}
     \begin{split}
             &\|\Phi_{21}(\cdot,y)\|^2_{H^{\bar{s}_2}(\R^2)}\leq \|\Phi_{21}(\cdot,y)\|^2_{H^{s_2}(\R^2)}\\
    &\leq C\chi^{-\zeta}\|\Lambda_{\mathcal{A}^{2},\Phi_{2}}- \Lambda_{\mathcal{A}^{1},\Phi_{1}}\|_{*}^{\zeta}\\
    &\leq C\left(\frac{1+\mathfrak{A}}{\mathfrak{C}}\right)^{\max(\frac{\delta+\kappa}{\delta\mathfrak{b}_1\mu},\frac{\gamma+\beta}{2\gamma})\frac{\mathfrak{b}_1\mathfrak{b}_2\delta\mu}{(1+\bar{s}_2+\mathfrak{b}_2)(\delta+\kappa)}}\|\Lambda_{\mathcal{A}^{2},\Phi_{2}}- \Lambda_{\mathcal{A}^{1},\Phi_{1}}\|_{*}^{\zeta}\\
    &\leq C\left(\frac{1+\mathfrak{A}}{\mathfrak{C}}\right)^{(\frac{\delta+\kappa}{\delta\mathfrak{b}_1\mu}+\frac{\gamma+\beta}{2\gamma})\frac{\mathfrak{b}_1\mathfrak{b}_2\delta\mu}{(1+\bar{s}_2+\mathfrak{b}_2)(\delta+\kappa)}}\|\Lambda_{\mathcal{A}^{2},\Phi_{2}}- \Lambda_{\mathcal{A}^{1},\Phi_{1}}\|_{*}^{\zeta}\\
     &\leq C\left(\frac{1+\mathfrak{A}}{\mathfrak{C}}\right)^{\frac{\gamma+\beta}{\gamma}+\frac{\delta+\kappa}{\delta\mathfrak{b}_1\mu}\frac{\mathfrak{b}_1\mathfrak{b}_2\delta\mu}{(1+\bar{s}_2+\mathfrak{b}_2)(\delta+\kappa)}}\|\Lambda_{\mathcal{A}^{2},\Phi_{2}}- \Lambda_{\mathcal{A}^{1},\Phi_{1}}\|_{*}^{\zeta}\\
     &=C\left(\frac{1+\mathfrak{A}}{\mathfrak{C}}\right)^{\frac{\gamma+\beta}{\gamma}+\frac{\mathfrak{b}_2}{1+\bar{s}_2+\mathfrak{b}_2}}\|\Lambda_{\mathcal{A}^{2},\Phi_{2}}- \Lambda_{\mathcal{A}^{1},\Phi_{1}}\|_{*}^{\zeta}\\
     &\leq C\frac{ 4^{3\mathfrak{A}}}{(\mathfrak{C}\mathfrak{D})^2}\left(\frac{1+\mathfrak{A}}{\mathfrak{C}}\right)^{\frac{\gamma+\beta}{\gamma}+\frac{2\mathfrak{c}\mathfrak{b}_1}{1+\mathfrak{a}+\mathfrak{c}}+\frac{\mathfrak{b}_2}{1+\bar{s}_2+\mathfrak{b}_2}} \|\Lambda_{\mathcal{A}^{2},\Phi_{2}}-\Lambda_{\mathcal{A}^{1},\Phi_{1}}\|_{*}^{\zeta}.
     \end{split}
 \end{equation}
 Therefore, we have shown \eqref{eq: main stability estimate for external potential} in each case and we can conclude the proof.
\end{proof}

\subsection{Proof of Theorem \ref{main theorem}}
\label{subsec: proof main result}

Finally, we present here the proof of our main stability result.

\begin{proof}[Proof of Theorem \ref{main theorem}]
    Our main stability result is a consequence of Lemma \ref{lemma: stability estimate for vector potential} and Lemma \ref{lemma: stability estimate for potential}. To obtain the desired stability estimate \eqref{eq: final stability estimate} it is enough to notice the following two facts: First, the exponents
    \begin{equation}
        \begin{split}
            \mu&=\frac{2\gamma\mathfrak{c}}{(1+\mathfrak{a}+\mathfrak{c})(\gamma+\beta)}\\
            \zeta&= \frac{\mathfrak{b}_1\mathfrak{b}_2\delta\mu}{(1+\bar{s}_2+\mathfrak{b}_2)(\delta+\kappa)}
        \end{split}
    \end{equation}
    are increasing in $\mathfrak{c}$, decreasing in $\mathfrak{a}$ and increasing in $\mathfrak{b}_2$, decreasing in $\bar{s}_2$, respectively. Thus, by definition of the coefficients we may lower bound them as follows:
    \begin{equation}
    \begin{split}
         \mu &\geq \frac{2\gamma\mathfrak{C}}{(1+\mathfrak{A}+\mathfrak{C})(\gamma+\beta)}\geq \frac{\gamma\mathfrak{C}^2}{(1+\mathfrak{A}+\mathfrak{C})(\gamma+\beta)}\geq \frac{\delta \gamma\mathfrak{C}^4}{(1+\mathfrak{A}+\mathfrak{C})^2(\gamma+\beta)^2},\\
        \zeta &\geq \frac{\delta\mathfrak{C}^2}{(1+\mathfrak{A}+\mathfrak{C})(\delta+\kappa)}\geq \frac{\delta \gamma\mathfrak{C}^4}{(1+\mathfrak{A}+\mathfrak{C})^2(\delta+\kappa)^2}.
    \end{split}
    \end{equation}
    Secondly, we can take $\gamma=\delta$ and $\beta=\kappa$ in Lemma \ref{lemma: stability estimate for vector potential} and Lemma \ref{lemma: stability estimate for potential} as long as they satisfy $0<\gamma<1/11$ and $\beta\geq 3+10\gamma$. Finally, the constant in \eqref{eq: final stability estimate} is immediate from \eqref{eq: main stability estimate for vector potential} and \eqref{eq: main stability estimate for external potential}. Hence, we can conclude the proof.
\end{proof}

	\bigskip

	\noindent\textbf{Acknowledgments.}
	\begin{itemize}
		\item M.~Kumar acknowledges the support of PMRF (Prime minister research fellowship) from the government of India for his research. M. Kumar thanks Manmohan Vashisth and Venky Krishnan for helpful discussions at the beginning of the project. 
		\item P.~Zimmermann was supported by the Swiss National Science Foundation (SNSF), under grant number 214500.
	\end{itemize}

	\bibliography{refs} 
	
	\bibliographystyle{alpha}
	
\end{document}